\documentclass[11pt, a4paper]{amsart}

\usepackage[utf8]{inputenc}
\usepackage{amsmath, mathtools, amssymb, amsthm}
\usepackage[margin=0.8in]{geometry}

\usepackage{setspace, version}
\usepackage{natbib}
\usepackage{booktabs}
\usepackage{cancel}

\usepackage{graphicx}
\usepackage{url}
\usepackage[colorinlistoftodos]{todonotes}
\usepackage[colorlinks=true, allcolors=blue]{hyperref}
\usepackage[title]{appendix}
\usepackage[normalem]{ulem}
\usepackage{xcolor} 

\newtheorem{remark}{Remark}
\newtheorem{theorem}{Theorem}[section]
\newtheorem{corollary}{Corollary}[theorem]
\newtheorem{lemma}[theorem]{Lemma}

\newcommand{\up}{\breve{p}}
\newcommand{\un}{\breve{n}}
\newcommand{\upsi}{\breve{\psi}}

\newcommand{\utau}{\breve{\tau}}

\newcommand{\uu}{{\bf u}}

\newcommand{\bz}{{\bf z}}

\newcommand{\bm}{\boldsymbol}
\newcommand{\ep}{e_{p}}
\newcommand{\en}{e_{n}}

\newcommand{\uep}{\breve{e}_{p}}
\newcommand{\uen}{\breve{e}_{n}}
\newcommand{\uepsi}{\breve{e}_{\psi}}
\newcommand{\umu}{\breve{\mu}}
\newcommand{\unu}{\breve{\nu}}

\newcommand{\uuu}{\breve{\uu}}
\newcommand{\uphi}{\breve{\phi}}

\newcommand{\eu}{e_{\uu}}

\newcommand{\ueu}{\breve{e}_{\uu}}
\newcommand{\ueru}{\breve{e}_{R_N \uu}}
\newcommand{\uephi}{\breve{e}_{\phi}}

\newcommand{\uemu}{\breve{e}_{\mu}}
\newcommand{\uenu}{\breve{e}_{\nu}}

\newcommand{\projN}{\Pi_N}

\newcommand{\divergence}{\nabla \cdot}

\newcommand{\pt}{\partial_t}
\newcommand{\ptt}{\frac{\partial^2}{\partial t^2}}
\newcommand{\pttt}{\frac{\partial^3}{\partial t^3}}

\newcommand{\abs}[1]{\left\lvert#1\right\rvert}

\newcommand{\bigO}{\mathcal{O}}

\numberwithin{equation}{section}

\begin{document}

\title[structure preserving scheme for PNP-NS system]{A decoupled structure preserving scheme for the Poisson-Nernst-Planck Navier-Stokes equations and its error analysis}
\thanks{Cheng is supported in parts by NSFC 12301522 and the Fundamental Research Funds for the Central Universities, Shen is supported in parts by NSFC 11971407, and Wang is supported in parts by NSF 2101224. 
The main contents of this paper are part of Z. Yu's Ph.D thesis at Purdue University.}
\author[Z. Yu, J. Shen, C. Wang and Q. Cheng]{Ziyao Yu${}^\dag$, Jie Shen$^\star$, Changyou Wang$^\dag$, Qing Cheng$^\ddag$}
\thanks{${}^\dag$ Department of Mathematics, West Lafayette, IN 47907, USA (yu583@purdue.edu, wang2482@purdue.edu).\\
	$^\ddag$Key Laboratory of Intelligent Computing and Application (Tongji University), Ministry of Education,  and Department of Mathematics, Tongji University, Shanghai 200092, P. R. China (qingcheng@tongji.edu.cn). \\
	$^\star$ Eastern Institute of Technology, Ningbo, Zhejiang 315200, P. R. China (jshen@eitech.edu.cn).
	}

\begin{abstract}
We consider in this paper  numerical approximations for the  Poisson-Nernst-Planck-Navier-Stokes (PNP-NS) system. We propose  a decoupled semi-discrete and fully discrete scheme that enjoys the nice properties of positivity preserving, mass conserving, and unconditionally energy stability. Then, 
we establish the well-posedness and regularity for the initial and (periodic) boundary value problem of the PNP-NS system  under suitable assumptions on the initial data, and carry out a rigorous convergence analysis for the fully discretized scheme.    We also present some numerical results to validate the positivity preserving property and the accuracy for our decoupled numerical  scheme.

\end{abstract}


\maketitle

\section{Introduction}
In this paper, we consider a time-dependent system that describes the electrodiffusion of ions in an isothermal, incompressible, and viscous Newtonian fluid. Such a system is called the Poisson-Nernst-Planck-Navier-Stokes (PNP-NS) system \cite{gong2021partial,liu2021positivity,schmuck2011modeling}, which is widely applied in {\color{blue} fields such as microfluids which has numerous applications in lab-on-a-chip system; biology including vesicle motion,  membrane fluctuations, electroporation; and electrochemistry such as porous electrode charging, desalination dynamics, dendritic growth \cite{bazant2010induced}.  An introduction to some basic physical and mathematical descriptions can be found in \cite{rubinstein1990electro}. }  

{\color{blue}
We consider a solution of a monovalent symmetric strong salt. The Poisson-Nernst-Planck equations and incompressible Navier-Stokes Equations describe the system as
\begin{align*}
	& p_t + (\uu \cdot \nabla) p = D \divergence (\nabla p + \frac{e}{k_B T} p \nabla \Psi),\\
	& n_t + (\uu \cdot \nabla) n = D \divergence (\nabla n - \frac{e}{k_B T}  n \nabla \Psi), \\
	& -\epsilon \Delta \Psi = \rho_e, \\
	& \uu_t + (\uu \cdot \nabla) \uu - \nu_{vis} \Delta \uu + \nabla P = -\nabla \Psi \rho_e, \\
	& \divergence \uu = 0, 
\end{align*}
where $\uu$ and $P$ denote the velocity field of the fluid and the pressure function, respectively.  The variables $p$ and $n$ represent the concentration functions of positive and negative ions in the fluid, respectively,  and $\Psi$ is the electric potential.  Here $\rho_e = e(p - n)$ represents the free charge density for a monovalent symmetric salt (here ionic valence $z = \pm 1$), and $e$ is elementary charge.  $k_B$ is Boltzemann's constant, $T$ is temperature and $D$ is the diffusion coefficient of ions.  Moreover,  $\epsilon$, $\nu_{vis}$ are the dielectric permittivity and viscosity of the fluid.  Normalizing the electric potential by introducing $\psi$:
$ \psi = \frac{e}{k_B T} \Psi. $
The PNP-NS system is therefore given by
\begin{align}
	& p_t + (\uu \cdot \nabla) p = D \divergence (\nabla p + p \nabla \psi), \label{gov_equ_1}\\
	& n_t + (\uu \cdot \nabla) n = D \divergence (\nabla n - n \nabla \psi), \label{gov_equ_2}\\
	& -\varepsilon \Delta \psi = p - n, \label{gov_equ_3}\\
	& \uu_t + (\uu \cdot \nabla) \uu -\nu_{vis} \Delta \uu + \nabla P = -\kappa  \nabla \psi (p - n), \label{gov_equ_4}\\
	& \divergence \uu = 0,\label{gov_equ_5}
\end{align}  
with 
$ \varepsilon = \frac{\epsilon k_B T}{e^2} $, and $\kappa = k_B T$. It is worth to note that $\varepsilon = 2 c_{bulk} \lambda_D^2$,  where $\lambda_D$ is the Debye screening length \cite{bazant2010induced} defined by $\lambda_D = \sqrt{\frac{\epsilon k_B T}{2 c_{\text{bulk}} e^2}}$ and $c_{\text{bulk}}$ is a reference bulk concentration of ions.  
}
The system \eqref{gov_equ_1}-\eqref{gov_equ_5} is subjected to a set of initial and boundary conditions, which will be specified later.

There has been considerable interest in the mathematical analysis of the PNP-NS system. For example, Schmuck \cite{schmuck2009analysis} established the global existence of weak solutions in three dimensions under the blocking boundary condition for $(p, n)$ and the zero Neumann boundary condition for $\psi$; Gong-Wang-Zhang \cite{gong2021partial} established the existence and partial regularity of suitable weak solutions in three dimensions under the zero Neumann boundary condition for $p$, $n$, and $\psi$; Constantin-Ignatova \cite{constantin2019nernst} proved the global existence and stability result in two dimensions, with the blocking and selective boundary conditions for $(p, n)$ and the Dirichlet boundary condition for $\psi$. We emphasize that the solutions of the PNP-NS system are positive ($n, p > 0$), mass-conserving, and energy-dissipative.

In recent years, a large effort has been devoted to constructing positivity-preserving schemes for various problems in different areas \cite{MR4186541,MR4107225,MR3880256,MR4253790,liu2018positivity,van2019positivity,shen2016maximum,zhornitskaya1999positivity,chen2019positivity}. There are also quite a few numerical investigations on the PNP-NS system \eqref{gov_equ_1}-\eqref{gov_equ_5}. It was shown in \cite{flavell2014conservative} that it is important for numerical schemes to maintain mass conservation. Prohl-Schmuck proposed in \cite{prohl2010convergent} a coupled fully implicit first-order scheme with a finite-element method in space for the PNP-NS system and studied its convergence. Additionally, a first-order time-stepping method was proposed in \cite{liu2017efficient} with spectral method discretization in space. Several structure-preserving numerical methods have been proposed for the PNP equations, for example, \cite{cheng2022new2, cheng2022new, flavell2014conservative, hu2020fully, huang2021bound, liu2021positivity, prohl2009convergent, shen2021unconditionally, liu2021efficient, ding2019positivity}. {\color{blue} There are also some recent studies reformulate the PNP system into Maxwell-Ampere Nernst-Planck(MANP) system \cite{qiao2023structure}.} However, there appears to be no scheme available in the literature for the PNP-NS system \eqref{gov_equ_1}-\eqref{gov_equ_5} that enjoys the properties of unique solvability, mass- and {positivity-preserving}, and energy stability.

In this paper, we propose a decoupled, mass- and positivity-preserving, and unconditionally energy-stable scheme for the PNP-NS system and carry out a rigorous error analysis. The main contributions of this paper include:
\begin{itemize}
	\item We propose a totally decoupled, mass- and positivity-preserving, and unconditionally energy-stable scheme for the PNP-NS system by combining the following techniques:
	\begin{itemize}
		\item Rewriting the PNP system as a Wasserstein gradient flow and using the technique introduced in \cite{shen2021unconditionally} to preserve positivity and energy stability for the PNP system;
		\item Using a projection-type method \cite{guermond2006overview,guermond1998approximation,guermond1998stability} to decouple the velocity and pressure;
		\item Introducing an extra $\mathcal{O}(\Delta t)$ term as in \cite{shen2015decoupled}, which allows us to treat the convective term in the PNP equations explicitly while maintaining stability.
	\end{itemize}
	\item We derive the existence and regularity results of the PNP-NS system \eqref{gov_equ_1}-\eqref{gov_equ_5} with periodic boundary conditions under suitable assumptions on the initial data.
	
	\item To carry out an error analysis, it is necessary to have $L^{\infty}$ bounds for $n$ and $p$, which are not available through energy stability. We use an approach similar to \cite{liu2021positivity} to derive these bounds by introducing a high-order asymptotic expansion for both the PNP equations and the Navier-Stokes equations.
\end{itemize}

This paper is organized as follows: In Section 2, we construct a semi-discrete (in time) scheme, followed by a fully discrete scheme with a generic spatial discretization, and prove that it preserves mass and positivity, and is unconditionally energy stable. In Section 3, we establish the well-posedness and regularity of the PNP-NS system under periodic boundary conditions. An error analysis of the fully discretized scheme is carried out in Section 4. Some numerical results are provided in Section 5.

\section{A Decoupled Numerical Scheme and Its Properties}
Let $\Omega$ be a bounded domain in $\mathbb{R}^2$. We consider the time discretization of the PNP-NS system \eqref{gov_equ_1}-\eqref{gov_equ_5} subjected to {\color{blue} boundary condition} either  
\begin{itemize}
	\item {\color{blue} $\mathcal{B}^{block}$:} the non-slip boundary condition for $\uu$, the homogeneous Neumann boundary condition for $(\phi, \ln{p} + \psi, \ln{n}-\psi)$, i.e., all the fluxes vanish on the boundary of $\Omega$:
	\begin{equation}\label{BC}
	\uu|_{\partial \Omega}=0,\;\; \nabla\phi \cdot \vec{\nu} \big|_{\partial \Omega}=	(\nabla p + p \nabla \psi) \cdot \vec{\nu} \big|_{\partial \Omega}=
		(\nabla n - n \nabla \psi) \cdot \vec{\nu} \big| _{\partial \Omega} = 0,
	\end{equation}
	\item {\color{blue} $\mathcal{B}^{periodic}$:} the periodic boundary conditions for all variables, 
\end{itemize}
along with the initial condition:
\begin{align}\label{ic}
	(\uu, p, n)(x,y,0)=(\uu^{\text{in}}, p^{\text{in}}, n^{\text{in}}) (x,y), \quad \text{for} \ (x,y)\in \Omega.
\end{align}
For either \eqref{BC} or the periodic boundary conditions, one observes that the mass of ions is conserved, i.e.,
\begin{equation*}
	\int_{\Omega} p(x, t) \, dx = \int_{\Omega} p(x, 0) \, dx, \quad \int_{\Omega} n(x, t) \, dx = \int_{\Omega} n(x, 0) \, dx, \quad \forall t \in [0, T].
\end{equation*}
Another essential property of the PNP-NS system \eqref{gov_equ_1}-\eqref{gov_equ_5} is the following energy dissipation law:
\begin{equation}\label{law}
	\frac{d}{dt}E(p, n, \uu) = -\int_{\Omega} \left( {\color{blue} \nu_{vis}}  \rvert \nabla \uu \rvert^2 + {\color{blue} \kappa D} p \rvert \nabla \mu \rvert^2 +  {\color{blue} \kappa D } n \rvert \nabla \nu \rvert^2 \right) dx,
\end{equation}
where $\mu = \ln{p} + \psi$ and $\nu = \ln{n} - \psi$ are chemical potentials of the PNP-NS system, and $E$ is the total energy given by 
\begin{equation*}
	E(p, n, \uu) = \int_{\Omega} {\color{blue} \kappa} \left( p (\ln{p} - 1) + n (\ln{n} - 1) + \frac{\varepsilon}{2} \rvert \nabla \psi \rvert^2 \right) + \frac{1}{2} \rvert \uu \rvert^2  dx.
\end{equation*}

\subsection{Time Discretization}
We first consider the time discretization.  {\color{blue} For simplicity,  we set various constants $D = \varepsilon = \kappa = \nu_{vis} = 1$ for the rest analysis. }In order to construct an efficient time discretization scheme, we first rewrite the right-hand side of equation \eqref{gov_equ_4} as
\begin{equation*}
	-\nabla \psi(p-n) = -(p\nabla \mu + n \nabla \nu) + \nabla(p + n),
\end{equation*}
and introduce a modified pressure $\phi = P - p - n$. Then, the PNP-NS system \eqref{gov_equ_1}-\eqref{gov_equ_5} can be reformulated as 
\begin{align}
	& p_t + (\uu \cdot \nabla) p = \divergence (p\nabla \mu), \label{gov_equ_mod1}\\
	& n_t + (\uu \cdot \nabla) n = \divergence (n\nabla \nu), \label{gov_equ_mod2}\\
	& -\Delta \psi = p - n, \label{gov_equ_mod3}\\
	& \uu_t + (\uu \cdot \nabla) \uu - \Delta \uu + \nabla \phi = -(p\nabla \mu + n \nabla \nu), \label{gov_equ_mod4}\\
	& \divergence \uu = 0. \label{gov_equ_mod5}
\end{align}
{\color{blue}
Depending on boundary condition choice $\mathcal{B}$, we define the function space $X(\mathcal{B}),  U(\mathcal{B}),  W(\mathcal{B})$:
\begin{itemize}
	\item $X(\mathcal{B}^{block}) = X(\mathcal{B}^{periodic}) = H^{1}(\Omega)$,
	\item $U(\mathcal{B}^{block}) = U(\mathcal{B}^{periodic}) =  \{q \in L^2(\Omega) : \int_\Omega q \, dx = 0\} $,
	\item $
				W(\mathcal{B}) = \big\{ \begin{array}{l}
															H^{1}_{0}(\Omega	),  \text{ if } \mathcal{B} = \mathcal{B}^{block}, \\
															H^{1}(\Omega),  \text{ if } \mathcal{B} = \mathcal{B}^{periodic}.
														\end{array}
			$
\end{itemize}
}

Following some of the ideas in \cite{shen2021unconditionally, liu2021positivity, shen2015decoupled}, we construct a first-order time discretization scheme as follows: {\color{blue} under boundary condition $\mathcal{B}$ being either $\mathcal{B}^{block}$ or $\mathcal{B}^{periodic}$,} for any given $(p^{m}, n^{m}, \uu^{m}, \phi^{m})$ with $\int_{\Omega} (p^{m} - n^{m}) \, dx = 0$, $(p^{m}, n^{m}) > 0$ and $\divergence \uu^{m} = 0$ in $\Omega$, we compute {\color{blue} $(p^{m+1}, n^{m+1}, \uu^{m+1}, \phi^{m+1})$ } in three steps:
\begin{itemize}
	\item Step 1: Solve ${\color{blue} (p^{m+1}, n^{m+1})   \in X(\mathcal{B}) \times X(\mathcal{B})}$  from
	\begin{align}
		& \frac{p^{m+1} - p^{m}}{\Delta t} + \divergence(p^{m} \uu^{m}) = \divergence(p^{m}(1 + 2\Delta t p^{m}) \nabla \mu^{m+1}), \label{time_dist_mod1} \\ 
		& \frac{n^{m+1} - n^{m}}{\Delta t} + \divergence(n^{m} \uu^{m}) = \divergence(n^{m}(1 + 2\Delta t n^{m}) \nabla \nu^{m+1}), \label{time_dist_mod2}\\
&	-\Delta \psi^{m+1} = p^{m+1} - n^{m+1}. \label{time_dist_mod3}
\end{align}
	where 
	\begin{equation*}
		\mu^{m+1} = \ln{p^{m+1}} + \psi^{m+1} \quad {\text{and}} \quad \nu^{m+1} = \ln{n^{m+1}} - \psi^{m+1}.
	\end{equation*}
	\item Step 2: Solve ${\color{blue} \tilde{\uu}^{m+1}  \in W(\mathcal{B})^2 }$ from 
	{\color{blue}
			\begin{equation}
					\frac{\tilde{\uu}^{m+1} - \uu^{m}}{\Delta t} + (\uu^{m} \cdot \nabla) \tilde{\uu}^{m+1} - \Delta \tilde{\uu}^{m+1} + \nabla \phi^{m} = -\left(p^{m} \nabla \mu^{m+1} + n^{m} \nabla \nu^{m+1}\right),\label{time_dist_mod4} 			
			\end{equation}
	}
	\item  Step 3: Solve ${\color{blue} (\uu^{m+1}, \phi^{m+1}) \in W(\mathcal{B})^2 \times U(\mathcal{B})  }$ from
	{\color{blue}
	\begin{align}
		& \frac{\uu^{m+1} - \tilde{\uu}^{m+1}}{\Delta t} + \nabla(\phi^{m+1} - \phi^{m}) = 0, \label{time_dist_mod5} \\
		& \divergence \uu^{m+1} = 0.		\label{time_dist_mod6}
	\end{align}}
\end{itemize}
The first step involves solving a coupled nonlinear system for $(p^{m+1}, n^{m+1},\psi^{m+1})$ which can be formulated as a minimization problem for a convex functional, see \cite{shen2021unconditionally} and also Theorem \ref{thm:four_property}. The second step solves a Poisson-type equation for $\tilde{\uu}^{m+1}$. 
And the third step is equivalent to solving
	\begin{align}
\Delta(\phi^{m+1} - \phi^{m}) =  \frac{1}{\Delta t} \divergence \tilde{\uu}^{m+1}, \label{time_dist_mod5b} 
	\end{align}
	along with either $(\phi^{m+1} - \phi^{m}) \cdot \vec{\nu} \big|_{\partial \Omega} = 0$ or the periodic boundary condition, and
	\begin{align}
{\color{blue} \uu^{m+1}} = \tilde{\uu}^{m+1} - \Delta t \nabla(\phi^{m+1} - \phi^{m}). \label{time_dist_modb6}
\end{align}
Thus, the scheme \eqref{time_dist_mod1}-\eqref{time_dist_mod6} can be efficiently implemented.


\begin{remark}
	{\color{blue} In \eqref{time_dist_mod1} \eqref{time_dist_mod2}, we discretized the mobility term as $p^{m}(1 + 2\Delta t p^{m})$, $n^{m}(1 + 2\Delta t n^{m})$,  moving the $\mathcal{O}(\Delta t)$ terms to the left,  therefore the} first step can be rewritten as:  
	\begin{align*}
		& \frac{p^{m+1} - p^{m}}{\Delta t} + \divergence(p^{m} \uu^{m}_{*,p}) = \divergence(p^{m} \nabla \mu^{m+1}),  \\ 
		& \frac{n^{m+1} - n^{m}}{\Delta t} + \divergence(n^{m} \uu^{m}_{*,n}) = \divergence(n^{m} \nabla \nu^{m+1}), 
	\end{align*}
	where 
	\begin{align*}
		\uu^{m}_{*,p} = \uu^{m} - 2 \Delta t p^{m} \nabla \mu^{m+1}, \\
		\uu^{m}_{*,n} = \uu^{m} - 2 \Delta t n^{m} \nabla \nu^{m+1}.
	\end{align*}
	This is similar to the decoupling technique introduced by \cite{shen2015decoupled}, {\color{blue} where specific additional $\mathcal{O}(\Delta t)$ terms are introduced such that the decoupled discrete numerical scheme is unconditionally energy stable,  see Theorem \ref{thm:four_property} below.}
\end{remark}

\subsection{Fully Discretized Scheme}
In this subsection, we shall consider a generic spatial discretization for \eqref{time_dist_mod1}-\eqref{time_dist_mod6}. Let $\Sigma_N$ be a set of mesh points or collocation points in $\bar{\Omega}$. Note that $\Sigma_N$ should not include the points on the part of the boundary where a Dirichlet (or essential) boundary condition is prescribed, while it should include the points on the part of the boundary where a Neumann or mixed (or non-essential) boundary condition is prescribed.
    
We consider a Galerkin-type discretization with finite elements, spectral methods, or finite differences with summation-by-parts in a subspace $X_N \subset X$, and define a discrete inner product, i.e., numerical integration, on $\Sigma_N = \{\bm{z}\}$ in $\bar{\Omega}$:
{\color{blue}
\begin{equation}\label{numint}
  \langle u_N, v_N \rangle_{N, \omega} = \sum_{\bm{z} \in \Sigma_N} \omega_{\bm{z}} u_N(\bm{z}) v_N(\bm{z}), 
\end{equation}}
where we require that the weights $\omega_{\bm{z}} > 0$. We also denote the induced norm by {\color{blue}$\|u_N\| = \langle u_N, u_N \rangle_{N, \omega}^{\frac{1}{2}}$}. For finite element methods, the sum should be understood as $\sum_{K \subset \mathcal{T}} \sum_{\bm{z} \in Z(K)}$, where $\mathcal{T}$ is a given triangulation.  
We assume that there is a unique function $\psi_{\bm{z}}(\bm{x})$ satisfying $\psi_{\bm{z}}(\bm{z}') = \delta_{\bm{z}\bm{z}'}$ for $\bm{z}, \bm{z}' \in \Sigma_N$.  {\color{blue}
Under boundary condition $\mathcal{B}$, let $X_N$, ${W}_N$, and ${U}_N$ be suitable {\color{blue} discretization} subspaces of $X(\mathcal{B})$, $W(\mathcal{B})$, and $U(\mathcal{B})$, respectively.}

{\color{blue} To fix the idea and without loss of generality,  throughout the rest of the paper,  reader can think we are discussing under spectral method discretization framework,  and $X_N, W_N,  U_N$ are subspaces of $P_N$, where
$$
	 P_N := \big\{ \begin{array}{l}
	 						span\{e^{\mathrm{i} j x} e^{\mathrm{i} k y}: -\frac{N}{2} \leq k \leq \frac{N}{2} - 1 \},  \text{ if } \mathcal{B} = \mathcal{B}^{periodic}, \\
	 						span\{1,  x,  y,  xy, ...,    x^{N}y^{N} \}, \text{ if } \mathcal{B} = \mathcal{B}^{block}.
	 					\end{array}
$$
Under spectral method framework, the quadrature error is very small when $N$ is large enough,  and avoidable in numerical implementation by choosing quadrature points numbers $N_Q$ bigger than basis numbers $N$.  For simplicity, throughout the rest of the paper,  we ignore the quadrature error, and do not distinguish the continuous inner product $\langle u_N,  v_N \rangle$ and discrete inner product $\langle u_N,  v_N \rangle_{N, \omega}$.
}

Then, a fully discretized version of \eqref{time_dist_mod1}-\eqref{time_dist_mod6} for the PNP-NS system \eqref{gov_equ_mod1}-\eqref{gov_equ_mod5} is as follows:
    
Given $(p_N^m, n_N^m, \uu_N^{m}, \phi_N^m) \in X_N \times X_N \times {{W}_N^2} \times {{U}_N}$, with $p_N^m, n_N^m > 0$ in $\Omega$, $\displaystyle \langle p_N^m - n^m_N, 1 \rangle = 0$, and $\nabla \cdot \mathbf{u}_N^m = 0$ in $\Omega$, we proceed as follows:
\begin{itemize}
	\item Step 1: Solve $(p^{m+1}_N, n^{m+1}_N) \in X_N \times X_N$ from
		\begin{align}
			& \langle \frac{p^{m+1}_N - p^{m}_N}{\Delta t}, v_N \rangle - \langle p^{m}_N \uu^{m}_N, \nabla v_N \rangle + \langle p^{m}_N (1 + 2\Delta t p^{m}_N) \nabla \mu^{m+1}_N, \nabla v_N \rangle = 0, \;\forall v_N \in X_N, \vspace{0.5em} \label{scheme:p1} \\
			& \langle \frac{n^{m+1}_N - n^{m}_N}{\Delta t}, v_N \rangle - \langle n^{m}_N \uu^{m}_N, \nabla v_N \rangle + \langle n^{m}_N (1 + 2\Delta t n^{m}_N) \nabla \nu^{m+1}_N, \nabla v_N \rangle = 0, \;\forall v_N \in X_N, \vspace{0.5em} \label{scheme:n1} \\ 
			& \langle \nabla \psi^{m+1}_N, \nabla v_N \rangle = \langle p^{m+1}_N - n^{m+1}_N, v_N \rangle, \;\forall v_N \in X_N, \vspace{0.5em} \label{scheme:psi}
		\end{align} 
		where 
		\begin{equation} 
			\mu^{m+1}_N = \ln{p^{m+1}_N} + \psi^{m+1}_N, \quad
			\nu^{m+1}_N = \ln{n^{m+1}_N} - \psi^{m+1}_N.
		 \end{equation}
		
	\item Step 2: Solve $\tilde{\uu}^{m+1}_N \in {{W}_N^2}$ from 
			\begin{align}
				& \langle \frac{\tilde{\uu}^{m+1}_N - \uu^{m}_N}{\Delta t}, {{w_N}} \rangle + \langle (\uu^{m}_N \cdot \nabla) \tilde{\uu}^{m+1}_N, {{w_N}} \rangle + \langle \nabla \tilde{\uu}^{m+1}_N, \nabla {{w_N}} \rangle + \langle \nabla \phi^{m}_N, w_N \rangle, \nonumber  \\
				& \qquad + \langle p^{m}_N \nabla \mu^{m+1}_N + n^{m}_N \nabla \nu^{m+1}_N, {{w_N}} \rangle = 0, \;\forall w_N \in {{W}_N^2}, \vspace{0.5em} \label{scheme:u1}  
				\end{align}
	\item Step 3: Solve $(\uu^{m+1}_N, \phi^{m+1}_N) \in {{W}_N^2} \times {{U}_N}$ from
		\begin{align}
			& \langle \frac{\uu^{m+1}_N - \tilde{\uu}^{m+1}_N}{\Delta t}, v_N \rangle + \langle \nabla (\phi^{m+1}_N - \phi^{m}_N), v_N \rangle = 0, \; v_N \in X_N^2, \label{scheme:u2} \\
			& {{\langle \uu^{m+1}_N, \nabla q_N \rangle}} = 0, \; q_N \in {{U}_N}. \label{scheme:u3}
		\end{align}
\end{itemize}

We shall show below that the nonlinear system \eqref{scheme:p1}-\eqref{scheme:psi} in Step 1 can be interpreted as a minimization of a convex functional. In Step 2, we only need to solve a Poisson-type equation for $\tilde{\uu}^{m+1}_N$, and Step 3 is a discrete Darcy system which can be reduced to a discrete Poisson equation for $\phi^{m+1}_N - \phi^{m}_N$. Hence, the above scheme can be efficiently solved.

\subsection{Properties of the Numerical Scheme}
\label{ch-PNPNS-sec-property}

We show below that our decoupled numerical scheme \eqref{scheme:p1}-\eqref{scheme:u3} enjoys four properties: mass conservation, unique solvability, positivity-preserving, and unconditional energy stability.

Before proceeding to the proof, for any discrete positive function $\mathcal{M}(\bz) > 0$ for all $\bz \in \Sigma_N$, we introduce the operator $\mathcal{L}_{\mathcal{M}}: X_N \rightarrow X_N$ defined by
\begin{equation}
    \langle \mathcal{L}_{\mathcal{M}} f_N, v_N \rangle = \langle \mathcal{M} \nabla f_N, \nabla v_N \rangle, \quad \forall f_N, v_N \in X_N. \label{aux_operator}
\end{equation}
The operator $\mathcal{L}_{\mathcal{M}}$ is invertible on the space $\dot{X}_N = \big\{ f \in X_N \mid \langle f, 1 \rangle = 0 \big\}$, so we can define the inverse operator $\mathcal{L}^{-1}_{\mathcal{M}}: X_N \rightarrow \dot{X}_N$ and the induced norm
\begin{equation*}
    \| f_N \|_{\mathcal{L}^{-1}_{\mathcal{M}}} = \sqrt{\langle f_N, \mathcal{L}_{\mathcal{M}}^{-1} f_N \rangle}, \quad \forall f_N \in X_N.
\end{equation*}
If $\mathcal{M}(\bz) \equiv 1$ for all $\bz \in \Sigma_N$, then we have 
\begin{equation*}
    \mathcal{L}_{\mathcal{M}}(f_N) = -\Delta f_N \quad \text{and} \quad
    \| f_N \|_{-1,\Omega} = \sqrt{\langle f_N, (-\Delta)^{-1} f_N \rangle}, \quad \forall f_N \in X_N. 
\end{equation*}

\begin{lemma}\label{aux:lem:2}
    Suppose $f_N \in X_N$ and $\mathcal{M} \geq M_0$, then we have the estimate:
    \begin{equation*}
        \| \mathcal{L}^{-1}_{\mathcal{M}} f_N \|_{\infty} \leq \frac{C {\color{blue} N} }{M_0} \| f_N \|,
    \end{equation*}
    where $C$ depends only on $\Omega$. 
    \begin{proof}
       {Denote $u_N = \mathcal{L}^{-1}_{\mathcal{M}} f_N \in {\color{blue} \dot{X}_N}$. From \eqref{aux_operator} and using the Poincar\'e-Wirtinger inequality, we have}
        \begin{equation*}
            M_0 \| \nabla u_N \|^2 \leq \langle \mathcal{M} \nabla u_N, \nabla u_N \rangle = \langle f_N, u_N \rangle \leq \| f_N \| \| u_N \| \leq C \| f_N \| \| \nabla u_N \|,
        \end{equation*}
        and applying the Nikolskii's inequality, we have {\color{blue}
        \begin{equation*}
            \| u_N \|_{\infty} \leq C(\Omega) N \| u \|_2  \leq  C(\Omega) N \| \nabla u_N \| \leq \frac{C N }{M_0} \| f_N \|,
        \end{equation*} }
        where $C$ depends only on $\Omega$. 
    \end{proof}
\end{lemma}

\begin{theorem}
    \label{thm:four_property}
    Given $(p^{m}_N, n^{m}_N, \uu^{m}_N, \phi^{m}_N) \in X_N \times X_N \times {W}_N^2 \times {U}_N$,
    with $p_N^m(\bz), n_N^m(\bz) > 0$ for all $\bz \in \Sigma_N$, $\displaystyle \langle p_N^m - n^m_N , 1 \rangle = 0$, and $\nabla \cdot \mathbf{u}_N^m = 0$ in $\Omega$, then the scheme \eqref{scheme:p1}-\eqref{scheme:u3} enjoys the following properties:
    \begin{enumerate}
    	\item \textbf{Mass Conservation:}
    		\begin{equation*}
    			\langle p^{m+1}_N, 1 \rangle = \langle p^{m}_N, 1 \rangle, \quad 
    			\langle n^{m+1}_N, 1 \rangle = \langle n^{m}_N, 1 \rangle.
    		\end{equation*} 
    	
    	\item \textbf{Unique Solvability:} The scheme \eqref{scheme:p1}-\eqref{scheme:u2} has a unique solution 
    	\begin{equation*}
    		(p^{m+1}_N, n^{m+1}_N, \uu^{m+1}_N, \phi^{m+1}_N) \in X_N \times X_N \times {W}_N^2 \times {U}_N.
    	\end{equation*}
    	
    	\item \textbf{Positivity Preserving:} The unique solution $(p^{m+1}_N, n^{m+1}_N, \uu^{m+1}_N, \phi^{m+1}_{N})$
    	satisfies 
    	\begin{equation*}
    		p^{m+1}_N(\bz), \, n^{m+1}_N(\bz) > 0, \quad \forall \bz \in \Sigma_N.
    	\end{equation*}
    	
    	\item \textbf{Unconditional Energy Stability:} 
    		\begin{equation*}
    			\begin{split}
    				& \quad \frac{1}{\Delta t}\bigg( 
    				\big(E(p^{m+1}_N) + E(n^{m+1}_N) + \frac{1}{2}\| \nabla \psi^{m+1}_N \|^2 
    				+ \frac{1}{2}\| \uu^{m+1}_N \|^2 + \frac{\Delta t^2}{2} \| \nabla \phi^{m+1}_N \|^2 \big) \\
    				& \quad-
    				\big(E(p^{m}_N) + E(n^{m}_N) + \frac{1}{2}\| \nabla \psi^{m}_N \|^2
    					+ \frac{1}{2}\| \uu^{m}_N \|^2 + \frac{\Delta t^2}{2} \| \nabla \phi^{m}_N \|^2 \big)
    				\bigg) \\
    				& \quad + \| \nabla \tilde{\uu}^{m+1}_N \|^2 
    				  + \langle p^{m}_N | \nabla \mu^{m+1}_N |^2, 1 \rangle
    				  + \langle n^{m}_N | \nabla \nu^{m+1}_N |^2, 1 \rangle \\
    				& \quad + \frac{1}{2\Delta t} \| \nabla (\psi^{m+1}_N - \psi^{m}_N) \|^2 
    					+ \frac{1}{2\Delta t} \| \uu^{m+1}_N - \tilde{\uu}^{m+1}_N \|^2 \\
    				& \quad + \frac{1}{2\Delta t} \| \tilde{\uu}^{m+1}_N - \uu^{m}_N \|^2
    					+ \frac{\Delta t}{2} \| \nabla(\phi^{m+1}_N - \phi^{m}_N) \|^2 \\
    				& \leq -\| \nabla \tilde{\uu}^{m+1}_N \|^2 
    				 -\langle p^{m}_N | \nabla \mu^{m+1}_N |^2, 1 \rangle
    				 -\langle n^{m}_N | \nabla \nu^{m+1}_N |^2, 1 \rangle.
    			\end{split}
    		\end{equation*}
    		where the energy is defined by 
    		\begin{equation*}
    			E(v_N) = \langle v_N(\ln{v_N} - 1), 1 \rangle,
    		\end{equation*}
    		for any function $v_N \in X_N$.
    \end{enumerate}
\end{theorem}

\begin{proof}

\smallskip
	\begin{enumerate}
		\item \textbf{Mass Conservation:} This follows directly by choosing the test function $v_N = 1$ in the equations \eqref{scheme:p1} and \eqref{scheme:n1}. 
		
		\item \textbf{Unique Solvability and Positivity Preserving:}
				The numerical solution $\{p^{m+1}_N, n^{m+1}_N\}$ of \eqref{scheme:p1}-\eqref{scheme:psi} 
				is obtained through the minimization of the discrete energy functional:
				\begin{equation*}
					\begin{split}
						J(p^{*}_N, n^{*}_N) 
						&= \frac{1}{2\Delta t}\Big(\| p^{*}_N - p^{m}_N \|_{\mathcal{L}^{-1}_{p^{m}_N(1 + 2\Delta t p^{m}_N)}}^2 
							+ \| n^{*}_N - n^{m}_N \|_{\mathcal{L}^{-1}_{n^{m}_N(1 + 2\Delta t n^{m}_N)}}^2 \Big) \\
						&\quad + \langle \nabla \cdot (p^m_N \mathbf{u}_N^m) p^{*}_N, 1 \rangle
							+ \langle \nabla \cdot (n^m_N \mathbf{u}_N^m) n^{*}_N, 1 \rangle  \\
						&\quad + \langle p^{*}_N (\ln{p^{*}_N} - 1), 1 \rangle
							   + \langle n^{*}_N (\ln{n^{*}_N} - 1), 1 \rangle
							   + \frac{1}{2} \| p^{*}_N - n^{*}_N \|_{-1,\Omega}^2,
					\end{split}
				\end{equation*}
				over the admissible space
				\begin{equation*}
				\hat Y_N = \Big\{(p_N, n_N) \in X_N^2 \mid 0 < p_N(\bz), \, n_N(\bz) < M_N,  \forall \bz \in \Sigma_N,  \langle p_N, 1 \rangle = \langle n_N, 1 \rangle = \beta_0 |\Omega| \Big\},
				\end{equation*} 
				where 
				$$\beta_0 = \frac{1}{|\Omega|} \langle p_N^m, 1 \rangle = \frac{1}{|\Omega|} \langle n_N^m, 1 \rangle$$
				is the average of $p_N^m$ (and $n_N^m$), and 
				$$ M_N = \frac{\beta_0 |\Omega| N^3}{8\pi^3}.$$
    
			Below we show uniqueness, solvability, and positivity for scheme \eqref{scheme:p1}-\eqref{scheme:u2} by suitable modifications of \cite{liu2021positivity} and \cite{shen2021unconditionally}. 
    			 
			Firstly, we observe that every term in the functional $J(p_N, n_N)$ is strictly convex or linear with respect to the variables $(p_N, n_N)$ over the admissible space $\hat Y_N$. To show the existence of a unique minimizer of $J(p_N, n_N)$ over $\hat Y_N$,
			we proceed as follows. For a sufficiently small $0 < \delta < \beta_0$, whose value is to be determined later, we define 
			$$Y_{N,\delta} = \Big\{(p_N, n_N) \in \hat Y_N \mid \delta \le p_N(\bz), \, n_N(\bz) \le M_N - \delta, \quad \forall \bz \in \Sigma_N \Big\}.$$
			Since $Y_{N,\delta}$ is a compact subset of $\hat Y_N$, there exists a minimizer $(p^*_N, n^*_N) \in Y_{N,\delta}$ 
			of $J(p_N, n_N)$ over $Y_{N,\delta}$. Next, we need to show that $(p^*_N, n^*_N)$ lies in the interior of $Y_{N,\delta}$, 
			provided $\delta > 0$ is chosen to be sufficiently small.
			
			 Suppose the contrary that for an arbitrarily small $\delta$, the minimizer of $J(p_N, n_N)$ occurs at the boundary of $Y_{N,\delta}$, i.e., $(p^*_N, n^*_N) \in \partial Y_{N,\delta}$ for all $\delta > 0$. For simplicity, we only consider the case that there exists a point $(x_0, y_0) \in \Sigma_N$ such that $p^*_N(x_0, y_0) = \delta$ (the other case can be handled similarly). 
			Notice that there exists another point $(x_1, y_1) \neq (x_0, y_0)$ and $(x_1, y_1) \in \Sigma_N$ such that 
			$p^*_N(x_1, y_1) = \max_{{\bf x} \in \Sigma_N} p^*_N({\bf x}) \geq \beta_0$.  
			Now we can choose the test function $\psi_N$ as $\psi_N = \phi^N_{(x_0, y_0)}(x, y) - \phi^N_{(x_1, y_1)}(x, y)$, where $\phi^N_{(x_0, y_0)}(x, y)$ and $\phi^N_{(x_1, y_1)}(x, y)$ are Lagrange polynomials satisfying the following property: for all $(x, y) \in \Sigma_N$
            \begin{equation*}
                \begin{array}{cc}
                     & \phi_{N}^{(x_0, y_0)}(x, y) = \delta_{(x_0, y_0)}(x, y),\\
                     & \phi_{N}^{(x_1, y_1)}(x, y) = \delta_{(x_1, y_1)}(x, y),
                \end{array}
            \end{equation*}
            where $\delta_{(x_0,y_0)}(x, y)$ and $\delta_{(x_1, y_1)}(x, y)$ are the Kronecker delta functions. 
			Since $(p^*_N, n^*_N)$ is the minimizer and $(p^{*}_N + s \psi_N, n^{*}_N) \in Y_{N,\delta}$ for $s \geq 0$ small, we have
				$$ \frac{d}{ds}J(p^*_N + s\psi_N, n^*_N)\bigg|_{s=0} = 0.$$
				Direct computations imply
				\begin{align}\label{test:2}
				\frac{d}{ds}J(p^*_N + s\psi_N, n^*_N)\bigg|_{s=0} &= \frac{1}{\Delta t} \langle \mathcal{L}^{-1}_{p_N^m(1 + 2\Delta t p_N^m)} (p^*_N - p^m_N), \psi_N \rangle + \langle \ln p^*_N, \psi_N \rangle \nonumber \\
				&\quad + \int_\Omega \nabla \cdot (p^m_N \mathbf{u}_N^m) \psi_N \, dx + \langle (-\Delta)^{-1}(p^*_N - n^*_N), \psi_N \rangle.
				\end{align}
				Plugging $\psi_N = \phi_N^{(x_0, y_0)}(x, y) - \phi_N^{(x_1, y_1)}(x, y)$ into \eqref{test:2}, we obtain
				{
				\begin{align}\label{directional-d}
					-\ln \left(\frac{p^*_N(x_0, y_0)}{p^*_N(x_1, y_1)}\right) &= \nabla \cdot (p^m_N \mathbf{u}_N^m)(x_0, y_0) - \nabla \cdot (p^m_N \mathbf{u}_N^m)(x_1, y_1) \nonumber \\
					&\quad + (-\Delta)^{-1}(p^*_N - p_N^m)(x_0, y_0) - (-\Delta)^{-1}(p^*_N - p_N^m)(x_1, y_1) \nonumber \\
				&\quad + \frac{1}{\Delta t} \Big( \mathcal{L}^{-1}_{p_N^m(1 + 2\Delta t p_N^m)} (p^*_N - p^m_N)(x_0, y_0)
				- \mathcal{L}^{-1}_{p_N^m(1 + 2\Delta t p_N^m)} (p^*_N - p^m_N)(x_1, y_1) \Big).
				\end{align}
				}
		It is readily seen that
		$$
			{-\ln \left(\frac{p^*_N(x_0, y_0)}{p^*_N(x_1, y_1)}\right) \geq -\ln\left(\frac{\delta}{\beta_0}\right),}
		$$		 
        and
        $$
        \left| \nabla \cdot (p^m_N \mathbf{u}_N^m)(x_0, y_0) - \nabla \cdot (p^m_N \mathbf{u}_N^m)(x_1, y_1) \right| \leq 2 \| \nabla \cdot (p^{m}_N \uu^{m}_N) \|_{\infty}.
        	$$		 
		Furthermore, using Lemma \ref{aux:lem:2}, we obtain
		$$
		\left| (-\Delta)^{-1}(p^*_N - p_N^m)(x_0, y_0) - (-\Delta)^{-1}(p^*_N - p_N^m)(x_1, y_1) \right| \leq 2 C {\color{blue} N} M_N,
		$$
		and
		$$
		\left| \mathcal{L}^{-1}_{p_N^m(1 + 2\Delta t p_N^m)} (p^*_N - p^m_N)(x_0, y_0)
				- \mathcal{L}^{-1}_{p_N^m(1 + 2\Delta t p_N^m)} (p^*_N - p^m_N)(x_1, y_1) \right| \leq 2 C {\color{blue} N} \frac{M_N}{{\min_{\bz \in \Sigma_z} p^{m}_N(\bz)}}.
				$$
				Substituting the inequalities derived above into (\ref{directional-d}), we obtain 
				\begin{equation}\label{contradiction}
					0 \leq \ln\left(\frac{\delta}{\beta_0}\right)
					+ 2 C {\color{blue} N} \left( M_N + \frac{M_N}{\Delta t \cdot {\min_{\bz \in \Sigma_z} p^{m}_N(\bz)}} \right) + 2 \| \nabla \cdot (p^{m}_N \uu^{m}_N) \|_{\infty}.
				\end{equation}
				This is impossible for any fixed $N$ and $\Delta t$, since we can choose $\delta > 0$ to be sufficiently small. 
				This implies that the absolute minimum of $J(p_N, n_N)$ over $Y_{N,\delta}$ can only occur at an interior point of $Y_{N,\delta}$, provided
				$\delta > 0$ is chosen to be sufficiently small. Since $J(p_N, n_N)$ is smooth, we conclude that there exists a solution $(p^*_N, n^*_N) \in \hat Y_N$ such that
				$$\frac{d}{ds}\Bigg|_{s=0} J(p^*_N + s\phi_N, n^*_N + s\psi_N) = 0, \quad \forall (\phi_N, \psi_N) \in \dot{X}_N \times \dot{X}_N.$$
				Thus, $(p^*_N, n^*_N)$ is a positive solution of the modified discrete PNP-NSE system \eqref{scheme:p1}-\eqref{scheme:psi}.
				The uniqueness of positive solutions to \eqref{scheme:p1}-\eqref{scheme:psi} follows from the strict convexity of $J(p_N, n_N)$ over $\hat Y_N$.
				The existence and uniqueness of $\{ \uu^{m+1}_N, \phi^{m+1}_N \}$ can be easily observed from \eqref{scheme:u1}-\eqref{scheme:u3}. 
			 
		\item \textbf{Unconditional Energy Stability:}
				We first derive the energy inequality for \eqref{scheme:p1}-\eqref{scheme:psi}. Taking the test function $v_N = \mu^{m+1}_N$ 
				in \eqref{scheme:p1} and $v_N = \nu^{m+1}_N$ in \eqref{scheme:n1},
				    we have 
				\begin{align}
                    	& \left\langle \frac{p^{m+1}_N - p^{m}_N}{\Delta t},\ \ln{p^{m+1}_N} + \psi^{m+1}_N \right\rangle 
                    	+ \left\langle p^{m}_N \left| \nabla \mu^{m+1}_N \right|^2,\, 1 \right\rangle \nonumber \\
                            & \quad + \left\langle \frac{n^{m+1}_N - n^{m}_N}{\Delta t},\ \ln{n^{m+1}_N} - \psi^{m+1}_N \right\rangle 
                            + \left\langle n^{m}_N \left| \nabla \nu^{m+1}_N \right|^2,\, 1 \right\rangle \nonumber \\
                            & = \left\langle p^{m}_N \mathbf{u}^{m}_N - 2\Delta t\, (p^{m}_N)^2 \nabla \mu^{m+1}_N,\ \nabla \mu^{m+1}_N \right\rangle \nonumber \\
                            & \quad + \left\langle n^{m}_N \mathbf{u}^{m}_N - 2\Delta t\, (n^{m}_N)^2 \nabla \nu^{m+1}_N,\ \nabla \nu^{m+1}_N \right\rangle.
                            \label{thm:es_p1}
                    \end{align}
				From the convexity of the function $x(\ln{x} - 1)$ for $x > 0$, we know 
				\begin{align}
					& \langle \frac{p^{m+1}_N - p^{m}_N}{\Delta t}, \ln{p^{m+1}_N} \rangle  
						\geq \frac{1}{\Delta t} \bigg( \langle p^{m+1}_N(\ln{p^{m+1}_N} - 1), 1 \rangle
						- \langle p^{m}_N(\ln{p^{m}_N} - 1), 1 \rangle \bigg),
						\label{thm:es_p2} \\
					& \langle \frac{n^{m+1}_N - n^{m}_N}{\Delta t}, \ln{n^{m+1}_N} \rangle  
						\geq \frac{1}{\Delta t} \bigg( \langle n^{m+1}_N(\ln{n^{m+1}_N} - 1), 1 \rangle
						- \langle n^{m}_N(\ln{n^{m}_N} - 1), 1 \rangle \bigg).
						\label{thm:es_n2}
				\end{align} 
				Applying $a(a - b) = \frac{1}{2}(a^2 - b^2 + (a - b)^2)$ and the fact that  
				$$\langle p_N^m - n_N^m, \psi^{m+1}_N \rangle = \| \nabla \psi^{m+1}_N \|^2,$$ we have 
				\begin{equation}
				\label{thm:es_psi}
					\langle \frac{p^{m+1}_N - p^{m}_N}{\Delta t} - \frac{n^{m+1}_N - n^{m}_N}{\Delta t}, \psi^{m+1}_N \rangle 
					= \frac{1}{2\Delta t}(\| \nabla \psi^{m+1}_N \|^2 - \| \nabla \psi^{m}_N \|^2 + \| \nabla(\psi^{m+1}_N - \psi^{m}_N) \|^2).
				\end{equation}
				Combining \eqref{thm:es_p1}, \eqref{thm:es_p2}, \eqref{thm:es_n2}
				 with \eqref{thm:es_psi} we obtain 
				{\color{blue} 
				\begin{equation}
					\label{thm:es_pnp}
					\begin{split}
						& \quad \frac{1}{\Delta t}\bigg( 
						\big(E(p^{m+1}_N) + E(n^{m+1}_N) + \frac{1}{2}\|  \nabla \psi^{m+1}_N \| ^2 \big)
						- 
						\big(E(p^{m}_N) + E(n^{m}_N) + \frac{1}{2}\|  \nabla \psi^{m}_N \| ^2 \big)
						\bigg) \\
						& \quad + \frac{1}{2 \Delta t} \| \nabla (\psi^{m+1}_N - \psi^{m}_N) \| ^2 
						  + \langle p^{m}_N \rvert \nabla \mu^{m+1}_N\rvert ^2,1\rangle
						  + \langle n^{m}_N \rvert \nabla \nu^{m+1}_N\rvert ^2,1\rangle \\
						& \leq \langle p^{m}_N {\bf u}^{m}_N
						 - 2\Delta t (p^{m}_N)^2 \nabla \mu^{m+1}_N, \nabla \mu^{m+1}_N \rangle  
								+ \langle n^{m}_N {\bf u}^{m}_N 
								- 2\Delta t (n^{m}_N)^2 \nabla \nu^{m+1}_N, \nabla \nu^{m+1}_N \rangle.
					\end{split}
				\end{equation}
				}
				Now we derive the energy inequality for \eqref{scheme:u1}-\eqref{scheme:u3}. 
				Taking the test function $v_N= \tilde{\uu}^{m+1}_N$ in \eqref{scheme:u1},
				$v_N= \uu^{m+1}_N$ in \eqref{scheme:u2}, 
				we have 
				\begin{equation}
					\label{thm:es_nse_1}
					\begin{split}
						& \quad \frac{1}{2\Delta t}(\| \tilde{\uu}^{m+1}_N \|^2 - \| \uu^{m}_N \|^2 + \| \tilde{\uu}^{m+1}_N - \uu^{m}_N \|^2) + \| \nabla \tilde{\uu}^{m+1}_N \|^2 
							+ \langle \nabla \phi^{m}_N, \tilde{\uu}^{m+1}_N \rangle \\
						& = -\langle p^{m}_N \nabla \mu^{m+1}_N + n^{m}_N \nabla \nu^{m+1}_N, \tilde{\uu}^{m+1}_N \rangle.
					\end{split}  
				\end{equation}
				and
				\begin{equation}
					\label{thm:es_nse_2}
					\frac{1}{2\Delta t}(\| \uu^{m+1}_N \|^2 - \| \tilde{\uu}^{m+1}_N \|^2 + \| \uu^{m+1}_N - \tilde{\uu}^{m+1}_N \|^2) = 0,  
				\end{equation}
				{where we have used \eqref{scheme:u3} that yields}
				\begin{equation*}
					\langle (\uu^{m}_N \cdot \nabla) \tilde{\uu}^{m+1}_N , \tilde{\uu}^{m+1}_N \rangle 
					= \frac{1}{2} \langle \uu^{m}_N, \nabla | \tilde{\uu}^{m+1}_N |^2 \rangle = 0,
				\end{equation*}
				and
				\begin{equation*}
					\langle \nabla (\phi^{m+1}_N - \phi^{m}_N), \uu^{m+1}_N \rangle = 0.
				\end{equation*}
 				To estimate the term $\langle \nabla \phi^{m}_N, \tilde{\uu}^{m+1}_N \rangle$ in \eqref{thm:es_nse_1}, we take the test function $v_N = \nabla \phi^{m}_N$ in \eqref{scheme:u2}, and obtain
				\begin{equation}
					\label{thm:es_nse_3}
					\langle \tilde{\uu}^{m+1}_N, \nabla \phi^{m}_N \rangle = \frac{\Delta t}{2} \left( \| \nabla \phi^{m+1}_N \|^2 - \| \nabla \phi^{m}_N \|^2 + \| \nabla (\phi^{m+1}_N - \phi^{m}_N) \|^2 \right).
				\end{equation}
				Combining \eqref{thm:es_nse_1}, \eqref{thm:es_nse_2} with \eqref{thm:es_nse_3}, we have 
				\begin{equation}
					\label{thm:es_nse}
					\begin{split}
						& \quad \frac{1}{2\Delta t}(\| \uu^{m+1}_N \|^2 - \| \uu^{m}_N \|^2 + \| \tilde{\uu}^{m+1}_N - \uu^{m}_N \|^2 + \| \uu^{m+1}_N - \tilde{\uu}^{m+1}_N \|^2 ) \\
						& \quad + \frac{\Delta t}{2} (\| \nabla \phi^{m+1}_N \|^2 - \| \nabla \phi^{m}_N \|^2 + \| \nabla(\phi^{m+1}_N - \phi^{m}_N) \|^2) + \| \nabla \tilde{\uu}^{m+1}_N \|^2 \\
						& = -\langle p^{m}_N \nabla \mu^{m+1}_N + n^{m}_N \nabla \nu^{m+1}_N, \tilde{\uu}^{m+1}_N \rangle.
					\end{split}
				\end{equation}
				{\color{blue}
				Combining \eqref{thm:es_pnp} with \eqref{thm:es_nse}, we have 
				\begin{equation}
					\label{thm:es_pnp_nse_0}
					\begin{split}
						& \quad \frac{1}{\Delta t}\bigg( 
						\big(E(p^{m+1}_N) + E(n^{m+1}_N) + \frac{1}{2}\|  \nabla \psi^{m+1}_N \| ^2 + \frac{1}{2}\|  \uu^{m+1}_N \| ^2 + \frac{\Delta t^2}{2} \|  \nabla \phi^{m+1}_N \| ^2  \big) \\
						& \quad - 
						\big(E(p^{m}_N) + E(n^{m}_N) + \frac{1}{2}\|  \nabla \psi^{m}_N \| ^2
							+ \frac{1}{2}\|  \uu^{m}_N \| ^2 + \frac{\Delta t^2}{2} \|  \nabla \phi^{m}_N \| ^2  \big)
						\bigg) \\
						& \quad + \|  \nabla \tilde{\uu}^{m+1}_N \| ^2 
						  + \langle p^{m}_N \rvert \nabla \mu^{m+1}_N\rvert ^2,1\rangle
						  + \langle n^{m}_N \rvert \nabla \nu^{m+1}_N\rvert ^2,1\rangle \\
						& \quad + \frac{1}{2\Delta t} \|  \nabla (\psi^{m+1}_N - \psi^{m}_N) \| ^2 
							+ \frac{1}{2\Delta t} \|  \uu^{m+1}_N - \tilde{\uu}^{m+1}_N \| ^2 \\
						& \quad + \frac{1}{2\Delta t} \|  \tilde{\uu}^{m+1}_N - \uu^{m}_N \| ^2
							+ \frac{\Delta t}{2} \|  \nabla(\phi^{m+1}_N - \phi^{m}_N) \| ^2 \\
						\leq &\quad  \langle p^{m}_N {\bf u}^{m}_N 
						- 2\Delta t (p^{m}_N)^2 \nabla \mu^{m+1}_N, \nabla \mu^{m+1}_N \rangle  
						-\langle p^{m}_N \nabla \mu^{m+1}_N, \tilde{\uu}^{m+1}_N \rangle \\
						&\quad	+ \langle n^{m}_N {\bf u}^{m}_N 
						- 2\Delta t (n^{m}_N)^2 \nabla \nu^{m+1}_N, \nabla \nu^{m+1}_N \rangle  
						-\langle  n^{m}_N \nabla \nu^{m+1}_N, \tilde{\uu}^{m+1}_N \rangle.
					\end{split}
				\end{equation}
				Now if we denote 
				\begin{align*}
					& \uu^{m}_{*,p} = \uu^{m}_N - 2 \Delta t p^{m}_N \nabla \mu^{m+1}_N, \\
					& \uu^{m}_{*,n} = \uu^{m}_N - 2 \Delta t n^{m}_N \nabla \nu^{m+1}_N, 
				\end{align*}
				    the terms of the right hand side of \eqref{thm:es_pnp_nse_0} can be rewritten as	
				\begin{equation}
					\label{thm:es_pnp_nse_1}
					\begin{split}
						& \quad \langle p^{m}_N \uu^{m}_N - 2\Delta t (p^{m}_N)^2 \nabla \mu^{m+1}_N, \nabla \mu^{m+1}_N \rangle - 
						\langle p^{m}_N \nabla \mu^{m+1}_N, \tilde{\uu}^{m+1}_N \rangle \\
						& = \langle \uu^{m}_N - 2\Delta t p^{m}_N \nabla \mu^{m+1}_N, p^{m}_N \nabla \mu^{m+1}_N \rangle 
							- \langle \tilde{\uu}^{m+1}_N, p^{m}_N \nabla \mu^{m+1}_N \rangle \\
						& = \frac{1}{2\Delta t} \langle \uu^{m}_{*, p} - \tilde{\uu}^{m+1}_N, \uu^{m}_N - \uu^{m}_{*,p} \rangle \\
						& = \frac{1}{2\Delta t} (\langle \uu^{m}_{*,p} - \uu^{m}_N , \uu^{m}_N \rangle  
							+ \langle \uu^{m}_N - \tilde{\uu}^{m+1}_N, \uu^{m}_N \rangle
							- \langle \uu^{m}_{*,p} - \tilde{\uu}^{m+1}_N, \uu^{m}_{*,p} \rangle ) \\
						& = \frac{1}{4\Delta t}\big( 
							-\|  \uu^{m}_N - \uu^{m}_{*,p} \| ^2
							-\|  \uu^{m}_{*,p} - \tilde{\uu}^{m+1}_N \| ^2
							+ \|  \tilde{\uu}^{m+1}_N - \uu^{m}_N \| ^2  
							\big),
					\end{split}
				\end{equation}
				where we have used the following  identity in the last step
				\begin{equation*}
					(a - b) a = \frac{1}{2}\big(a^2 - b^2 + (a - b)^2\big).
				\end{equation*}
				Similarly, we have
				\begin{equation}
					\label{thm:es_pnp_nse_2}
					\begin{split}
						& \quad \langle n^{m}_N \uu^{m}_N - 2\Delta t (n^{m}_N)^2 \nabla \nu^{m+1}_N, \nabla \nu^{m+1}_N \rangle - 
						\langle n^{m}_N \nabla \nu^{m+1}_N, \tilde{\uu}^{m+1}_N \rangle \\
						& = \langle \uu^{m}_N - 2\Delta t n^{m}_N \nabla \nu^{m+1}_N, n^{m}_N 
							\nabla \nu^{m+1}_N \rangle 
							- \langle \tilde{\uu}^{m+1}_N, n^{m}_N \nabla \nu^{m+1}_N \rangle \\
						& = \frac{1}{2\Delta t} \langle \uu^{m}_{*, n} - \tilde{\uu}^{m+1}_N, \uu^{m}_N - \uu^{m}_{*,n} \rangle \\
						& = \frac{1}{4\Delta t}\big( 
							-\|  \uu^{m}_N - \uu^{m}_{*,n} \| ^2
							-\|  \uu^{m}_{*,n} - \tilde{\uu}^{m+1}_N \| ^2
							+ \|  \tilde{\uu}^{m+1}_N - \uu^{m}_N \| ^2  
							\big).
					\end{split}
				\end{equation}
				Now plug \eqref{thm:es_pnp_nse_1} {\color{blue} and} \eqref{thm:es_pnp_nse_2} into \eqref{thm:es_pnp_nse_0}, we have
				\begin{equation*}
					\begin{split}
						& \quad \frac{1}{\Delta t}\bigg( 
						\big(E(p^{m+1}_N) + E(n^{m+1}_N) + \frac{1}{2}\|  \nabla \psi^{m+1}_N \| ^2 + \frac{1}{2}\|  
						\uu^{m+1}_N \| ^2 + \frac{\Delta t^2}{2}\|  \nabla \phi^{m+1}_N \| ^2  \big) \\
						& \quad - 
						\big(E(p^{m}_N) + E(n^{m}_N) + \frac{1}{2}\|  \nabla \psi^{m}_N \| ^2
							+ \frac{1}{2}\|  \uu^{m}_N \| ^2 + \frac{\Delta t^2}{2} \|  \nabla \phi^{m}_N \| ^2  \big)
						\bigg) \\
						&\le - \|  \nabla \tilde{\uu}^{m+1}_N \| ^2 
						  - \langle p^{m}_N \rvert \nabla \mu^{m+1}_N\rvert ^2,1 \rangle
						  - \langle n^{m}_N \rvert \nabla \nu^{m+1}_N\rvert ^2,1 \rangle  \\
						& \quad -\frac{1}{2\Delta t} \|  \nabla (\psi^{m+1}_N - \psi^{m}_N) \| ^2
							- \frac{1}{2\Delta t} \|  \uu^{m+1}_N - \tilde{\uu}^{m+1}_N \| ^2 
							- \frac{\Delta t}{2} \|  \nabla(\phi^{m+1}_N - \phi^{m}_N) \| ^2 \\ 
						& \quad - \frac{1}{4\Delta t}\big( \|  \tilde{\uu}^{m+1}_N - \uu^{m+1}_{*,p} \| ^2
							+ \|  \uu^{m}_N - \uu^{m+1}_{*,p} \| ^2 
						 +\|  \uu^{m}_N - \uu^{m+1}_{*,n} \| ^2 + \| \tilde{\uu}^{m+1}_N -\uu^{m+1}_{*,n}\| ^2 \big)\\
						&\le - \|  \nabla \tilde{\uu}^{m+1}_N \| ^2 
						  - \langle p^{m}_N \rvert \nabla \mu^{m+1}_N\rvert ^2,1\rangle
						  - \langle  n^{m}_N \rvert \nabla \nu^{m+1}_N\rvert ^2,1 \rangle .
					\end{split}
				\end{equation*}
				This yields the energy inequality for \eqref{scheme:p1} - \eqref{scheme:u2}.				
								
				}
	\end{enumerate}
\end{proof}

\section{Well-posedness and Regularity}
\label{ch-PNPNS-sec-reg}

In this section, we shall establish the well-posedness and regularity of the PNP-NS system. 
For simplicity, we shall focus on periodic boundary conditions , for which the regularity of the solution can be determined by the regularity of the initial conditions. More precisely, we set $\Omega = (0, 2\pi)^2$ and assume that 
\begin{equation}\label{bc}
    \begin{array}{cc}
        & (p, n, \psi, \uu)(2\pi, y) = (p, n, \psi, \uu)(0, y), \quad y \in (0, 2\pi); \\
        & (p, n, \psi, \uu)(x, 2\pi) = (p, n, \psi, \uu)(x, 0), \quad x \in (0, 2\pi).
    \end{array}
\end{equation}

\begin{theorem} \label{existence} 
Let $\Omega = (0, 2\pi)^2$, and assume the initial conditions $(p^{in}, n^{in}) \in L^r(\Omega) \cap W^{2,q}(\Omega)$, with $r = 2q > 4$, are positive and satisfy $\displaystyle\int_\Omega (p^{in} - n^{in}) \, dx = 0$, and the velocity  $\uu^{in} \in W^{1,r}_0(\Omega,\mathbb{R}^2)$ is divergence-free. Then there exists a unique global 
strong solution of \eqref{gov_equ_1}--\eqref{gov_equ_5} with the initial condition \eqref{ic} and the periodic boundary condition \eqref{bc}. Moreover, there exists a constant $C_r$
depending on $\varepsilon$ and the initial energy $E(p^{in}, n^{in}, \uu^{in})$, $\| p^{in}\|_{L^r}$, $\| n^{in}\|_{L^r}$, and $\| \uu^{in}\|_{L^2}$
such that 
\begin{equation*}
    \sup_{0 \leq t < \infty} \| p(t) \|_{L^{r}(\Omega)} \le C_r, \quad
    \sup_{0 \le t < \infty}  \| n(t) \|_{L^{r}(\Omega)} \le C_r, \quad
    \sup_{0 \le t < \infty} \| \psi(t) \|_{W^{2,r}(\Omega)} \le C_r. 
\end{equation*}
Furthermore, 
\begin{equation*}
    \sup_{0 \leq t < \infty} \| p(t) \|_{L^{\infty}(\Omega)} + \sup_{0 \le t < \infty} \| p(t) \|_{H^{1}(\Omega)} \le C, \quad
    \sup_{0 \leq t < \infty} \| n(t) \|_{L^{\infty}(\Omega)} + \sup_{0 \le t < \infty} \| n(t) \|_{H^{1}(\Omega)} \le C,
\end{equation*}
and the velocity field $\uu$ satisfies
$$
\| \uu \|_{L^\infty(0,T; H^1(\Omega))}^2 + \int_0^T \| \uu(t) \|_{H^2(\Omega)}^2\, dt \le  C T,
$$
for any $0 < T < \infty$, where $C$ depends on initial energy, {$\| p^{in} \|_{L^{r}(\Omega)}$, $\| n^{in} \|_{L^{r}(\Omega)}$}, and $\| \uu^{in}\|_{H^1(\Omega)}$. 
\end{theorem}

{
\begin{proof}
A similar result for blocking boundary conditions has been obtained by Constantin and Ignatova \cite{constantin2019nernst}. Their argument remains applicable for periodic boundary conditions, which will be sketched here for completeness. For the full proof, refer to \cite{constantin2019nernst}.   

\textbf{Step 1:} Firstly, we have
\begin{equation}\label{thm_reg_psi_bound}
    \| \psi \|_{L^\infty(\Omega \times [0, T])} \leq C E(p^{in}, n^{in},  \psi^{in}),
\end{equation}		
which is a direct application of 
Lemma 1 in \cite{constantin2019nernst}, following the same proof for periodic boundary conditions.  

\noindent \textbf{Step 2:} We aim to show $p > 0$ and $n > 0$ in $\Omega \times [0,T]$. To see this, let $F:\mathbb{R} \to \mathbb{R}$ be a nonnegative,
$C^2$-convex function such that $F(t) = 0$ for $t > 0$, and $F(t) > 0$ for $t < 0$, and
$$F''(t)t^2 \le C F(t), \quad \forall t \in \mathbb{R}.$$ 
Multiplying \eqref{gov_equ_1} by $F'(p)$ and integrating over $\Omega$, using the periodic boundary conditions  and integration by parts,
we obtain that $\displaystyle\int_\Omega \uu \cdot \nabla F(p) \, dx = -\int_\Omega \nabla \cdot \uu \, F(p)\, dx = 0$, and hence
$$
\frac{d}{dt}\int_\Omega F(p) \, dx = -\int_\Omega F''(p) \left[ |\nabla p|^2 + p \nabla \psi \cdot \nabla p \right]\, dx,
$$
which, combined with the Cauchy-Schwarz inequality $| p \nabla \psi \cdot \nabla p | \leq \frac{1}{2} |\nabla p|^2 + \frac{1}{2} p^2 |\nabla \psi|^2$, yields
\begin{equation} \label{thm_reg_aux1}
    \frac{d}{dt}\int_\Omega F(p) \, dx \le - \frac{1}{2} \int_{\Omega} F''(p) |\nabla p|^2 \, dx  + \frac{1}{2} \int_{\Omega} F''(p) p^2 |\nabla \psi|^2 \, dx.
\end{equation}
From the properties of $F$, we have
$$
\frac{d}{dt}\int_\Omega F(p) \, dx  \le \frac{C}{2} \| \nabla\psi\|_{L^\infty(\Omega)}^2 \int_\Omega F(p)\, dx. 
$$
By the Gronwall inequality and $\displaystyle\int_\Omega F(p^{in}) \, dx = 0$, we conclude that 
$\displaystyle\int_\Omega F(p)\, dx = 0$, and hence 
$F(p) \equiv 0$, which yields that $p > 0$ in $\Omega \times [0,T]$. Similarly, $n > 0$ in $\Omega \times [0,T]$. 

\noindent \textbf{Step 3:} We aim to estimate the local uniform bound for $\| (p, n) \|_{L^{1}_{t}L^{r}(\Omega)}$.  Because of the energy dissipation law \eqref{law}, we have 
\begin{equation*}
    \int_0^T \int_{\Omega} p | \nabla \ln (p e^{\psi})|^2 \, dx \, dt =   \int_0^{T} \int_{\Omega} p | \nabla \mu |^2 \, dx \, dt \leq E(p^{in}, n^{in}, \psi^{in}) \triangleq \Gamma.
\end{equation*}
Using \eqref{thm_reg_psi_bound} in Step 1, we know that the auxiliary function 
\begin{equation*}
    \tilde{p} \triangleq p e^{\psi}
\end{equation*}
satisfies the estimate
\begin{equation*}
    \int_{0}^{T} \int_{\Omega} \frac{1}{\tilde{p}} | \nabla \tilde{p}|^2 \, dx \, dt \leq \Gamma e^{C\Gamma}. 
\end{equation*}
From the mass conservation property and \eqref{thm_reg_psi_bound}, we have 
\begin{equation*}
     \int_{\Omega} p e^{\psi} \, dx  \leq  e^{C\Gamma} \int_{\Omega} p^{in} \, dx.
\end{equation*}
Combining the previous two equations, for any $t_0 \in [0, T]$ and $\tau \in [0, T - t_0]$, we have 
\begin{equation*}
    \int_{t_0}^{t_0 + \tau} \| \sqrt{\tilde{p}} \|_{H^1(\Omega)} \, dt \leq  e^{C\Gamma} \left( \Gamma + \tau \int_{\Omega} p^{in} \, dx \right). 
\end{equation*}
Thus, from the Sobolev embedding $\| \sqrt{\tilde{p}} \|_{L^{r}(\Omega)} \leq \| \sqrt{\tilde{p}} \|_{H^{1}(\Omega)}$ for any $r \in [1, \infty)$, applying \eqref{thm_reg_psi_bound} again, we have the local uniform estimate for $\| p \|_{L^{r}(\Omega)}$
\begin{equation}\label{thm_reg_local}
    \int_{t_0}^{t_0 + \tau} \| p \|_{L^r(\Omega)} \, dt \leq  C_r e^{C\Gamma} \left( \Gamma + \tau \int_{\Omega} p^{in} \, dx \right),
\end{equation}
where $C_r$ depends on $r$. Similar estimates hold for $n$.

\noindent \textbf{Step 4:} Now we can estimate the global bound for $\| (p, n) \|_{L^{r}(\Omega)}$. To do this, taking $F(p) = \frac{1}{r(r-1)} p^{r} $ in \eqref{thm_reg_aux1}, we obtain
\begin{equation*}\label{thm_reg_lr_est}
    \frac{1}{r(r-1)} \frac{d}{dt} \int_{\Omega} | p |^r \, dx  \leq -\frac{1}{2} \int_{\Omega} | \nabla p|^2 p^{r-2} \, dx + \frac{1}{2} \| \nabla \psi \|_{L^{\infty}(\Omega)} \int_{\Omega} |p|^r \, dx.  
\end{equation*}
Similar estimates hold for $n$:
\begin{equation*}
    \frac{1}{r(r-1)} \frac{d}{dt} \int_{\Omega} | n |^r \, dx  \leq -\frac{1}{2} \int_{\Omega} | \nabla n|^2 n^{r-2} \, dx + \frac{1}{2} \| \nabla \psi \|_{L^{\infty}(\Omega)} \int_{\Omega} |n|^r \, dx.  
\end{equation*}	
From the regularity of the Poisson equation, we know that
\begin{equation*}
    \| \nabla  \psi \|_{L^{\infty}(\Omega)} \leq \frac{C_{r}}{\varepsilon} \| p - n \|_{L^{r}(\Omega)} \leq \frac{C_{r}}{\varepsilon} \left( \| p \|_{L^{r}(\Omega)} + \| n \|_{L^{r}(\Omega)} \right).
\end{equation*}
From here we obtain
\begin{equation}\label{thm_reg_gronwall}
    \frac{1}{r(r-1)} \frac{d}{dt} A_{r} \leq -\frac{1}{2} \int_{\Omega} \left( | \nabla p |^2 p^{r-2} + | \nabla n|^2 n^{r-2} \right) \, dx + \frac{C_r}{2\varepsilon} A_r^{\frac{1}{r}} A_r,
\end{equation}
where $A_r = \| p \|^r_{L^{r}(\Omega)} + \| n \|^{r}_{L^{r}(\Omega)}$.  From \eqref{thm_reg_local}, we have 
\begin{equation}\label{thm_reg_local_combined}
    \int_{t_0}^{t_0 + \tau} A_r^{\frac{1}{r}} \, dt \leq C_r e^{C\Gamma} \left( \Gamma + \tau \int_{\Omega} p^{in} + n^{in} \, dx \right ) \triangleq \Gamma^{*}, 
\end{equation}
Combining this with \eqref{thm_reg_gronwall}, we obtain
\begin{equation}\label{thm_reg_ext}
    A_r(t_0 + \tau) \leq A_r(t_0) e^{C_r \Gamma^{*} / \varepsilon}.
\end{equation}
Now we cover the interval $[0,T]$ with fixed time step intervals $\{ (t_k, t_k + \frac{\tau}{2}) \mid k \in \mathbb{N} \}$.  From \eqref{thm_reg_local_combined}, for any $k$, there exists some $t^{*} \in [t_k - \frac{\tau}{2}, t_k]$ such that $A_r^{\frac{1}{r}}(t^{*}) \leq  \Gamma_{\tau} \triangleq \max\left( \frac{2}{\tau} \Gamma^{*}, \left( \| p^{in} \|_{L^r(\Omega)} + \| n^{in} \|_{L^{r}(\Omega)} \right)^{\frac{1}{r}} \right)$, which, combining with \eqref{thm_reg_ext} and $[t_k, t_k + \frac{\tau}{2}] \subset [t^{*}, t^{*} + \tau]$, yields
\begin{equation*}
    \sup_{t\in [t_k, t_k + \frac{\tau}{2}]} A_r(t) \leq \Gamma_{\tau},
\end{equation*} 
for a slightly different $\Gamma_{\tau}$. 
Notice that the right-hand side only depends on initial energy $\Gamma$, initial ion mass $\int_{\Omega} p^{in} + n^{in} \, dx$, $\varepsilon$, and $r$; it is independent of time $T$. We can extend the estimate to the entire time interval by an induction argument, and from the regularity of the Poisson equation obtain the global bound 
\begin{equation}\label{thm_reg_lr_bound}
    \sup_{0 \leq t < \infty} \| p(t) \|_{ L^{r}(\Omega)}, \quad \sup_{0 \leq t < \infty} \| n(t)  \|_{ L^{r}(\Omega)}, \quad \sup_{0 \leq t < \infty} \| \psi(t) \|_{W^{2,r}(\Omega)}  \leq C_r^{*},
\end{equation}
where $C_r^{*}$ depends on $r$, $\varepsilon$, initial energy,  initial ion mass, $\| p^{in} \|_{L^{r}(\Omega)},  \| n^{in} \|_{L^{r}(\Omega)}$.
Returning to \eqref{thm_reg_gronwall}, we know that 
\begin{equation}\label{thm_reg_grad_p_local_uniform}
    \int_{t_0}^{t_0 + \tau} \int_{\Omega} \left( |\nabla p|^2 p^{r-2} + |\nabla n|^2 n^{r-2} \right) \, dx \, dt \leq \Gamma_{\tau}, 
\end{equation}
for some $\Gamma_\tau$ depending on $\Gamma^{*}$ and $\tau$.	 

\textbf{Step 5:} Now we are ready to estimate $\| p,  n \|_{L^{\infty}(\Omega)}$.
Multiplying \eqref{gov_equ_1} by $-\Delta p$ and integrating, we have 
\begin{align}
    \frac{1}{2} \frac{d}{dt} \| \nabla p \|^2_{L^{2}(\Omega)} &= - \| \Delta p \|_{L^{2}(\Omega)}^2 - \int_{\Omega} \nabla \cdot (p \nabla \psi) \Delta p \, dx - \int_{\Omega} \uu \cdot \nabla p \Delta p \, dx,  \nonumber \\
    & \leq  - \| \Delta p \|_{L^{2}(\Omega)}^2  + \| \nabla p \|_{L^{4}(\Omega)}  \left( \| \nabla \psi \|_{L^{4}(\Omega)} + \| \uu \|_{L^{4}(\Omega)} \right) \| \Delta p \|_{L^{2}(\Omega)}   \nonumber \\
    & \quad + \| p \|_{L^{4}(\Omega)} \| \Delta \psi \|_{L^{4}(\Omega)}  \| \Delta p \|_{L^{2}(\Omega)}.  \label{thm_reg_gronwall_grad_p}
\end{align}
We have a global bound for $\| \nabla \psi \|_{L^{4}(\Omega)}$, $\| p \|_{L^{4}(\Omega)}$, $\| \Delta \psi \|_{L^{4}(\Omega)}$ from \eqref{thm_reg_lr_bound}. 
And from energy law \eqref{law}, we know that $\max_{t \in [0,T]} \| \uu(t) \|_{L^2(\Omega)}^{2} $ and $ \int_{0}^{T}  \| \nabla \uu \|_{L^2(\Omega)}^{2}  \, dt$ are bounded by initial energy $\Gamma$. Hence, we have the uniform bound for $\| \uu \|_{L^{4}([0, T] \times \Omega)}$,
\begin{align*}
    \int_{0}^{T} \| \uu \|_{L^{4}(\Omega)}^4 \, dt & \leq \int_{0}^{T} C \| \uu \|_{L^2(\Omega)}^{2}   \| \nabla \uu \|_{L^2(\Omega)}^{2}  \, dt, \\
    & \leq C  \max_{t \in [0,T]} \| \uu(t) \|_{L^2(\Omega)}^{2} \int_{0}^{T}  \| \nabla \uu \|_{L^2(\Omega)}^{2}  \, dt, \\
    & \leq C \Gamma^2.
\end{align*}
Applying these bounds to \eqref{thm_reg_gronwall_grad_p}, we have 
\begin{equation*}
    \frac{d}{dt} \| \nabla p \|_{L^{2}(\Omega)}^2 + \| \Delta p \|_{L^{2}(\Omega)}^2 \leq \Gamma \| \nabla p \|_{L^{2}(\Omega)}^2. 
\end{equation*}
Applying the local uniform bound for $\| \nabla p \|_{L^{2}_t L^{2}(\Omega)} $ from \eqref{thm_reg_grad_p_local_uniform} with $r = 2$, we cover the interval $[0,T]$ with fixed time step intervals $\{ (t_k, t_k + \frac{\tau}{2}) \mid k \in \mathbb{N} \}$. With a similar argument as in Step 4, we have 
\begin{equation*}
    \sup_{0 \leq t \leq T} \| \nabla p \|_{L^{2}(\Omega)} \leq \Gamma_{\tau},
\end{equation*}
and for any $[t_0, t_0 + \tau] \subset [0, T]$
\begin{equation*}
    \int_{t_0}^{t_0 + \tau} \| \Delta p \|_{L^{2}(\Omega)}^2 \, dt \leq \Gamma_{\tau}.
\end{equation*}
Hence, we have the local uniform bound of $\| p \|_{L^{\infty}(\Omega)}$
\begin{equation}\label{thm_reg_p_infty_local_uniform}
    \int_{t_0}^{t_0 + \tau} \| p \|_{L^{\infty}(\Omega)} \, dt \leq \int_{t_0}^{t_0 + \tau} \| p \|_{H^{2}(\Omega)}  \, dt \leq \Gamma_\tau.
\end{equation}
Now, multiplying \eqref{gov_equ_1} by $p^{r-1}$ and integrating over $\Omega$, we have
\begin{align*}
    \frac{1}{r} \frac{d}{dt} \int_{\Omega} p^r \, dx   
    &= -(r-1) \int_{\Omega} |\nabla p|^2 p^{r-2} \, dx +  \int_{\Omega} \nabla p \cdot \nabla \psi \, p^{r-1} \, dx + 		\int_{\Omega} \Delta \psi \, p^{r} \, dx,  \\
    &\leq -\frac{r-1}{2} \int_{\Omega} |\nabla p|^2 p^{r-2} \, dx + \frac{1}{2(r-1)} \| \nabla \psi \|^2_{L^{\infty}(\Omega)} \| p \|^r_{L^{r}(\Omega)} + \| \Delta \psi \|_{L^{\infty}(\Omega)} \| p \|^r_{L^{r}(\Omega)}.
\end{align*}
Therefore, for any $t \ge t_0$ we have
\begin{equation*}
    \| p(t) \|_{L^{r}(\Omega)} \leq \| p(t_0)  \|_{L^{r}(\Omega)} e^{\int_{t_0}^{t} \left( \frac{1}{2(r-1)} \| \nabla \psi \|^2_{L^{\infty}(\Omega)}  + \| \Delta \psi \|_{L^{\infty}(\Omega)} \right) \, dt }.
\end{equation*}				
Taking the limit as $r \rightarrow \infty$, we obtain
\begin{equation*}
    \| p (t)  \|_{L^{\infty}(\Omega)} \leq \| p (t_0)  \|_{L^{\infty}(\Omega)}  e^{\int_{t_0}^{t}  \| \Delta \psi \|_{L^{\infty}(\Omega)} \, dt }.
\end{equation*}
Combining $\| p \|_{L^{\infty}(\Omega)}$ local uniform estimate \eqref{thm_reg_p_infty_local_uniform} and cover interval $[0,T]$ with fixed time step intervals $\{ (t_k, t_k + \frac{\tau}{2}) \mid k \in \mathbb{N} \}$, with a similar induction argument as in Step 4,  we have 
\begin{equation*}
    \sup_{0 \leq t < T} \| p(t) \|_{L^{\infty}(\Omega)} \leq \Gamma_{\tau}.
\end{equation*}
Since the forcing term in \eqref{gov_equ_4} is in $L^{2}(\Omega)$, from the energy inequality \eqref{gov_equ_1} and on the standard estimates on non-stationary Navier-Stokes equation, we have 
\begin{equation*}
    \| \uu \|_{L^\infty(0,T; H^1(\Omega))}^2 + \int_0^T \| \uu(t) \|_{H^2(\Omega)}^2\, dt \le  C T,
\end{equation*}
where $C$ depends on the initial energy and other constants.
This completes the proof. 
\end{proof} 
}

\begin{corollary}(Maximum principle)\label{maximal_principle}
	Assuming $p^{in} \geq \delta_p, \ n^{in} \geq \delta_n$ for some $\delta_p, \delta_n > 0$, then we have $p \geq \delta_p, n \geq \delta_n$ on $\Omega \times [0, T]$. 
\end{corollary}
\begin{proof}
This proof follows from the positivity proof for $(p, n)$ in Theorem \ref{existence}.
\end{proof}

Next we derive the higher order regularity for the global strong solutions obtained in Theorem \ref{existence} when the initial data $(p^{in},
n^{in}, \uu^{in})$ is assumed to have higher regularity.

{
\begin{theorem}\label{regularity}
Suppose, in addition, that the initial data satisfies $(p^{in}, n^{in}, \uu^{in}) \in H^{2m+1}(\Omega) \times H^{2m+1}(\Omega) \times H^{2m+1}(\Omega)$ for $m \geq 0$. Then the solution $(p, n, \uu)$ obtained in Theorem \ref{existence} satisfies
    \begin{equation*}
        \sum_{k=0}^{m+1} \| (\partial_t^k p, \partial_t^k n, \partial_t^k \uu) \|_{L^2(0,T; H^{2m+2-2k}(\Omega))} 
        \leq C(T, \| p^{in} \|_{H^{2m+1}(\Omega)}, \| n^{in} \|_{H^{2m+1}(\Omega)}, \| \uu^{in} \|_{H^{2m+1}(\Omega)}).
    \end{equation*}
\end{theorem}

\begin{proof}
The proof proceeds by induction on $m$. The case $m = 0$ was proved in Theorem \ref{existence}. Assume the theorem holds for some non-negative integer $m$, and suppose the initial data satisfies
    \begin{equation*}
        (p^{in}, n^{in}, \uu^{in}) \in H^{2m+3}(\Omega) \times H^{2m+3}(\Omega) \times H^{2m+3}(\Omega).
    \end{equation*}
We can verify that
    \begin{equation}\label{high_init_cond}
        (\partial_t^{k} p^{in}, \partial_t^{k} n^{in}, \partial_t^{k} \uu^{in} ) \in H^{2m - 2k + 3}(\Omega) \times H^{2m - 2k + 3}(\Omega) \times H^{2m - 2k + 3}(\Omega), \quad \forall\, k = 1, \ldots, m+1.
    \end{equation}

Now set $\tilde{p} \coloneqq \partial_t^{m+1} p$, $\tilde{n} \coloneqq \partial_t^{m+1} n$, $\tilde{\uu} \coloneqq \partial_t^{m+1} \uu$, and $\tilde{\psi} \coloneqq \partial_t^{m+1} \psi$. Differentiating the system \eqref{gov_equ_1}--\eqref{gov_equ_5} with respect to $t^{m+1}$, we find that $(\tilde{p}, \tilde{n}, \tilde{\uu})$ satisfies the following system:
    \begin{align}
        \tilde{p}_t - \Delta \tilde{p} &= \nabla \cdot \left( \partial_t^{m+1} (p \nabla \psi - p \uu) \right), \label{high_p} \\
        \tilde{n}_t - \Delta \tilde{n} &= \nabla \cdot \left( \partial_t^{m+1} (-n \nabla \psi - n \uu) \right), \label{high_n} \\
        -\epsilon \Delta \tilde{\psi} &= \tilde{p} - \tilde{n}, \\
        \tilde{\uu}_t - \Delta \tilde{\uu} &= \partial_t^{m+1} \left( -\nabla P - (\uu \cdot \nabla) \uu - \nabla \psi (p - n) \right), \label{high_u} \\
        \nabla \cdot \tilde{\uu} &= 0.
    \end{align}

\textbf{Step 1:} Multiply equation \eqref{high_p} by $\tilde{p}$ and integrate over $\Omega$. Observing that there are no boundary term contributions due to the periodic boundary condition, we obtain
    \begin{equation*}
        \begin{split}
            \frac{1}{2} \frac{d}{dt} \| \tilde{p} \|_{L^2(\Omega)}^2 &= - \| \nabla \tilde{p} \|_{L^2(\Omega)}^2 + \int_{\Omega} \partial_t^{m+1} (p \nabla \psi - p \uu) \cdot \nabla \tilde{p} \, dx, \\
            &\leq -\frac{1}{2} \| \nabla \tilde{p} \|_{L^2(\Omega)}^2 + \frac{1}{2} \| \partial_t^{m+1} (p \nabla \psi - p \uu) \|_{L^2(\Omega)}^2.
        \end{split}
    \end{equation*}
Applying the induction hypothesis, we have
    \begin{equation}\label{high_p_term}
        \begin{split}
            &\quad \| \partial_t^{m+1} (p \nabla \psi - p \uu) \|_{L^2(0,T; L^2(\Omega))} \\
            &\leq \| \tilde{p} \|_{L^2(0,T; L^2(\Omega))} \| \nabla \psi - \uu \|_{L^{\infty}(0,T; L^{\infty}(\Omega))} \\
            &\quad + \left( \| \nabla \tilde{\psi} \|_{L^2(0,T; L^2(\Omega))} + \| \tilde{\uu} \|_{L^2(0,T; L^2(\Omega))} \right) \| p \|_{L^{\infty}(0,T; L^{\infty}(\Omega))} \\
            &\quad + \sum_{k=1}^{m} \| \partial_t^{k} p \|_{L^{\infty}(0,T; L^{2}(\Omega))} \| \partial_t^{m+1-k} (\nabla \psi - \uu) \|_{L^{2}(0,T; L^{\infty}(\Omega))} \\
            &\leq C \left( \| \tilde{p} \|_{L^2(0,T; L^2(\Omega))}^2 + \| \tilde{n} \|_{L^2(0,T; L^2(\Omega))}^2 + \| \tilde{\uu} \|_{L^2(0,T; L^2(\Omega))}^2 \right) \\
            &\quad + \sum_{k=1}^{m} \| \partial_t^{k} p \|_{L^2(0,T; H^1(\Omega))} \| \partial_t^{k+1} p \|_{L^2(0,T; H^{-1}(\Omega))} \| \partial_t^{m+1-k} (\nabla \psi - \uu) \|_{L^{2}(0,T; H^{2}(\Omega))} \\
            &\leq C.
        \end{split}
    \end{equation}
Here $C$ depends on $T$ and the initial data and we used the estimate: for any function $f$
\begin{equation}\label{interpolation-ineq}
	\| f \|_{L^{\infty}_t H^{l+1}_{\Omega}} \leq C \| f \|_{L^2_{t} H^{l+1}_{\Omega}}    \| \pt f \|_{L^2_{t} H^{l}_{\Omega}}.
\end{equation}
Therefore, from the initial condition \eqref{high_init_cond}, we have
    \begin{equation}\label{gronwall_high_p}
        \sup_{0 \leq t \leq T} \| \tilde{p}(t) \|_{L^{2}(\Omega)} + \int_{0}^{T} \| \nabla \tilde{p} \|_{L^{2}(\Omega)}^2 \, dt \leq C.
    \end{equation}
Similarly, we obtain
    \begin{equation}\label{gronwall_high_n}
        \sup_{0 \leq t \leq T} \| \tilde{n}(t) \|_{L^{2}(\Omega)} + \int_{0}^{T} \| \nabla \tilde{n} \|_{L^{2}(\Omega)}^2 \, dt \leq C.
    \end{equation}

Multiplying \eqref{high_u} by $\tilde{\uu}$ and integrating over $\Omega$, we have
    \begin{equation*}
        \frac{1}{2} \frac{d}{dt} \| \tilde{\uu} \|_{L^2(\Omega)}^2 = - \| \nabla \tilde{\uu} \|_{L^{2}(\Omega)}^2 - \int_{\Omega} \partial_t^{m+1} \left( (\uu \cdot \nabla) \uu + \nabla \psi (p - n) \right) \cdot \tilde{\uu} \, dx.
    \end{equation*}
Applying the Ladyzhenskaya inequality and the induction hypothesis, we estimate
    \begin{equation*}
        \begin{split}
            &\int_0^T \int_{\Omega} \partial_t^{m+1} \left( (\uu \cdot \nabla) \uu \right) \cdot \tilde{\uu} \, dx \, dt \\
            &= -\int_0^T \int_{\Omega} (\tilde{\uu} \cdot \nabla) \tilde{\uu} \cdot \uu + \sum_{j=1}^{m} (\partial_t^{j} \uu \cdot \nabla) \tilde{\uu} \cdot \partial_t^{m+1 - j} \uu \, dx \, dt \\
            &\leq \int_0^T \| \tilde{\uu} \|_{L^4(\Omega)} \| \uu \|_{L^4(\Omega)} \| \nabla \tilde{\uu} \|_{L^2(\Omega)} \, dt + \sum_{j=1}^{m} \| \partial_t^{j} \uu \|_{L^{\infty}(\Omega)} \| \partial_t^{m+1-j} \uu \|_{L^2(\Omega)} \| \nabla \tilde{\uu} \|_{L^2(\Omega)} \, dt \\
            &\leq \int_0^T \| \tilde{\uu} \|_{L^2(\Omega)}^{1/2} \| \uu \|_{L^2(\Omega)}^{1/2} \| \nabla \uu \|_{L^2(\Omega)}^{1/2} \| \nabla \tilde{\uu} \|_{L^2(\Omega)}^{3/2} \, dt + \sum_{j=1}^{m} \| \partial_t^{j} \uu \|_{L^2(0,T; H^2(\Omega))} \| \partial_t^{m+1-j} \uu \|_{L^{\infty}(0,T; L^2(\Omega))} \\
            &\leq \frac{1}{2} \| \nabla \tilde{\uu} \|_{L^2(0,T; L^{2}(\Omega))}^2 + C,
        \end{split}
    \end{equation*}
where $C$ depends on $T$ and the initial data. We also have
    \begin{equation*}
        \| \partial_t^{m+1} \left( \nabla \psi (p - n) \right) \|_{L^2(0,T; L^2(\Omega))} \leq C.
    \end{equation*}
Combining these estimates, we obtain
    \begin{equation}\label{gronwall_high_u}
        \sup_{0 \leq t \leq T} \| \tilde{\uu}(t) \|_{L^2(\Omega)}^2 + \int_{0}^{T} \| \nabla \tilde{\uu} \|_{L^{2}(\Omega)}^2 \, dt \leq C.
    \end{equation}
    
\textbf{Step 2:} Multiply \eqref{high_p} by $\Delta \tilde{p}$ and integrate over $\Omega$ to obtain
    \begin{equation*}
        \begin{split}
            \frac{1}{2} \frac{d}{dt} \| \nabla \tilde{p} \|_{L^2(\Omega)}^2 &= - \| \Delta \tilde{p} \|_{L^2(\Omega)}^2 + \int_{\Omega} \nabla \cdot \left( \partial_t^{m+1} (p \nabla \psi - p \uu) \right) \cdot \Delta \tilde{p} \, dx \\
            &\leq -\frac{1}{2} \| \Delta \tilde{p} \|_{L^2(\Omega)}^2 + \frac{1}{2} \| \nabla \cdot \left( \partial_t^{m+1} (p \nabla \psi - p \uu) \right) \|_{L^2(\Omega)}^2.
        \end{split}
    \end{equation*}
Using estimates similar to \eqref{high_p_term} and the results \eqref{gronwall_high_p}, \eqref{gronwall_high_n}, and \eqref{gronwall_high_u}, we verify that
    \begin{equation*}
        \| \nabla \cdot \left( \partial_t^{m+1} (p \nabla \psi - p \uu) \right) \|_{L^2(0,T; L^2(\Omega))}^2 \leq C.
    \end{equation*}
Combining these inequalities with the initial condition \eqref{high_init_cond}, we obtain
    \begin{equation}\label{high_p_improv}
        \sup_{0 \leq t \leq T} \| \nabla \tilde{p}(t) \|_{L^{2}(\Omega)}^2 + \int_0^{T} \| \Delta \tilde{p} \|_{L^2(\Omega)}^2 \, dt \leq C.
    \end{equation}
Analogously, we have
    \begin{equation}\label{high_n_improv}
        \sup_{0 \leq t \leq T} \| \nabla \tilde{n}(t) \|_{L^{2}(\Omega)}^2 + \int_0^{T} \| \Delta \tilde{n} \|_{L^2(\Omega)}^2 \, dt \leq C.
    \end{equation}

Multiplying \eqref{high_u} by $\Delta \tilde{\uu}$ and integrating over $\Omega$, we obtain
    \begin{equation*}
        \begin{split}
            \frac{1}{2} \frac{d}{dt} \| \nabla \tilde{\uu} \|_{L^2(\Omega)}^2 &= - \| \Delta \tilde{\uu} \|_{L^{2}(\Omega)}^2 - \int_{\Omega} \partial_t^{m+1} \left( (\uu \cdot \nabla) \uu + \nabla \psi(p - n) \right) \cdot \Delta \tilde{\uu} \, dx \\
            &\leq -\frac{1}{2} \| \Delta \tilde{\uu} \|_{L^{2}(\Omega)}^2 + \| \partial_t^{m+1} \left( (\uu \cdot \nabla) \uu \right) \|_{L^2(\Omega)}^2 + \| \partial_t^{m+1} \left( \nabla \psi(p - n) \right) \|_{L^2(\Omega)}^2.
        \end{split}
    \end{equation*}
Applying the Ladyzhenskaya inequality and induction estimates, we have
    \begin{equation*}
        \begin{split}
            \| \partial_t^{m+1} \left( (\uu \cdot \nabla) \uu \right) \|_{L^2(0,T; L^2(\Omega))} &\leq \int_0^T \| \tilde{\uu} \|_{L^4(\Omega)}^2 \| \nabla \uu \|_{L^4(\Omega)}^2 \, dt + \| \uu \|_{L^{2}(0,T; L^{\infty}(\Omega))} \| \nabla \tilde{\uu} \|_{L^{\infty}(0,T; L^2(\Omega))} \\
            &\quad + \sum_{j=1}^{m} \| \partial_t^{j} \uu \|_{L^2(0,T; H^{2}(\Omega))} \| \partial_t^{m+1-j} \uu \|_{L^{2}(0,T; H^{2}(\Omega))} \| \partial_t^{m+2-j} \uu \|_{L^{2}(0,T; L^{2}(\Omega))} \\
            &\leq C.
        \end{split}
    \end{equation*}
Therefore, with the initial condition \eqref{high_init_cond}, we have
    \begin{equation}\label{high_u_improv}
        \sup_{0 \leq t \leq T} \| \nabla \tilde{\uu}(t) \|_{L^{2}(\Omega)}^2 + \int_0^{T} \| \Delta \tilde{\uu} \|_{L^2(\Omega)}^2 \, dt \leq C.
    \end{equation}
Using estimates \eqref{high_p_improv}, \eqref{high_n_improv}, and \eqref{high_u_improv} in equations \eqref{high_p}, \eqref{high_n}, and \eqref{high_u}, we verify that
    \begin{equation*}
        \| \partial_t \tilde{p} \|_{L^2(0,T; L^2(\Omega))}, \quad \| \partial_t \tilde{n} \|_{L^2(0,T; L^2(\Omega))}, \quad \| \partial_t \tilde{\uu} \|_{L^2(0,T; L^2(\Omega))} \leq C.
    \end{equation*}
This completes the proof.
\end{proof}
}

\section{Error Analysis}

{In this section, we will carry out a detailed error analysis for the positivity-preserving scheme \eqref{scheme:p1}-\eqref{scheme:u3} under the periodic boundary condition \eqref{bc}, for which the scheme \eqref{scheme:p1}-\eqref{scheme:u3} can be made more specific as follows:}

We denote the Fourier collocation points as  $\Sigma_N= \left\{\left( x_i = \dfrac{2\pi i}{N},\ y_j = \dfrac{2\pi j}{N} \right) \bigg| \ 0\leq i,j \leq N-1\right\}$. Then the discrete inner product for two functions $u, v$ is defined by  
\begin{equation*}
	\langle u,v \rangle=\sum_{\mathbf{z} \in \Sigma_N} w_{\mathbf{z}}u(\mathbf{z})v(\mathbf{z}),
\end{equation*}
where $w_{\mathbf{z}}=\left( \dfrac{2\pi}{N} \right)^2$ is the quadrature weight in 2D.

We also introduce the corresponding induced discrete norm by $\|u\|=\langle u,u \rangle^{\frac 12}$ for any function $u$.  We define the discrete Fourier space
\begin{equation*}
	X_N:= \text{span}\left\{ e^{i k x} \ \bigg| \ x \in \Sigma_N,\ 0 \leq |k| \leq N-1 \right\},
\end{equation*}
and set {$W_N= U_N=X_N$}.

Let $(p, n, \mathbf{u})$ be the exact solution of the system \eqref{gov_equ_1}-\eqref{gov_equ_5} with initial condition \eqref{ic}. Denote $(p^{m}, n^{m}, \mathbf{u}^{m}, \phi^{m})$ as the $L^2$-orthogonal projections of $(p, n, \mathbf{u}, \phi)$ at time $m\Delta t$ onto $X_N \times X_N \times X_N^2 \times X_N$, i.e.,
	\begin{equation*}
		\label{lem:ea_var_def}
		\begin{split}
			& p^{m} = \projN p(m\Delta t), \quad n^{m} = \projN n(m\Delta t), \\
			& \mathbf{u}^{m} = \projN \mathbf{u}(m \Delta t), \quad \phi^{m} = \projN \phi(m\Delta t),
		\end{split}
	\end{equation*}
and set 
\begin{equation*}
\label{lem:ea_var_def:2}
	\psi^{m} = \projN \left[ (-\Delta)^{-1}(p^{m}-n^{m}) \right], \quad \mu^{m} = \projN \left[ \ln{p^{m}} + \psi^{m} \right], \quad \nu^{m} = \projN \left[ \ln{n^{m}}-\psi^{m} \right].
\end{equation*}

In order to establish the error analysis for the pressure correction scheme of the Navier-Stokes equations \eqref{scheme:u1}-\eqref{scheme:u3}, we need to introduce an intermediate function $R_N \mathbf{u}^{m+1} \in X_N^2$, defined by
\begin{equation*}
	\label{lem:ea_RN_def}
	\left\langle \dfrac{\mathbf{u}^{m+1}-R_N \mathbf{u}^{m+1}}{\Delta t}, v_N \right\rangle + \left\langle \nabla(\phi^{m+1}-\phi^{m}), v_N \right\rangle = 0, \quad \forall v_N \in X_N^2.
\end{equation*}  
We define the error functions by 
\begin{equation*}\label{err_term_def}
	\begin{split}
		& e_p^{m} = p^{m}-p^{m}_N, \quad e_n^{m} = n^{m}-n^{m}_N, \quad e_\psi^{m} = \psi^{m}-\psi^{m}_{N},\\
		& e_{\tilde{\mathbf{u}}}^{m} = R_N \mathbf{u}^{m}-\tilde{\mathbf{u}}_N^{m}, \quad e_{\mathbf{u}}^{m} = \mathbf{u}^{m}-\mathbf{u}^{m}_N, \quad e_\phi^{m} = \phi^{m}-\phi^{m}_N.
	\end{split}
\end{equation*}

The main result of this section is
\begin{theorem}\label{thm:main}
    Assume the initial data $(p^{\text{in}}, n^{\text{in}}, \mathbf{u}^{\text{in}}) \in H^{k+7}(\Omega) \times H^{k+7}(\Omega) \times H^{k+7}(\Omega)$, for some $k \geq 2$, and $p^{\text{in}}, n^{\text{in}} \geq \delta_0$ for some $\delta_0 > 0$. Then, provided $\Delta t$ and $\dfrac{1}{N}$ are sufficiently small, under the refinement requirement $\Delta t \leq C \dfrac{1}{N}$, we have the following error estimate for the scheme \eqref{scheme:p1}-\eqref{scheme:u3}:
    \begin{equation*}
		\begin{split}
			& \| e_p^{m} \| + \| e_n^{m} \| + \| e_{\mathbf{u}}^{m} \| + \Delta t \| \nabla e_\phi^{m} \|   \\
			& \quad + \left( \Delta t \sum_{l=1}^{m} \left( \| \nabla e_p^{l} \| ^2 + \| \nabla e_n^{l} \| ^2 + \| \nabla e_{\tilde{\mathbf{u}}}^{l} \| ^2 \right) \right)^{\frac{1}{2}} \leq C \left( \Delta t + N^{-k} \right),
		\end{split}
	\end{equation*}
	for all positive integers $m$ such that $m\Delta t \le T$, where $C$ is independent of $\Delta t$ and $N$.
\end{theorem}

To prove this theorem, it is vital to establish a uniform strictly positive lower bound for the numerical solution $(p^{m+1}_N, n^{m+1}_N)$, analogous to the strictly positive lower bound property of continuous solutions $(p,n)$ described in Corollary \ref{maximal_principle}. Recall that we established upper and lower bounds for $(p^{m+1}_N, n^{m+1}_N)$ in Theorem \ref{thm:four_property}; however, the lower bound implied in \eqref{contradiction} depends on the norms of previous step solutions, and is insufficient to establish a uniform strictly positive lower bound for $(p^{m+1}_N, n^{m+1}_N)$ for arbitrary $m$. To overcome this difficulty, we use an approach similar to \cite{liu2021positivity}. In Section 4.1, by assuming sufficient regularity of the PNP-NS system solution, we establish the procedure of building supplementary fields with high-order local truncation errors through Lemma \ref{lem:hoa}. With Lemma \ref{lem:pre}, we perform a rough error analysis that gives the minimum order required of the error terms to establish the lower bound for the numerical solution $(p^{m+1}_N, n^{m+1}_N)$. In Section 4.3, with Theorem \ref{thm:main_sub}, we conduct a refined error analysis, recover the assumption in Lemma \ref{lem:pre}, and prove the error estimates for the supplementary fields built in Lemma \ref{lem:hoa}. Thus, Theorem \ref{thm:main} is a direct combination of Theorem \ref{regularity} and Theorem \ref{thm:main_sub}; the proof will be presented in Section 4.3.


\subsection{A rough error analysis}	
Assume that the solution of PNP-NS system is smooth enough. Then applying Tylor expansion to the system,  one obtains
	\begin{align*} 
		& \langle \frac{p^{m+1}-p^{m}}{\Delta t}, v_N \rangle-\langle p^{m} \uu^{m}, \nabla v_N \rangle + \langle p^{m}(1 + 2\Delta t p^{m}) \nabla \mu^{m+1}, \nabla v_N \rangle = \tau_p^{m+1}( v_N ), \forall v_N \in X_N,
			\\
		& \langle \frac{n^{m+1}-n^{m}}{\Delta t}, v_N \rangle-\langle n^{m} \uu^{m}, \nabla v_N \rangle + \langle n^{m}(1 + 2\Delta t n^{m}) \nabla \nu^{m+1}, \nabla v_N \rangle = \tau_n^{m+1}(v_N),   \forall v_N \in X_N,
			\\
		& \langle \nabla \psi^{m+1}, \nabla v_N \rangle-\langle p^{m+1}-n^{m+1}, v_N \rangle = 0, \quad \forall v_N \in X_N,
		 \\
		& \langle \frac{R_N\uu^{m+1}-\uu^{m}}{\Delta t}, v_N \rangle +  \langle (\uu^{m} \cdot \nabla)R_N\uu^{m+1},  v_N  \rangle + \langle \nabla R_N\uu^{m+1}, \nabla v_N \rangle + \langle \nabla \phi^{m}, v_N \rangle \nonumber \\
		& \qquad + \langle p^{m} \nabla \mu^{m+1} + n^{m} \nabla \nu^{m+1}, v_N \rangle = \tau^{m+1}_{\uu}(v_N), \quad \forall v_N \in X_N^2,
			\\
		& \langle \frac{\uu^{m+1}-R_N \uu^{m+1}}{\Delta t}, v_N \rangle + \langle \nabla(\phi^{m+1}-\phi^{m}), v_N \rangle = 0,   \quad \forall v_N \in { X_N^2},\\
		& \langle \uu^{m+1}, \nabla v_N  \rangle = 0,  \quad \forall v_N \in X_N,
	\end{align*}
	we have the following  local truncation error {(see more computation details in Appendix \eqref{lem:hoa_o_p} - \eqref{lem:hoa_o_u})}:
	\begin{equation*}
		 \abs{\tau_p^{m+1}(v_N)} , \abs{ \tau_n^{m+1}(v_N)} , \abs{\tau_{\uu}^{m+1}(v_N)} \leq C_{k}(\Delta t + N^{- k})\| v_N \|_{H^{1}},
	\end{equation*}
	where $C_k$ depends only on regularity of $(p, n, \psi, \uu, \phi)$.

{\bf High-Order Consistent analysis.}
{
As stated above, we only have a first-order truncation error in time for $n^{m+1}$ and $p^{m+1}$, which is insufficient to establish a priori strictly positive lower bound for the numerical solution $(p^{m+1}_N,n^{m+1}_N)$.
Using the technique similar to \cite{liu2021positivity}, we will construct the supplementary fields $(\up, \un, \uuu, \uphi, \umu, \unu, \upsi)$,
providing sufficient regularity for the solution $(p, n, \uu, \phi)$, a higher order $\bigO(\Delta t^3 + N^{-k})$ consistency local truncation error will be established.
}
\begin{lemma}
    \label{lem:hoa}
	Let $(p, n , \uu)$ be the solution of the PNP-NS system \eqref{gov_equ_1}-\eqref{gov_equ_5} satisfying the following properties:
	\begin{enumerate}
		\item The ionic concentrations are strictly positive
		        \begin{equation*}
		        \label{hoa:pos}
		            p, n \geq \delta_0 > 0,
                \end{equation*}
		\item The solution satisfies
		      \begin{equation*}
		      \label{hoa:reg}
		        \begin{array}{ll}
		            (\pt^{4} p, \pt^{4} n, \pt^{4}\uu) \in L^{\infty}(0, T; L^{2}(\Omega)), (\pt^{3} p, \pt^{3} n, \pt^{3} \uu) \in L^{\infty}(0, T; H^{k+1}(\Omega)),\ (k \ge 2),
		        \end{array}
		      \end{equation*}
	\end{enumerate}
	 we can construct { correction functions $(p_{\Delta t, i}, n_{\Delta t, i}, \uu_{\Delta t, i}, \phi_{\Delta t, i})(i = 1,2)$ depending only on $(p, n, \uu, \phi)$}, such that 
{
the 
supplementary fields $(\up, \un, \uuu, \uphi, \umu, \unu, \upsi)$, defined by 
\begin{equation}
	\label{lem:hoa_vardef}
	\begin{array}{l}
		\up = p + \Delta t p_{\Delta t, 1} + \Delta t^2 p_{\Delta t,2}, \ \un = n + \Delta t n_{\Delta t, 1} + \Delta t^2 n_{\Delta t,2},  \\
		\uuu = \uu + \Delta t \uu_{\Delta t, 1} + \Delta t^2 \uu_{\Delta t, 2}, \ \uphi = \phi + \Delta t \phi_{\Delta t, 1} + \Delta t^2 \phi_{\Delta t,2},  \\
		{
		\upsi = (-\Delta)^{-1}(\up-\un),} \\
		{
		\umu = \ln{\up} + \upsi, \ \unu = \ln{\un}-\upsi,
		}
	\end{array}
\end{equation}			
	}
	has higher order consistency truncation error as defined in \eqref{hoa:p1}-\eqref{hoa:u2}
	\begin{equation*}
		\rvert \utau^{m+1}_{p}(v_N)\rvert , \rvert \utau^{m+1}_{n}(v_N)\rvert , \rvert \utau^{m+1}_{\uu}(v_N)\rvert  \leq C((\Delta t)^3 + N^{-k}) \|  v_N \| _{H^{1}}. 
	\end{equation*}
	Moreover, with $\Delta t, \frac{1}{N}$ chosen small enough, we have 
	\begin{enumerate}
		\item The 
		{
		supplementary 
		}
		 functions are strictly positive:
		    \begin{equation}
		    \label{hoa:mod_pos}
		        \up, \un \ge \delta_0^{*} > 0,
		    \end{equation}
		\item The 
		{
		supplementary 
		}
		functions satisfy 
		    \begin{equation}
		        \label{hoa:mod_reg}
		        (\up, \un, \uuu) \in L^{\infty}(0, T, W^{1, \infty}).
		    \end{equation}
	\end{enumerate}
\end{lemma}

\smallskip
The detail of constructing $(\up, \un, \uuu, \uphi, \umu, \unu, \upsi)$ in Lemma \ref{lem:hoa} will be given in the Appendix.

Now we start to make an error analysis for the scheme \eqref{scheme:p1}-\eqref{scheme:u3} by analyzing its truncation error
{ for supplementary fields $(\up, \un, \uuu, \uphi, \umu, \unu, \upsi)$ }
 Denote the error functions by
\begin{equation}\label{err_term_mod_def}
	\begin{array}{l}
		\uep^{m} = \up^{m}-p^{m}_N, \ \uen^{m} = \un^{m}-n^{m}_N, \ \uemu^{m} = \umu^{m}-\mu^{m}_N, \ \uenu^{m} = \unu^{m}-\nu^{m}_N, \\
		\uepsi^{m} = \upsi^{m}-\psi^{m}_{N}, \ \ueru^{m} = R_N\uuu^{m}-\tilde{\uu}_N^{m}, \ \ueu^{m} = \uuu^{m}-\uu^{m}_N, \ \uephi^{m} = \uphi^{m}-\phi^{m}_N.
	\end{array}
\end{equation}

 Denote by $(\up^{m}, \un^{m}, \uuu^{m}, \uphi^{m})$ the $L^{2}$-orthogonal projection of $(\up, \un, \uuu, \uphi)$ at time $m\Delta t$ onto $X_N \times X_N \times X_N^2 \times X_N$. We have the expression for the consistency truncation error $(\utau_{p}, \utau_{n}, \utau_{\uu})$ for the modified functions:
\begin{align}
	& \langle \frac{\up^{m+1}-\up^{m}}{\Delta t}, v_N \rangle-\langle \up^{m} \uuu^{m}, \nabla v_N \rangle + \langle \up^{m}(1 + 2\Delta t \up^{m}) \nabla \umu^{m+1}, \nabla v_N \rangle 
		= \utau^{m+1}_{p}(v_N),
		\label{hoa:p1} \\
	& \langle \frac{\un^{m+1}-\un^{m}}{\Delta t}, v_N \rangle-\langle \un^{m} \uuu^{m}, \nabla v_N \rangle + \langle \un^{m}(1 + 2\Delta t \un^{m}) \nabla \unu^{m+1}, \nabla v_N \rangle 
		= \utau^{m+1}_{n}(v_N),
		\label{hoa:n1} \\
	& \langle \nabla \upsi^{m+1}, \nabla v_N \rangle = \langle \up^{m+1}-\un^{m+1}, v_N \rangle, \label{hoa:psi1} \\
	& \langle \frac{R_N\uuu^{m+1}-\uuu^{m}}{\Delta t}, v_N \rangle +  \langle (\uuu^{m} \cdot \nabla)R_N\uuu^{m+1}, v_N \rangle + \langle \nabla R_N\uuu^{m+1}, \nabla v_N \rangle + \langle \nabla \uphi^{m}, v_N \rangle \nonumber \\
	& \quad + \langle \up^{m} \nabla \umu^{m+1} + \un^{m} \nabla \unu^{m+1}, v_N \rangle = \utau^{m+1}_{\uu}(v_N),
		\label{hoa:u1} \\
	& \langle \frac{\uuu^{m+1}-R_N \uuu^{m+1}}{\Delta t}, v_N \rangle + \langle \nabla(\uphi^{m+1} -\uphi^{m}), v_N \rangle = 0, \label{hoa:u2}
\end{align}
where 
\begin{equation*}
	\umu^{m+1} = \projN (\ln{\up^{m+1}} + \upsi^{m+1}); \ \unu^{m+1} = \projN(\ln{\un^{m+1}}-\upsi^{m+1}).
\end{equation*}

Subtracting \eqref{scheme:p1}-\eqref{scheme:u2} from \eqref{hoa:p1}-\eqref{hoa:u2}, we have
\begin{align}
	& \langle \frac{\uep^{m+1}-\uep^{m}}{\Delta t}, v_N \rangle 
		- \langle \up^{m}\uuu^{m}-p^{m}_N \uu^{m}_N , \nabla v_N \rangle 
		\nonumber \\
	& \qquad = -\langle \up^{m}(1 + 2\Delta t \up^{m}) \nabla \umu^{m+1}-p^{m}_N (1 + 2\Delta t p^{m}_N) \nabla \mu^{m+1}_N, \nabla v_N \rangle + \utau^{m+1}_p(v_N) \vspace{0.8em} \label{err:preeq1}, \\
	& \langle \frac{\uen^{m+1}-\uen^{m}}{\Delta t}, v_N \rangle-\langle \un^{m}\uuu^{m}-n^{m}_N \uu^{m}_N , \nabla v_N \rangle \nonumber \\
	& \qquad =-\langle \un^{m}(1 + 2\Delta t \un^{m}) \nabla \unu^{m+1}-n^{m}_N (1 + 2\Delta t n^{m}_N) \nabla \nu^{m+1}_N, \nabla v_N \rangle +  \utau^{m+1}_n(v_N) \vspace{0.8em} \label{err:preeq3}, \\
	& \langle \nabla \uepsi^{m+1}, \nabla v_N \rangle = \langle \uep^{m+1}-\uen^{m+1}, v_N \rangle \label{err:preeq5}, \\
	& \langle \frac{\ueru^{m+1}-\ueu^{m}}{\Delta t} , v_N \rangle + \langle (\uuu^{m} \cdot \nabla) R_N\uuu^{m+1}-(\uu^{m}_N \cdot \nabla) \tilde{\uu}^{m+1}_N, v_N \rangle + \langle \nabla \ueru^{m+1}, \nabla v_N \rangle + \langle \nabla \uephi^{m}, v_N \rangle \nonumber \\ 
	& \qquad = -\langle \up^{m}\nabla \umu^{m+1} + \un^{m} \nabla \unu^{m+1}, v_N \rangle + \langle p^{m}_N \nabla \mu^{m+1}_N + n^{m}_N \nabla \nu^{m+1}_N , v_N \rangle + \utau^{m+1}_{\uu} (v_N) \label{err:preeq6}, \\
	& \langle \frac{\ueu^{m+1}-\ueru^{m+1}}{\Delta t}, v_N \rangle + \langle \nabla(\uephi^{m+1}-\uephi^{m}), v_N \rangle = 0,   \\
	& \langle \ueu^{m+1}, \nabla v_N \rangle = 0 \label{err:preeq7}.
\end{align}
To simplify the presentation,  we rewrite the third term in \eqref{err:preeq1} as
\begin{equation*}
	\begin{split}
		& \quad -\langle \up^{m}(1 + 2\Delta t \up^{m}) \nabla \umu^{m+1}-p^{m}_N (1 + 2\Delta t p^{m}_N) \nabla \mu^{m+1}_N, \nabla v_N \rangle \\
		& = -\langle \up^{m}(1 + 2\Delta t \up^{m}) \nabla \umu^{m+1}-p^{m}_N (1 + 2\Delta t p^{m}_N) \nabla \umu^{m+1}, \nabla v_N \rangle \\
		& \quad -\langle p^{m}_N (1 + 2\Delta t p^{m}_N) \nabla \umu^{m+1}-p^{m}_N (1 + 2\Delta t p^{m}_N) \nabla \umu^{m+1}_N, \nabla v_N \rangle  \\ 
		& = -\langle \uep^{m}(1 + 2\Delta t (\up^{m} + p^{m}_N)) \nabla \umu^{m+1}, \nabla v_N \rangle 
			- \langle p^{m}_N (1 + 2\Delta t p^{m}_N) \nabla \uemu^{m+1}, \nabla v_N \rangle.
	\end{split} 
\end{equation*}
Rewrite the second term of \eqref{err:preeq1} into
\begin{equation*}
	\begin{split}
		& \quad-\langle \up^{m}\uuu^{m}-p^{m}_N \uu^{m}_N , \nabla v_N \rangle \\
		& = -\langle \up^{m} \uuu^{m}-p^{m}_N \uuu^{m}, \nabla v_N \rangle 
		   -\langle p^{m}_N \uuu^{m}-p^{m}_N \uu^{m}_N, \nabla v_N \rangle \\
		& = -\langle \uep^{m} \uuu^{m}, \nabla v_N \rangle 
			- \langle p^{m}_N \ueu^{m}, \nabla v_N \rangle.
	\end{split}
\end{equation*}
Similarly, for the third and second term of  \eqref{err:preeq3}, we have
\begin{equation*}
	\begin{split}
		& \quad -\langle \un^{m}(1 + 2\Delta t \un^{m}) \nabla \unu^{m+1}-n^{m}_N (1 + 2\Delta t n^{m}_N) \nabla \nu^{m+1}_N, \nabla v_N \rangle \\
		& = -\langle \uen^{m}(1 + 2\Delta t (\un^{m} + n^{m}_N)) \nabla \unu^{m+1}, \nabla v_N \rangle 
			- \langle n^{m}_N (1 + 2\Delta t n^{m}_N) \nabla \uenu^{m+1}, \nabla v_N \rangle,
	\end{split}
\end{equation*}
and
\begin{equation*}
	\begin{split}
		& \quad-\langle \un^{m}\uuu^{m}-n^{m}_N \uu^{m}_N , \nabla v_N \rangle \\
		& = -\langle \uen^{m} \uuu^{m}, \nabla v_N \rangle 
			- \langle n^{m}_N \ueu^{m}, \nabla v_N \rangle.
	\end{split}
\end{equation*}
For the Navier-Stokes equation, in \eqref{err:preeq6}, we have
\begin{equation*}
	\begin{split}
		& \quad \langle (\uuu^{m} \cdot \nabla) R_N\uuu^{m+1}-(\uu^{m}_N \cdot \nabla) \tilde{\uu}^{m+1}_N, v_N \rangle  \\
		& = \langle (\uuu^{m} \cdot \nabla) R_N\uuu^{m+1}-(\uu^{m}_N \cdot \nabla) R_N\uuu^{m+1}, v_N \rangle \\
		& \quad + \langle (\uu^{m}_N \cdot \nabla) R_N\uuu^{m+1}-(\uu^{m}_N \cdot \nabla) \tilde{\uu}^{m+1}_N, v_N \rangle \\
		& = \langle (\ueu^{n} \cdot \nabla) R_N\uuu^{m+1}, v_N \rangle 
			+ \langle (\uu^{m}_N \cdot \nabla) \ueru^{m+1}, v_N \rangle,
	\end{split}
\end{equation*}
and 
\begin{equation*}
	\begin{split}
		& \quad -\langle \up^{m}\nabla \umu^{m+1} + \un^{m} \nabla \unu^{m+1}, v_N \rangle 
			+ \langle p^{m}_N \nabla \mu^{m+1}_N + n^{m}_N \nabla \nu^{m+1}_N , v_N \rangle \\
		& =-\langle \up^{m}\nabla \umu^{m+1} + \un^{m} \nabla \unu^{m+1}, v_N \rangle
			+ \langle p^{m}_N \nabla \umu^{m+1} + n^{m}_N \nabla \unu^{m+1}, v_N \rangle \\
		& \quad-\langle p^{m}_N \nabla \umu^{m+1} + n^{m}_N \nabla \unu^{m+1}, v_N \rangle
				+ \langle p^{m}_N \nabla \mu^{m+1}_N + n^{m}_N \nabla \nu^{m+1}_N , v_N \rangle \\
		& = -\langle \uep^{m} \nabla \umu^{m+1} + \uen^{m} \nabla \unu^{m+1}, v_N \rangle 
			- \langle p^{m}_N \nabla \uemu^{m+1} + n^{m}_N \nabla \uenu^{m+1}, v_N \rangle.
	\end{split}
\end{equation*}
Collecting all previous equations, the error equations \eqref{err:eq1}-\eqref{err:eq8} could be rewritten as 
\begin{align}
    \left\langle \frac{\uep^{m+1}-\uep^{m}}{\Delta t}, v_N \right\rangle 
    & - \left\langle \uep^{m} \uuu^{m}, \nabla v_N \right\rangle 
    - \left\langle p^{m}_N \ueu^{m}, \nabla v_N \right\rangle \nonumber \\
    & = -\left\langle \uep^{m}\left(1 + 2\Delta t (\up^{m} + p^{m}_N)\right) \nabla \umu^{m+1}, \nabla v_N \right\rangle \nonumber \\
    & \quad - \left\langle p^{m}_N \left(1 + 2\Delta t\, p^{m}_N\right) \nabla \uemu^{m+1}, \nabla v_N \right\rangle + \utau^{m+1}_p(v_N),
    \label{err:eq1} \\[1ex]
    \left\langle \frac{\uen^{m+1}-\uen^{m}}{\Delta t}, v_N \right\rangle 
    & - \left\langle \uen^{m} \uuu^{m}, \nabla v_N \right\rangle 
    - \left\langle n^{m}_N \ueu^{m}, \nabla v_N \right\rangle \nonumber \\
    & = -\left\langle \uen^{m}\left(1 + 2\Delta t (\un^{m} + n^{m}_N)\right) \nabla \unu^{m+1}, \nabla v_N \right\rangle \nonumber \\
    & \quad - \left\langle n^{m}_N \left(1 + 2\Delta t\, n^{m}_N\right) \nabla \uenu^{m+1}, \nabla v_N \right\rangle + \utau^{m+1}_n(v_N),
    \label{err:eq3} \\[1ex]
    \left\langle \nabla \uepsi^{m+1}, \nabla v_N \right\rangle 
    & = \left\langle \uep^{m+1}-\uen^{m+1}, v_N \right\rangle,
    \label{err:eq5} \\[1ex]
    \left\langle \frac{\ueru^{m+1}-\ueu^{m}}{\Delta t}, v_N \right\rangle 
    & + \left\langle (\ueu^{m} \cdot \nabla) R_N\uuu^{m+1}, v_N \right\rangle 
    + \left\langle (\uu^{m}_N \cdot \nabla) \ueru^{m+1}, v_N \right\rangle \nonumber \\
    & + \left\langle \nabla \ueru^{m+1}, \nabla v_N \right\rangle 
    + \left\langle \nabla \uephi^{m}, v_N \right\rangle \nonumber \\ 
    & = -\left\langle \uep^{m} \nabla \umu^{m+1} + \uen^{m} \nabla \unu^{m+1}, v_N \right\rangle \nonumber \\
    & \quad - \left\langle p^{m}_N \nabla \uemu^{m+1} + n^{m}_N \nabla \uenu^{m+1}, v_N \right\rangle 
    + \utau^{m+1}_{\uu} (v_N),
    \label{err:eq6} \\[1ex]
    \left\langle \frac{\ueu^{m+1}-\ueru^{m+1}}{\Delta t}, v_N \right\rangle 
    & + \left\langle \nabla(\uephi^{m+1}-\uephi^{m}), v_N \right\rangle = 0,
    \label{err:eq7} \\[1ex]
    \left\langle \ueu^{m+1}, \nabla v_N \right\rangle & = 0.
    \label{err:eq8}
\end{align}

To finish the error analysis, we will need Lemma \ref{lem:pre} below. 
\begin{lemma}
\label{lem:pre}
    Under the same assumption  and procedure as in Lemma \ref{lem:hoa}{, we build supplementary fields $(\up, \un,  \uuu, \uphi)$,  for the numerical error defined in \eqref{err_term_mod_def},}  assume that for { $2 < \alpha < 3 \text{, } 2 < \beta < k $} the error estimate holds for the $m$-th step, i.e. 
    \begin{align}
		\|  \uep^{m} \| _2 & \leq \Delta t^{{\alpha}} + (\frac{1}{N})^{\beta}, \label{thm:res1} \\
		\|  \uen^{m} \| _2 & \leq \Delta t^{{ \alpha}} + (\frac{1}{N})^{\beta}, \label{thm:res2} \\
		\|  \ueu^{m} \| _2 & \leq \Delta t^{{\alpha}} + (\frac{1}{N})^{ \beta}, \label{thm:res3} 
    \end{align}
    under the linear refinement requirement $\Delta t \leq C\frac{1}{N}$, we have the following $L^{\infty}$-estimate for the $(m+1)$-th step, i.e.
	\begin{align*}
		& \|  \uep^{m+1} \| _{\infty} \leq C \Big( \Delta t^{{\alpha - 2}} + (\frac{1}{N})^{{\beta - 2}} \Big), \\
		& \|  \uen^{m+1} \| _{\infty} \leq C \Big( \Delta t^{{\alpha - 2}} + (\frac{1}{N})^{{\beta - 2}} \Big),
	\end{align*}
	where $C$ is independent  of $\Delta t, N$,  { and }
\end{lemma}
\begin{proof}
	First, from Lemma \ref{lem:hoa}, we can construct $(\up, \un, \uuu)$ that satisfies \eqref{hoa:mod_pos} \eqref{hoa:mod_reg}. 
To obtain the bound of $p^{m}_N, n^{m}_N, \|  \nabla p^{m}_N \| _{\infty}, \| \nabla n^{m}_N \|_{\infty}$, given the $a$ $priori$ estimate \eqref{thm:res1}, a direct application of inverse inequalities implies
	\begin{align*}
		& \|  \uep^{m} \| _{\infty} \leq C N \|  \uep^{m} \| _{2} \leq C \Big(\Delta t^{ { \alpha - 1} } + (\frac{1}{N})^{{ \beta - 1}}\Big), \\
		& \|  \nabla \uep^{m} \| _{\infty} \leq C N \|  \uep^m \| _{\infty} \leq C\Big(\Delta t^{{ \alpha - 2}} + (\frac{1}{N})^{{ \beta - 2}} \Big),
	\end{align*}
	where we used $\Delta t \leq C \frac{1}{N}$. Similarly, we have 
	\begin{align*}
		& \|  \uen^{m} \| _{\infty} \leq  C \Big(\Delta t^{{ \alpha - 1}} + (\frac{1}{N})^{{ \beta - 1}}\Big), \\
		& \|  \nabla \uen^{m} \| _{\infty} \leq  C\Big(\Delta t^{{ \alpha - 2}} + (\frac{1}{N})^{{ \beta - 2}} \Big).
	\end{align*}
	Provided $\Delta t, \frac{1}{N}$ are sufficiently small, we have 
        \begin{equation}\label{bound:eq:1}
        \left\{
        \begin{aligned}
            & \| \uep^{m} \|_{\infty},\ \| \uen^{m} \|_{\infty} \leq \dfrac{\delta_0^{*}}{2}, \\
            & \| \nabla \uep^{m} \|_{\infty},\ \| \nabla \uen^{m} \|_{\infty} \leq \dfrac{\delta_0^{*}}{2}.
        \end{aligned}
        \right.
        \end{equation}
	where $\delta_0^{*} > 0$ is a small constant. 
	
	Combining \eqref{bound:eq:1} with the regularity of $(\up, \un)$ as in \eqref{hoa:mod_reg}, we obtain bounds for $p^{m}_N, n^{m}_N, \nabla p^{m}_N, \nabla n^{m}_N$:
	\begin{align}
		& \frac{\delta_0^{*}}{2} \leq \min \up^{m}-\|  \uep^{m} \| _{\infty} \leq p^{m}_N \leq \|  \up^{m} \| _{\infty} + \|  \uep^{m} \| _{\infty} \leq M + \frac{\delta_0^{*}}{2} \label{lem:pre_bdd1_p}, \\
		& \frac{\delta_0^{*}}{2} \leq \min \un^{m}-\|  \uen^{m} \| _{\infty} \leq n^{m}_N \leq \|  \un^{m} \| _{\infty} + \|  \uen^{m} \| _{\infty} \leq M + \frac{\delta_0^{*}}{2}  \label{lem:pre_bdd1_n}, \\
		& \|  \nabla p^{m}_N \| _{\infty} \leq \|  \nabla \up^{m} \| _{\infty} + \|  \nabla \uep^{m} \| _{\infty}  \leq M + \frac{\delta_0^{*}}{2} \label{lem:pre_bdd2_p}, \\ 
		& \|  \nabla n^{m}_N \| _{\infty} \leq \|  \nabla \un^{m} \| _{\infty} + \|  \nabla \uen^{m} \| _{\infty}  \leq M + \frac{\delta_0^{*}}{2} \label{lem:pre_bdd2_n}.
	\end{align}
	
    Taking $v_N = \uemu^{m+1}$ in \eqref{err:eq1}, using the equality $\uemu^{m+1} = \ln \up^{m+1}-\ln p^{m+1}_N + \uepsi^{m+1}$, we obtain the left hand side of \eqref{err:eq1}: 
	\begin{equation}
	\label{lem:pre_lhs_ep1}
		\begin{split}
			LHS_p
			& = \frac{1}{\Delta t} \langle \uep^{m+1}, \ln{\up^{m+1}}-\ln{p^{m+1}_N} \rangle  
				+ \frac{1}{\Delta t} \langle \uep^{m+1} , \uepsi^{m+1} \rangle  \\
			& \quad	- \frac{1}{\Delta t} \langle \uep^{m}, \uemu^{m+1} \rangle
			   -\langle \uep^{m} \uuu^{m} + p^{m}_N \ueu^{m}, \nabla \uemu^{m+1} \rangle,
		\end{split}
	\end{equation}
	and the right hand side of \eqref{err:eq1}:
	\begin{equation}
	\label{lem:pre_rhs_ep1}
	    \begin{split}
	        RHS_p &=-\langle \uep^{m}\left(1  + 2\Delta t(\up^{m} + p^{m}_N)\right) \nabla \umu^{m+1}, \nabla \uemu^{n+1} \rangle \\
			& \quad-\int_{\Omega} p^{m}_N (1 + 2\Delta t p^{m}_N) \rvert  \nabla \uemu^{m+1} \rvert ^2 dx + \utau^{m+1}_p(\uemu^{m+1}).
	    \end{split}
	\end{equation}
	Similarly taking $v_N = \uenu^{m+1}$ in \eqref{err:eq3}, we obtain  
	\begin{equation}
	\label{lem:pre_lhs_en1}
		\begin{split}
			LHS_n
			& = \frac{1}{\Delta t} \langle \uen^{m+1}, \ln{\un^{m+1}}-\ln{n^{m+1}_N} \rangle 
				- \frac{1}{\Delta t} \langle \uen^{m+1} , \uepsi^{m+1} \rangle  \\
			& \quad-\frac{1}{\Delta t} \langle \uen^{m}, \uenu^{m+1} \rangle
			   - \langle \uen^{m} \uuu^{m} + n^{m}_N \ueu^{m}, \nabla \uenu^{m+1} \rangle,
		\end{split}
	\end{equation}
	and
	\begin{equation}
	\label{lem:pre_rhs_en1}
	    \begin{split}
	        RHS_n 
    		& =-\langle \uen^{m}\left(1  + 2\Delta t(\un^{m} + n^{m}_N)\right) \nabla \unu^{m+1}, \nabla \uenu^{n+1} \rangle  \\
    		& \quad-\int_{\Omega} n^{m}_N (1 + 2\Delta t n^{m}_N) \rvert  \nabla \uenu^{m+1} \rvert ^2 dx + \utau^{m+1}_n(\uenu^{m+1}).
	    \end{split}
	\end{equation}
	From the monotonicity of $\ln{x}$ for $x > 0$, we obtain that
	\begin{align}
		& \langle \uep^{m+1}, \ln{\up^{m+1}}-\ln{p^{m+1}_N} \rangle = \langle \up^{m+1}-p^{m+1}_N, \ln{\up^{m+1}}-\ln{p^{m+1}_N} \rangle \geq 0 \label{lem:pre_lhs_lnp}, \\
		& \langle \uen^{m+1}, \ln{\un^{m+1}}-\ln{n^{m+1}_N} \rangle = \langle \un^{m+1}-n^{m+1}_N, \ln{\un^{m+1}}-\ln{n^{m+1}_N} \rangle \geq 0.
			\label{lem:pre_lhs_lnn}
	\end{align}
From \eqref{err:eq5}, we have 
	\begin{equation}
	\label{lem:pre_lhs_psi}
		\langle \uep^{m+1}-\uen^{m+1}, \uepsi^{m+1} \rangle = \|  \nabla \uepsi^{m+1} \| ^2 \geq 0.
	\end{equation}
	Combining \eqref{lem:pre_lhs_ep1}, \eqref{lem:pre_lhs_en1}, \eqref{lem:pre_lhs_lnp}, \eqref{lem:pre_lhs_lnn} with \eqref{lem:pre_lhs_psi}, we have 
	\begin{equation}
	\label{lem:pre_lhs_res}
		\begin{split}
			& LHS_p + LHS_n  \\
				& \geq -\frac{1}{\Delta t}(\langle \uep^{m}, \uemu^{m+1} \rangle + \langle \uen^{m}, \uenu^{m+1} \rangle ) 
				 -\langle \uep^{m} \uuu^{m} + p^{m}_N \ueu^{m}, \nabla \uemu^{m+1} \rangle
				 - \langle \uen^{m} \uuu^{m} + n^{m}_N \ueu^{m}, \nabla \uenu^{m+1} \rangle.
		\end{split}
	\end{equation}
	Summing up \eqref{lem:pre_rhs_ep1} and \eqref{lem:pre_rhs_en1}  and using \eqref{lem:pre_lhs_res}, we have 
	\begin{equation}
	\label{lem:pre_com}
		\begin{split}
			& \quad \int_{\Omega} p^{m}_N(1 + 2\Delta t p^{m}_N) \rvert  \nabla \uemu^{m+1} \rvert ^2 
			  	  + n^{m}_N(1 + 2\Delta t n^{m}_N) \rvert  \nabla \uenu^{m+1} \rvert ^2 dx \\
			& \leq -\langle \uep^{m}\left(1  + 2\Delta t(\up^{m} + p^{m}_N)\right) \nabla \umu^{m+1}, \nabla \uemu^{m+1} \rangle
					- \langle \uen^{m}\left(1  + 2\Delta t(\un^{m} + n^{m}_N)
				\right) \nabla \unu^{m+1}, \nabla \uenu^{m+1} \rangle \\
			& \quad + \frac{1}{\Delta t}(\langle \uep^{m}, \uemu^{m+1} \rangle + \langle \uen^{m}, \uenu^{m+1} \rangle ) 
				  + \langle \uep^{m} \uuu^{m} + p^{m}_N \ueu^{m}, \nabla \uemu^{m+1} \rangle
				  + \langle \uen^{m} \uuu^{m} + n^{m}_N \ueu^{m}, \nabla \uenu^{m+1} \rangle \\
			& \quad + \utau^{m+1}_p(\uemu^{m+1}) + \utau^{m+1}_n(\uenu^{m+1}).
		\end{split}
	\end{equation}
	Using the $L^{\infty}$ bound of $p^{m}_N, n^{m}_N$ in \eqref{lem:pre_bdd1_p} and \eqref{lem:pre_bdd1_n}, we have
	\begin{equation}
	\label{lem:pre_com_l}
		\int_{\Omega} p^{m}_N(1 + 2\Delta t p^{m}_N) \rvert  \nabla \uemu^{m+1} \rvert ^2 + n^{m}_N(1 + 2\Delta t n^{m}_N) \rvert  \nabla \uenu^{m+1} \rvert ^2 dx
		\ge 
		\frac{\delta_0^{*}}{2}(\|  \nabla \uemu^{m+1} \| ^2 + \| \nabla \uenu^{m+1} \| ^2).
	\end{equation}
	Applying Hölder and Young's inequalities, for the second term in \eqref{lem:pre_com}, we have 
	\begin{equation}
	\label{lem:pre_com_r1}
		\begin{split}
			& \quad-\langle \uep^{m}\left(1  + 2\Delta t(\up^{m} + p^{m}_N)\right) \nabla \umu^{m+1}, \nabla \uemu^{m+1} \rangle \\
			& \leq \|  1 + 2\Delta t(\up^{m} + p^{m}_N) \| _{\infty} \|  \uep^{m} \|  \|  \nabla \umu^{m+1} \| _{\infty} \|  \nabla \uemu^{m+1} \|   \\
			& \leq (4M + \delta_0^{*} + 1)  \|  \uep^{m} \|  \|  \nabla \umu^{m+1} \| _{\infty} \|  \nabla \uemu^{m+1} \|  \\
			& \leq \frac{\delta_0^{*}}{2} \frac{1}{8} \|  \nabla \uemu^{m+1} \| ^2 + \frac{4\tilde{C}}{\delta_0^{*}} \|  \uep^{m} \| ^2,
		\end{split}
	\end{equation}
	and for the third term in \eqref{lem:pre_com}, 
	\begin{equation}
	\label{lem:pre_com_r2}
		\begin{split}
			& \quad-\langle \uen^{m}\left(1  + 2\Delta t(\un^{m} + n^{m}_N)
				\right) \nabla \unu^{m}, \nabla \uenu^{n+1} \rangle \\
			& \leq \|  1 + 2\Delta t(\un^{m} + n^{m}_N) \| _{\infty} \|  \uen^{m} \|  \|  \nabla \unu^{m+1} \| _{\infty} \|  \nabla \uenu^{m+1} \|   \\
			& \leq (4M + \delta_0^{*} + 1) \|  \uen^{m} \|  \|  \nabla \unu^{m+1} \| _{\infty} \|  \nabla \uenu^{m+1} \|  \\
			& \leq \frac{\delta_0^{*}}{2} \frac{1}{8} \|  \nabla \uenu^{m+1} \| ^2 + \frac{4\tilde{C}}{\delta_0^{*}} \|  \uen^{m} \| ^2,
		\end{split}
	\end{equation}
	where $\tilde{C} \geq (4M+\delta_0^{*} + 1)^2 (\|  \nabla \umu^{m+1} \| _{\infty}^2 + \|  \nabla \unu^{m+1} \| _{\infty}^2)$. Note that by \eqref{hoa:mod_reg}, $\| \umu \|_{L^{\infty}_{t}W^{1,\infty}_{x}}$  and $\| \unu \|_{L^{\infty}_{t}W^{1,\infty}_{x}}$  are bounded.
	
	Using Hölder and Young's inequalities, we derive
	\begin{equation}
	\label{lem:pre_com_r3}
		\begin{split}
			& \quad \frac{1}{\Delta t}(\langle \uep^{m}, \uemu^{m+1} \rangle + \langle \uen^{m}, \uenu^{m+1} \rangle ) \\
			& \leq \frac{\delta_0^{*}}{2} \frac{1}{8} (\|  \nabla \uemu^{m+1} \| ^2 + \|  \nabla \uenu^{m+1} \| ^2 ) + \frac{4}{\delta_0^{*}} \frac{1}{\Delta t^2} (\|  \uep^{m} \| ^2_{H^{-1}} + \|  \uen^{m} \| ^2_{H^{-1}}) \\
			& \leq \frac{\delta_0^{*}}{2} \frac{1}{8} (\|  \nabla \uemu^{m+1} \| ^2 + \|  \nabla \uenu^{m+1} \| ^2 ) + \frac{4}{\delta_0^{*}} \frac{C}{\Delta t^2} (\|  \uep^{m} \| ^2 + \|  \uen^{m} \| ^2).
		\end{split}
	\end{equation}
	Using  the  bound  of $\| p^{m}_N \|_{\infty}$ from \eqref{lem:pre_bdd1_p}, and the bound of $\| \uuu^{m} \|_{\infty}$ from \eqref{hoa:mod_reg}, we obtain
	\begin{equation}
	\label{lem:pre_com_r4}
		\begin{split}
			&\quad \langle \uep^{m} \uuu^{m}, \nabla \uemu^{m+1} \rangle
				+ \langle p^{m}_N \ueu^{m}, \nabla \uemu^{m+1} \rangle \\
			& \leq \| \nabla \uemu^{m+1} \|  ( \|  \uuu^{m} \| _{\infty} \|  \uep^{m} \|  + \|  p^{m}_N \| _{\infty} \|  \ueu^{m} \| ) \\
			& \leq \frac{\delta_0^{*}}{2} \frac{1}{8} \|  \nabla \uemu^{m+1} \| ^2 + \frac{4C}{\delta_0^{*}} (\|  \uep^{m} \| ^2 + \|  \ueu^{m} \| ^2  ).
		\end{split}
	\end{equation}
	Similarly,  we obtain
	\begin{equation}
	\label{lem:pre_com_r5}
		\begin{split}
			&\quad \langle \uen^{m} \uuu^{m}, \nabla \uenu^{m+1} \rangle
				+ \langle n^{m}_N \ueu^{m}, \nabla \uenu^{m+1} \rangle \\
			& \leq \| \nabla \uenu^{m+1} \|  ( \|  \uuu^{m} \| _{\infty} \|  \uen^{m} \|  + \|  n^{m}_N \| _{\infty} \|  \ueu^{m} \| ) \\
			& \leq \frac{\delta_0^{*}}{2} \frac{1}{8} \|  \nabla \uenu^{m+1} \| ^2 + \frac{4C}{\delta_0^{*}} (\|  \uen^{m} \| ^2 + \|  \ueu^{m} \| ^2  ).
		\end{split}
	\end{equation}
	From Lemma \ref{lem:hoa}, we have 
	\begin{equation}
	\label{lem:pre_com_r6}
		\begin{split}
			\langle \utau^{m+1}_p, \uemu^{m+1} \rangle + \langle \utau^{m+1}_n, \uenu^{m+1}\rangle 
			& \leq C(\Delta t^3 + N^{-k})(\|  \uemu^{m+1} \| _{H^{1}} + \|  \uenu^{m+1} \| _{H^{1}} ) \\
			& \leq \frac{\delta_0^{*}}{2} \frac{1}{8}(\|  \nabla \uemu^{m+1} \| ^2 + \|  \nabla \uenu^{m+1} \| ^2  ) + \frac{4}{\delta_0^{*}} (C(\Delta t^3 + N^{-k}))^2,
		\end{split} 
	\end{equation}
	where the positive constant $C$ in \eqref{lem:pre_com_r6} is independent of $\Delta t$ and $N$.
	
	 Plugging \eqref{lem:pre_com_l}-\eqref{lem:pre_com_r6} into \eqref{lem:pre_com}, we have 
	\begin{equation}
		\begin{split}\label{lem:result}
			& \quad \frac{\delta_0^{*}}{4}(\|  \nabla \uemu^{m+1} \| ^2 + \|  \nabla \uenu^{m+1} \| ^2 ) \\
			& \leq  \frac{4\tilde{C}}{\delta_0^{*}} ( \|  \uep^{m} \| ^2 + \|  \uen^{m} \| ^2) 
				  + \frac{4}{\delta_0^{*}} \frac{C}{\Delta t^2} (\|  \uep^{m} \| ^2 + \|  \uen^{m} \| ^2) \\
			&	  + \frac{4C}{\delta_0^{*}} (\|  \uep^{m} \| ^2 + \|  \uen^{m} \| ^2 + 2 \|  \ueu^{m} \| ^2  ) 
				  + \frac{4}{\delta_0^{*}} (C(\Delta t^3 + N^{-k}))^2.
		\end{split}
	\end{equation}
	Combing \eqref{lem:result} with assumption \eqref{thm:res1}, \eqref{thm:res2}, \eqref{thm:res3},  we derive
	\begin{equation}
		\label{lem:pre_mu_nu_est}
		\|  \nabla \uemu^{m+1} \|  , \|  \nabla \uenu^{m+1} \|  \leq \hat{C} \frac{\Delta t^{{ \alpha} } + (\frac{1}{N})^{ {\beta} }}{\Delta t},
	\end{equation}
	where $\hat{C}$ depends only on $\delta_0^{*}, \up^{m+1}, \un^{m+1}, \upsi^{m+1}, \uuu^{m+1}$,  independent of  $\Delta t, \frac{1}{N}$. 
	
	Now taking the test function $ v_N = \uep^{m+1}-\uep^{m}$ in \eqref{err:eq1}, we have 
	\begin{equation}
		\label{lem:pre_bdd_0_p}
		\begin{split}
			& \quad \frac{1}{\Delta t} \|  \uep^{m+1}-\uep^{m} \| ^2 \\
			\leq 
			& \quad  \|  \uuu^{m} \uep^{m} + p^{m}_N \ueu^{m} \|  \|  \nabla (\uep^{m+1}-\uep^{m}) \|  \\
			& + \|  \uep^{m}(1 + 2\Delta t (\up^{m} + p^{m}_N)) \nabla \umu^{m+1}  + p^{m}_N(1 + 2\Delta t p^{m}_N)\nabla \uemu^{m+1} \|  \|  \nabla (\uep^{m+1}-\uep^{m}) \|  \\
			& + C(\Delta t^3 + N^{-k}) \|  \nabla(\uep^{m+1}-\uep^{m}) \|  \\
			\le 
			& \quad \bigg( \|  \uuu^{m} \| _{\infty} \|  \uep^{m} \|  
						+ \|  p^{m}_N \| _{\infty} \|  \ueu^{m} \|    
						+ \|  (1 + 2\Delta t (\up^{m} + p^{m}_N)) \nabla \umu^{m+1} \| _{\infty} \|  \uep^{m} \|  \\
			& 			+ \|  p^{m}_N (1 + 2\Delta t p^{m}_N) \| _{\infty} \|  \nabla \uemu^{m+1} \|  
						+ C(\Delta t^3 + N^{-k})
						\bigg) \|  \nabla( \uep^{m+1}-\uep^{m} )\|  \\
			\lesssim & \quad N \bigg( \|  \uuu^{m} \| _{\infty} \|  \uep^{m} \|  
						+ \|  p^{m}_N \| _{\infty} \|  \ueu^{m} \|    
						+ \|  (1 + 2\Delta t (\up^{m} + p^{m}_N)) \nabla \umu^{m+1} \| _{\infty} \|  \uep^{m} \|  \\
			& 			+ \|  p^{m}_N (1 + 2\Delta t p^{m}_N) \| _{\infty} \|  \nabla \uemu^{m+1} \|  
						+ C(\Delta t^3 + N^{-k})
						\bigg) \|  \uep^{m+1}-\uep^{m}\| ,
		\end{split}  
	\end{equation}
	where we have used the inverse inequality
	\begin{equation*}
		\|  \nabla(\uep^{m+1}-\uep^{m}) \|  \lesssim N \|  \uep^{m+1}-\uep^{m} \| .
	\end{equation*}
	Combining \eqref{lem:pre_bdd1_p}, \eqref{lem:pre_mu_nu_est} with \eqref{lem:pre_bdd_0_p}, we have 
	\begin{equation}
		\begin{split}\label{bound:er}
			\|  \uep^{m+1}-\uep^{m} \|  
			& \leq C N \Delta t ( \|  \uep^{m} \|  + \|  \ueu^{m} \|  + \|  \nabla \uemu^{m+1} \|  + \Delta t^3 + N^{-k}  )  \\
			& \leq C N \Delta t \frac{\Delta t^{ { \alpha}  } + N^{ {-\beta} } }{\Delta t} \\
			& \leq C (\Delta t^{ { \alpha - 1} }  + N^{ {-\beta + 1} } ),
		\end{split}
	\end{equation}
	where we have used $\Delta t \leq C \frac{1}{N}$ in \eqref{bound:er}. 
	
	Finally, using the triangle inequality and the inverse inequality, we have 
	\begin{align*}
		& \|  \uep^{m+1} \|  \leq \|  \uep^{m} \|  + \|  \uep^{m+1}-\uep^{m} \|  \leq  C(\Delta t^{ \alpha - 1} + N^{-\beta + 1} ),  \\
		& \|  \uep^{m+1} \| _{\infty}  \leq N \|  \uep^{m+1} \|  \leq
		C(\Delta t^{{ \alpha - 2} } + N^{{- \beta + 2}} ). 
	\end{align*} 
	Similarly, we can derive the bound for  $\uen^{m+1}$:
	\begin{align*}
		& \|  \uen^{m+1} \| \leq C(\Delta t^{ \alpha - 1}  + N^{ -\beta + 1} ), \\
		& \|  \uen^{m+1} \|_{\infty} \leq C(\Delta t^{ \alpha - 2} + N^{{ -\beta + 2} }).
	\end{align*}
	This completes the proof of the lemma. 
\end{proof}

 \subsection{A refined error analysis}
{
Firstly, 
for error terms as defined in \eqref{err_term_mod_def}, we provide following equations
\begin{align}
	& \nabla( \ln{\up^{m+1}} - \ln{p^{m+1}_N}) = \frac{1}{p^{m+1}_N} (\nabla \uep^{m+1} - \uep^{m+1} \nabla \ln{\up^{m+1}}  )\label{aux_1},   \\
	& \nabla( \ln{\un^{m+1}} - \ln{n^{m+1}_N})  = \frac{1}{n^{m+1}_N} (\nabla \uen^{m+1} - \uen^{m+1} \nabla \ln{\un^{m+1}}  )  \label{aux_2}.
\end{align}
Equation \eqref{aux_1} could be derived as 
\begin{equation*}
	\begin{split}
		& \quad \nabla( \ln{\up^{m+1}} - \ln{p^{m+1}_N})  \\
		& = (\frac{\nabla \up^{m+1}}{\up^{m+1}}-\frac{\nabla p^{m+1}_N}{p^{m+1}_N}) \\
		&=  (\frac{p^{m+1}_N}{p^{m+1}_N} \frac{\nabla \up^{m+1}}{\up^{m+1}}-\frac{\nabla p^{m+1}_N}{p^{m+1}_N})\\
		& = \frac{1}{p^{m+1}_N} (\frac{\up^{m+1} - \uep^{m+1}}{\up^{m+1}}\nabla \up^{m+1}  - \nabla p^{m+1}_N ) \\
		& = \frac{1}{p^{m+1}_N} (\nabla \uep^{m+1} - \uep^{m+1} \nabla \ln{\up^{m+1}}  ).
	\end{split}
\end{equation*}
And equation \eqref{aux_2} could be established similarly. 
} 

Now we proceed to a refined error analysis. The main result is 
\begin{theorem}\label{thm:main_sub}
	Under the same assumption  and procedure as in Lemma \ref{lem:hoa},  we can build supplementary fields $(\up, \un, \uuu, \uphi)$, provided $\Delta t$ and $\frac{1}{N}$ sufficiently small and under the linear refinement requirement $\Delta t \le \frac{1}{N}$,  for the numerical error between 
	numerical solution from scheme \eqref{scheme:p1}-\eqref{scheme:u2} 
	and  supplementary fields $(\up, \un, \uuu, \uphi)$,  as defined in \eqref{err_term_mod_def},  we have
	\begin{equation*}
		\begin{split}
			& \quad \|  \uep^{m} \|  + \|  \uen^{m} \|  + \|  \ueu^{m} \|  + \Delta t \|  \nabla \uephi^{m} \|   \\
			& \quad + C_{\delta_0^{*}, M}^{1} (\Delta t \sum_{l=1}^{m}(\|  \nabla \uep^{l} \| ^2 + \| \nabla \uen^{l} \| ^2 + \|  \nabla \ueru^{l} \| ^2 ))^{\frac{1}{2}} \\
			& \leq C_{\delta_0^{*}, M}^{2} (\Delta t^{3} + N^{-k}),
		\end{split}
	\end{equation*}
        for all positive integer $m$, such that $m\Delta t \leq T$,
	where $C_{\delta_0^{*}, M}^1, C_{\delta_0^{*}, M}^2$ are positive constants that are independent of the choice of $\Delta t, N$.
\end{theorem} 
 
\begin{proof}
    The proof of Theorem \ref{thm:main_sub} is divided into two steps:
    \begin{itemize}
    \item  {\bf Step 1:} Assume that the rough estimate \eqref{thm:res1}-\eqref{thm:res3} is true for all the $m \leq m^{*}$, where $m^{*}\Delta t \leq T $, we will obtain an error estimate for the $(m+1)$-th time step as \eqref{thm:err_gron0};
    \item  {\bf Step 2:} Recover the rough estimate \eqref{thm:res1}-\eqref{thm:res3} for the $(m^{*} + 1)$-th time step.
    \end{itemize}
    
    {\bf Step 1: A refined error analysis with a prior assumption. }
    
	First, from the choice of initial data:
	\begin{align*}
		& p^{0}_N = \projN p(\cdot, 0) = \up^{0}, \quad n^{0}_N = \projN n(\cdot, 0) = \un^{0}, \quad \psi^{0}_N = \projN \psi(\cdot, 0) = \upsi^{0}, \\
		& \uu^{0}_N = \projN \uu(\cdot, 0) = \uuu^{0}, \quad \phi^{0}_N = \projN \phi(\cdot, 0) =  \uphi^{0},
	\end{align*}
	we have 
	\begin{equation*}
		\uep^{0} = \uen^{0} = \uepsi^{0} = \ueu^{0} = \uephi^{0} = 0.
	\end{equation*}
	
    Assume \eqref{thm:res1}-\eqref{thm:res3} hold for the $m$-th time step  with $\alpha = \frac{11}{4}, \ \beta = k - \frac{1}{4}$. Then by Lemma \ref{lem:pre}, we have 
	\begin{align*}
		& \|  \uep^{m+1} \| _{\infty} \leq C(\Delta t^{\frac{3}{4}} + (\frac{1}{N})^{k-\frac{9}{4}}) \leq \frac{\delta_0^{*}}{2}, \\
		& \|  \uen^{m+1} \| _{\infty} \leq C(\Delta t^{\frac{3}{4}} + (\frac{1}{N})^{k-\frac{9}{4}}) \leq \frac{\delta_0^{*}}{2},
	\end{align*}
	where $\delta_0^{*} >0 $ is sufficiently small. Since $\up, \un$ are also bounded, we  obtain
	\begin{align}
		& \frac{\delta_0^{*}}{2} 
			\leq \min \up^{m+1}-\|  \uep^{m+1} \| _{\infty} 
			\leq p^{m+1}_N 
			\leq \|  \up^{m+1} \| _{\infty} + \|  \uep^{m+1} \| _{\infty} 
			\leq M + \frac{\delta_0^{*}}{2},
			\label{thm:err_bdd_p}	 \\
		& \frac{\delta_0^{*}}{2} 
			\leq \min \up^{m+1}-\|  \uep^{m+1} \| _{\infty} 
			\leq n^{m+1}_N 
			\leq \|  \un^{m+1} \| _{\infty} + \|  \uen^{m+1} \| _{\infty} 
			\leq M + \frac{\delta_0^{*}}{2}
			\label{thm:err_bdd_n}.
	\end{align}
	Now we proceed to the proof, which is divided into two steps.
	
	{ \bf (i) Estimate of \eqref{err:eq1}-\eqref{err:eq5}:}
	
	Taking the test function $v_N = \uep^{m+1}$ in \eqref{err:eq1}, we obtain
	\begin{equation}
		\label{thm:err_equ_p}
		\begin{split}
			& \quad \frac{1}{2\Delta t} (\|  \uep^{m+1} \| ^2-\|  \uep^{m} \| ^2 + \|  \uep^{m+1}-\uep^{m} \| ^2) \\ 
			& = \langle \uep^{m} \uuu^{m}, \nabla \uep^{m+1} \rangle 
				+ \langle p^{m}_N \ueu^{m}, \nabla \uep^{m+1} \rangle \\
			& \quad-\langle \uep^{m}(1 + 2\Delta t(\up^{m} + p^{m}_N)) \nabla \umu^{m+1}, \nabla \uep^{m+1} \rangle \\
			& \quad -\langle p^{m}_N(1 + 2\Delta t p^{m}_N) \nabla \uemu^{m+1}, \nabla \uep^{m+1} \rangle  \\
			& \quad + \langle \utau^{m+1}_{p},\uep^{m+1} \rangle.
		\end{split}
	\end{equation} 
	Using $\uemu^{m+1} = \ln{\up^{m+1}}-\ln{p^{m+1}_N} + \uepsi^{m+1}$  and \eqref{aux_1} we have 
	{
	\begin{equation}
		\label{thm:err_equ_p1}
		\begin{split}
			& \quad -\langle p^{m}_N(1 + 2\Delta t p^{m}_N) \nabla \uemu^{m+1}, \nabla \uep^{m+1} \rangle \\
			& = -\langle  p^{m}_N(1 + 2\Delta t p^{m}_N) \nabla (\ln{\up^{m+1}}-\ln{p^{m+1}_N}), \nabla \uep^{m+1} \rangle 
				- \langle p^{m}_N(1 + 2\Delta t p^{m}_N) \nabla \uepsi^{m+1}, \nabla \uep^{m+1} \rangle  \\ 
			& =-\langle \frac{p^{m}_N(1 + 2\Delta t p^{m}_N)}{p^{m+1}_N} \nabla \uep^{m+1}, \nabla \uep^{m+1} \rangle  
				- \langle  \frac{p^{m}_N(1 + 2\Delta t p^{m}_N)}{p^{m+1}_N} \uep^{m+1} \nabla \ln{ \up^{m+1}}, \nabla \uep^{m+1} \rangle \\
			& \quad-\langle p^{m}_N(1 + 2\Delta t p^{m}_N) \nabla \uepsi^{m+1}, \nabla \uep^{m+1} \rangle.
		\end{split}
	\end{equation}
	}
	Using the bounds of $p^{m}_N$ and $p^{m+1}_N$ given in \eqref{lem:pre_bdd1_p}, \eqref{thm:err_bdd_p},  we have 
	\begin{equation}
		\label{thm:err_equ_p2}
		- \langle \frac{p^{m}_N(1 + 2\Delta t p^{m}_N)}{p^{m+1}_N} \nabla \uep^{m+1}, \nabla \uep^{m+1} \rangle  
		\leq 
		- \frac{\delta_0^{*}}{2M + \delta_0^{*}} 
		\|  \nabla \uep^{m+1} \| ^2.
	\end{equation}
	For the last two terms in \eqref{thm:err_equ_p}, and the right hand side terms in \eqref{thm:err_equ_p1}, applying Hölder and Young's inequalities and the  properties in \eqref{thm:err_bdd_p}, \eqref{lem:pre_bdd1_p}, we have 
	\begin{equation}
		\label{thm:err_equ_p3}
		\begin{split}
			& \quad \langle \uep^{m} \uuu^{m}, \nabla \uep^{m+1} \rangle 
				+ \langle p^{m}_N \ueu^{m}, \nabla \uep^{m+1} \rangle \\
			& \leq \|  \uuu^{m} \| _{\infty} \|  \uep^{m} \|  \|  \nabla \uep^{m+1} \|  
				+  \|  p^{m}_N \| _{\infty} \|  \ueu^{m} \|  \|  \nabla \uep^{m+1} \|   \\
			& \leq \frac{1}{8} \frac{\delta_0^{*}}{2M + \delta_0^{*}} \|  \nabla \uep^{m+1} \| ^2 
				+ 
				\frac{4 M + 2 \delta_0^{*}}{\delta_0^{*}}
				\big(
				\|  \uuu^{m} \| _{\infty}^2  \|  \uep^{m} \| ^2
				+ \|  p^{m}_N \| _{\infty}^2  \|  \ueu^{m} \| ^2    
				\big) \\
			& \leq \frac{1}{8} \frac{\delta_0^{*}}{2M + \delta_0^{*}}  \|  \nabla \uep^{m+1} \| ^2_2
				+ C_{M, \delta_0^{*}}
				(\|  \uep^{m} \| ^2_2 + \|  \ueu^{m} \| ^2_2),
		\end{split}
	\end{equation}
	
	\begin{equation}
		\label{thm:err_equ_p4}
		\begin{split}
			& \quad \langle \uep^{m}(1 + 2\Delta t(\up^{m} + p^{m}_N)) \nabla \umu^{m+1}, \nabla \uep^{m+1} \rangle  \\
			& \leq \|  1 + 2\Delta t (\up^{m} + p^{m}_N) \| _{\infty} 
					\|  \nabla \umu^{m+1} \| _{\infty} 
			 		\|  \uep^{m} \|   \|  \nabla \uep^{m+1} \|  \\
			& \leq \frac{1}{8} \frac{\delta_0^{*}}{2M + \delta_0^{*}}  \|  \nabla \uep^{m+1} \| ^2 
				+ C_{M, \delta_0^{*}} \|  \uep^{m} \| ^2,
		\end{split}
	\end{equation}
	
	\begin{equation}
		\label{thm:err_equ_p5}
		\begin{split}
			& \quad  \langle  \frac{p^{m}_N(1 + 2\Delta t p^{m}_N)}{p^{m+1}_N} \uep^{m+1} { \nabla \ln{\up^{m+1}}}, \nabla \uep^{m+1} \rangle \\
			& \leq \|  \frac{p^{m}_N(1 + 2\Delta t p^{m}_N)}{p^{m+1}_N} \| _{\infty}  \| \frac{\nabla \up^{m+1}}{\up^{m+1}} \| _{\infty} \|  \uep^{m+1} \|  \|  \nabla \uep^{m+1} \|   \\
			& \leq \frac{1}{8} \frac{\delta_0^{*}}{2M + \delta_0^{*}}  \|  \nabla \uep^{m+1} \| ^2 
				+ C_{M, \delta_0^{*}} \|  \uep^{m+1} \| ^2, 
		\end{split}
	\end{equation}
	and
	\begin{equation}
		\label{thm:err_equ_p6}
		\begin{split}
			& \quad \langle p^{m}_N(1 +2 \Delta t p^{m}_N) \nabla \uepsi^{m+1}, \nabla \uep^{m+1} \rangle \\
			& \leq \|  p^{m}_N (1 +2 \Delta t p^{m}_N) \| _{\infty} 
					\|  \nabla \uepsi^{m+1} \|  
					\|  \nabla \uep^{m+1} \|  \\
			& \leq \frac{1}{16} \frac{\delta_0^{*}}{2M + \delta_0^{*}} \|  \nabla \uep^{m+1} \| ^2 
					+ C_{M, \delta_0^{*}} \|  \nabla \uepsi^{m+1} \| ^2 \\
			& \leq \frac{1}{16} \frac{\delta_0^{*}}{2M + \delta_0^{*}} \|  \nabla \uep^{m+1} \| ^2
					+ C_{M, \delta_0^{*}} (\|  \uep^{m+1} \| ^2 + \|  \uen^{m+1} \| ^2),
		\end{split}
	\end{equation}
	where we have used the elliptic estimate from \eqref{err:eq5} to get
	\begin{equation*}
		\label{thm:err_equ_p6_1}
		{ \|  \nabla \uepsi^{m+1} \| ^2} \leq C (\|  \uep^{m+1} \| ^2 + \|  \uen^{m+1} \| ^2).
	\end{equation*}
	From Lemma \ref{lem:hoa}, we have 
	\begin{equation}
		\label{thm:err_equ_p7}
		\begin{split}
			\utau^{m+1}_{p}(\uep^{m+1}) 
			& \leq C(\Delta t^3 + N^{-k}) \|  \uep^{m+1} \| _{H^{1}}  \\
			& \leq \frac{1}{16} \frac{\delta_0^{*}}{2M + \delta_0^{*}} \|  \nabla \uep^{m+1} \| ^2 
					+ C_{M, \delta_0^{*}} (\Delta t^3 + N^{-k})^2.
		\end{split}
	\end{equation}
	Plugging \eqref{thm:err_equ_p1}, \eqref{thm:err_equ_p2}, \eqref{thm:err_equ_p3}, \eqref{thm:err_equ_p4}, \eqref{thm:err_equ_p5},
		 \eqref{thm:err_equ_p6}, \eqref{thm:err_equ_p7} into \eqref{thm:err_equ_p}, we obtain 
	\begin{equation}
	\label{thm:err_gron_p}
		\begin{split}
			& \quad \frac{1}{2\Delta t} (\|  \uep^{m+1} \| ^2-\|  \uep^{m} \| ^2 + \|  \uep^{m+1}-\uep^{m} \| ^2) 
				+ \frac{1}{2} \frac{\delta_0^{*}}{2M + \delta_0^{*}} \|  \nabla \uep^{m+1} \| ^2  \\
			& \leq C_{M, \delta_0^{*}} \big(\|  \uep^{m} \| ^2 + 
										 \|  \ueu^{m} \| ^2 +
										 \|  \uep^{m+1} \| ^2
										 + \|  \uen^{m+1} \| ^2
										 + (\Delta t ^3 + N^{-k})^2 \big).
		\end{split} 
	\end{equation}
	Similarly, taking $v_N = \uen^{m+1}$ in \eqref{err:eq3}, we obtain
	\begin{equation}
	\label{thm:err_gron_n}
		\begin{split}
			& \quad \frac{1}{2\Delta t} (\|  \uen^{m+1} \| ^2-\|  \uen^{m} \| ^2 + \|  \uen^{m+1}-\uen^{m} \| ^2) 
				+ \frac{1}{2} \frac{\delta_0^{*}}{2M + \delta_0^{*}} \|  \nabla \uen^{m+1} \| ^2  \\
			& \leq C_{M, \delta_0^{*}} \big(\|  \uen^{m} \| ^2 + 
										 \|  \ueu^{m} \| ^2 +
										 \|  \uep^{m+1} \| ^2
										 + \|  \uen^{m+1} \| ^2
										 + (\Delta t ^3 + N^{-k})^2 \big).
		\end{split}
	\end{equation}
	
{\bf (ii) Estimate of \eqref{err:eq6}-\eqref{err:eq8}}. 
	
	Taking $v_N = \ueru^{m+1}$ in \eqref{err:eq6} yields 
	\begin{equation}
		\label{thm:err_equ_u0}
		\begin{split}
			& \quad \frac{1}{2\Delta t}( \|  \ueru^{m+1} \| ^2-\|  \ueu^{m} \| ^2 + \|  \ueru^{m+1}-\ueu^{m} \| ^2) \\
			& \quad + \langle (\ueu^{m} \cdot \nabla) R_N \uuu^{m+1}, \ueru^{m+1} \rangle
				 + \|  \nabla \ueru^{m+1} \| ^2 + \langle \nabla \uephi^{m}, \ueru^{m+1} \rangle \\
			& = -\langle \uep^{m} \nabla \umu^{m+1} + \uen^{m} \nabla \unu^{m+1}, \ueru^{m+1} \rangle
			 -\langle p^{m}_N \nabla \uemu^{m+1} + n^{m}_N \nabla \uenu^{m+1}, \ueru^{m+1} \rangle \\
			& \quad + \langle \utau^{m+1}_u, \ueru^{m+1} \rangle,
		\end{split}
	\end{equation}
	{ where we have used \eqref{scheme:u3} to obtain}
	\begin{equation*}
		\langle (\uu^{m}_N \cdot \nabla) \ueru^{m+1}, \ueru^{m+1} \rangle = 0.
	\end{equation*}
	Taking the test function $v_N = \frac{1}{2}(\ueu^{m+1} + \ueru^{m+1})$ in \eqref{err:eq7}, we obtain
	\begin{equation}
		\label{thm:err_equ_u1}
		\frac{1}{2 \Delta t}(\|  \ueu^{m+1} \| ^2-\|  \ueru^{m+1} \| ^2) + \frac{1}{2} \langle \nabla(\uephi^{m+1}-\uephi^{m}), \ueru^{m+1} \rangle = 0.
	\end{equation}
	Summing \eqref{thm:err_equ_u0}  with  \eqref{thm:err_equ_u1}, we have 
	\begin{equation}
		\label{thm:err_equ_u2}
		\begin{split}
			& \quad \frac{1}{2\Delta t}(\|  \ueu^{m+1} \| ^2-\|  \ueu^{m} \| ^2 + \|  \ueru^{m+1}-\ueu^{m} \| ^2) \\
			& \quad + \langle (\ueu^{m} \cdot \nabla) R_N \uuu^{m+1}, \ueru^{m+1} \rangle  
			  + \|  \nabla \ueru^{m+1} \| ^2 + \frac{1}{2}\langle \nabla(\uephi^{m+1} + \uephi^{m}), \ueru^{m+1} \rangle \\
			& = -\langle \uep^{m} \nabla \umu^{m+1} + \uen^{m} \nabla \unu^{m+1}, \ueru^{m+1} \rangle
			 -\langle p^{m}_N \nabla \uemu^{m+1} + n^{m}_N \nabla \uenu^{m+1}, \ueru^{m+1} \rangle \\
			& \quad + \langle  \utau^{m+1}_u, \ueru^{m+1} \rangle.
		\end{split}
	\end{equation}
	For the second term in \eqref{thm:err_equ_u2}, we have
	\begin{equation}
		\label{thm:err_equ_u3}
		\begin{split}
			& \quad \rvert \langle (\ueu^{m} \cdot \nabla) R_N \uuu^{m+1}, \ueru^{m+1} \rangle\rvert  \\
			& =-\langle (\ueu^{m} \cdot \nabla) \ueru^{m+1}, R_N \uuu^{m+1} \rangle \\
			& \leq \|  \ueu^{m} \|  \|  \nabla \ueru^{m+1} \|  \|  R_N \uuu^{m+1} \| _{\infty} \\
			& \leq \frac{1}{4} \|  \nabla \ueru^{m+1} \| ^2 + \|  R_N \uuu^{m+1} \| _{\infty}^2 \|  \ueu^{m} \| ^2.
		\end{split}
	\end{equation}
	Taking the test function $v_N = \nabla(\uephi^{m+1} + \uephi^{m})$ in \eqref{err:eq7}, we obtain
	\begin{equation}
		\label{thm:err_equ_u4}
		\begin{split}
			\langle \nabla(\uephi^{m+1} + \uephi^{m}), \ueru^{m+1} \rangle = \Delta t (\|  \nabla \uephi^{m+1} \| ^2-\|  \nabla \uephi^{m} \| ^2).
		\end{split}
	\end{equation}
	For the first and second term on the right hand side of \eqref{thm:err_equ_u2}, we have
	\begin{equation}
		\label{thm:err_equ_u5} 
		\begin{split}
			& \quad \rvert \langle \uep^{m} \nabla \umu^{m+1} + \uen^{m} \nabla \unu^{m+1}, \ueru^{m+1} \rangle\rvert  \\
			& \leq (\|  \nabla \umu^{m+1} \| _{\infty} \|  \uep^{m} \|  + \|  \nabla \unu^{m+1} \| _{\infty} \|  \uen^{m} \| ) \|  \ueru^{m+1} \|  \\
			& \leq \|  \ueru^{m+1} \| ^2 + \frac{1}{4} (\|  \nabla \umu^{m+1} \| ^2_{\infty} \|  \uep^{m} \| ^2 + \|  \nabla \unu^{m+1} \| ^2_{\infty} \|  \uen^{m} \| ^2),
		\end{split}
	\end{equation}
	and 
	{
	\begin{equation}
		\begin{split}\label{err:eq:1}
			& \quad \langle p^{m}_N \nabla \uemu^{m+1} + n^{m}_N \nabla \uenu^{m+1}, \ueru^{m+1} \rangle \\
			& = \langle p^{m}_N \nabla (\ln{\up^{m+1}}-\ln{p^{m+1}_N} + \uepsi^{m+1}), \ueru^{m+1} \rangle 
				+ \langle n^{m}_N \nabla (\ln{\un^{m+1}}-\ln{n^{m+1}_N}-\uepsi^{m+1}), \ueru^{m+1} \rangle \\
			& = \langle p^{m}_N \nabla (\ln{\up^{m+1}}-\ln{p^{m+1}_N})
						, \ueru^{m+1} \rangle 
				+ \langle n^{m}_N \nabla (\ln{\un^{m+1}}-\ln{n^{m+1}_N})
						, \ueru^{m+1} \rangle \\
			&\quad 	+ \langle (p^{m}_N-n^{m}_N) \nabla \uepsi^{m+1}, \ueru^{m+1} \rangle. 
		\end{split}
	\end{equation}
	}
	Consider the first two terms on right-hand side of \eqref{err:eq:1} { and apply \eqref{aux_1} \eqref{aux_2},}  we have
	{
	\begin{equation}
		\begin{split}\label{err:eq:1_1}
			& \quad \rvert \langle p^{m}_N \nabla (\ln{\up^{m+1}}-\ln{p^{m+1}_N})
						, \ueru^{m+1} \rangle  \rvert  \\
			& = \langle \frac{p^{m}_N}{p^{m+1}_N} \nabla \uep^{m+1}, \ueru^{m+1} \rangle
				- \langle \frac{p^{m}_N}{p^{m+1}_N} \uep^{m+1} \nabla \ln{\up^{m+1}}, \ueru^{m+1} \rangle  \\
			& \leq \|  \frac{p^{m}_N}{p^{m+1}_N} \| _{\infty} \|  \nabla \uep^{m+1} \|  \|  \ueru^{m+1} \|   
					+ \|  \frac{p^{m}_N}{p^{m+1}_N} \| _{\infty} \| \nabla \ln{\up^{m+1}} \|_{\infty}  \|  \uep^{m+1} \|  \|  \ueru^{m+1} \|  \\ 
			& \leq \frac{1}{8} \frac{\delta_0^{*}}{2M + \delta_0^{*}}  
					(\|  \nabla \uep^{m+1} \| ^2  + \|  \uep^{m+1} \| ^2) + C_{\delta_0^{*}, M} \|  \ueru^{m+1} \| ^2,
		\end{split}
	\end{equation}
and 
	\begin{equation}\label{err:eq:1_2}
		\rvert \langle n^{m}_N \nabla (\ln{\un^{m+1}}-\ln{n^{m+1}_N})
						, \ueru^{m+1} \rangle \rvert 
		\leq \frac{1}{8} \frac{\delta_0^{*}}{2M + \delta_0^{*}}  
					(\|  \nabla \uen^{m+1} \| ^2  + \|  \uen^{m+1} \| ^2) + C_{\delta_0^{*}, M} \|  \ueru^{m+1} \| ^2.
	\end{equation}
	}
	Using the estimate \eqref{thm:err_equ_p6_1}, for the final term in \eqref{err:eq:1}, we obtain 
	\begin{equation}
		\begin{split}\label{err:eq:1_3}
			& \quad \abs{ \langle (p^{m}_N-n^{m}_N) \nabla \uepsi^{m+1}, \ueru^{m+1} \rangle }  \\
			& \leq (\|  p^{m}_N \| _{\infty} + \|  n^{m}_N \| _{\infty} ) \|  \nabla \uepsi^{m+1} \| \|  \ueru^{m+1} \| \\
			& \leq C  (\|  p^{m}_N \| _{\infty} + \|  n^{m}_N \| _{\infty} ) 
					(\|  \uep^{m+1} \| + \|  \uen^{m+1} \|) 
					\|  \ueru^{m+1} \|  \\
			& \leq \|  \uep^{m+1} \|^2 + \|  \uen^{m+1} \|^2 + C_{\delta_0^{*}, M} \|  \ueru^{m+1} \| ^2.
		\end{split}
	\end{equation}
	Combining all these estimates \eqref{err:eq:1}-\eqref{err:eq:1_3}, we have 
	\begin{equation}
		\label{thm:err_equ_u6}
		\begin{split}
			& \quad \rvert \langle p^{m}_N \nabla \uemu^{m+1} + n^{m}_N \nabla \uenu^{m+1}, \ueru^{m+1} \rangle\rvert  \\
			& \leq \frac{1}{4} \frac{\delta_0^{*}}{2M + \delta_0^{*}} 
					(\|  \nabla \uep^{m+1} \| ^2 + \|  \nabla \uen^{m+1} \| ^2 )
					+ C_{\delta_0^{*}, M}(\|  \uep^{m+1} \| ^2_2 
										+ \|  \uen^{m+1} \| ^2_2
										+ \|  \ueru^{m+1} \| ^2_2   ).
		\end{split}
	\end{equation}
	From Lemma \ref{lem:hoa}, we have 
	\begin{equation}
		\label{thm:err_equ_u7}
		\begin{split}
		\langle	\utau_u^{m+1}, \ueu^{m+1} \rangle
			& \leq C(\Delta t^3 + N^{-k}) \|  \ueru^{m+1} \| _{H^{1}} \\
			& \leq \frac{1}{4} \|  \nabla \ueru^{m+1} \| ^2_2 + C(\Delta t^3 + N^{-k})^2.  
		\end{split}
	\end{equation}
	Plugging \eqref{thm:err_equ_u3}, \eqref{thm:err_equ_u4}, \eqref{thm:err_equ_u5}, \eqref{thm:err_equ_u6}, \eqref{thm:err_equ_u7} into \eqref{thm:err_equ_u2}, we obtain
	\begin{equation}
		\label{thm:err_gron_u_pre}
		\begin{split}
			& \quad \frac{1}{2\Delta t}(\|  \ueu^{m+1} \| ^2-\|  \ueu^{m} \| ^2 + \|  \ueru^{m+1}-\ueu^{m} \| ^2 + \Delta t^2 \|  \nabla \uephi^{m+1} \| ^2-\Delta t^2 \|  \nabla \uephi^{m} \| ^2  ) \\
			& \quad + \frac{1}{2} \|  \nabla \ueru^{m+1} \| ^2-\frac{1}{4} \frac{\delta_0^{*}}{2M + \delta_0^{*}} (\|  \nabla \uep^{m+1} \| ^2 + \|  \nabla \uen^{m+1} \| ^2) \\
			& \leq C_{\delta_0^{*}, M} (\|  \uep^{m} \| ^2 + \|  \uen^{m} \| ^2 + \|  \uep^{m+1} \| ^2 + \|  \uen^{m+1} \| ^2 + \|  \ueu^{m} \| ^2) \\
			& \quad  + C_{\delta_0^{*}, M} \|  \ueru^{m+1} \| ^2 + C(\Delta t^3 + N^{-k})^2.
		\end{split}
	\end{equation}
	Now taking the test function $v_N = \nabla(\uephi^{m+1}-\uephi^{m})$ in \eqref{err:eq7}, and combining \eqref{thm:err_equ_u1}, we have
	\begin{equation}\label{e:22}
		\|  \ueru^{m+1} \| ^2 = \|  \ueu^{m+1} \| ^2 + \Delta t^2 \|  \nabla(\uephi^{m+1}-\uephi^{m}) \| ^2.
	\end{equation}
	Plugging \eqref{e:22} into \eqref{thm:err_gron_u_pre}, we obtain 
	\begin{equation}
		\label{thm:err_gron_u}
		\begin{split}
			& \quad \frac{1}{2\Delta t}(\|  \ueu^{m+1} \| ^2-\|  \ueu^{m} \| ^2 + \|  \ueru^{m+1}-\ueu^{m} \| ^2 + \Delta t^2 \|  \nabla \uephi^{m+1} \| ^2-\Delta t^2 \|  \nabla \uephi^{m} \| ^2  ) \\
			& \quad + \frac{1}{2} \|  \nabla \ueru^{m+1} \| ^2-\frac{1}{4} \frac{\delta_0^{*}}{2M + \delta_0^{*}} (\|  \nabla \uep^{m+1} \| ^2 + \|  \nabla \uen^{m+1} \| ^2) \\
			& \leq C_{\delta_0^{*}, M} \big(\|  \uep^{m} \| ^2 + \|  \uen^{m} \| ^2 + \|  \uep^{m+1} \| ^2 + \|  \uen^{m+1} \| ^2 + \|  \ueu^{m} \| ^2 \\
			& \quad  + \|  \ueu^{m+1} \| ^2 + \Delta t^2 \|  \nabla \uephi^{m+1} \| ^2 + \Delta t^2 \|  \nabla \uephi^{m} \| ^2  \big) + C(\Delta t^3 + N^{-k})^2.
		\end{split}
	\end{equation}
	 
	{{\bf Step 2:} Recovery of the induction assumption \eqref{thm:res1}-\eqref{thm:res3} for the $(m^{*}+1)$-step. \quad }
	
	A summation of \eqref{thm:err_gron_p} \eqref{thm:err_gron_n} \eqref{thm:err_gron_u} leads to 
	\begin{equation}
		\label{thm:err_gron0}
		\begin{split}
			& \quad \frac{1}{2\Delta t}\big(\|  \uep^{m+1} \| ^2 + \|  \uen^{m+1} \| ^2 + \|  \ueu^{m+1} \| ^2
				+ \Delta t^2 \|  \nabla \uephi^{m+1} \| ^2 \\
			& \qquad-\|  \uep^{m} \| ^2-\|  \uen^{m} \| ^2-\|  \ueu^{m} \| ^2-\Delta t^2 \|  \uephi^{m} \| ^2 \\
			& \qquad + \|  \uep^{m+1}-\uep^{m} \| ^2  + \|  \uen^{m+1}-\uen^{m} \| ^2 + \|  \ueru^{m+1}-\ueu^{m} \| ^2  \big) \\
			& \quad + \frac{1}{4} \frac{\delta_0^{*}}{2M + \delta_0^{*}} (\|  \nabla \uep^{m+1} \| ^2 + \|  \nabla \uen^{m+1} \| ^2 )
				+ \frac{1}{2} \|  \nabla \ueru^{m+1} \| ^2  \\
			& \leq C_{\delta_0^{*}, M} \big( 
				\|  \uep^{m} \| ^2 + \|  \uep^{m+1} \| ^2 +
				\|  \uen^{m} \| ^2 + \|  \uen^{m+1} \| ^2  
				+ \|  \ueu^{m} \| ^2 \\
			& \quad + \|  \ueu^{m+1} \| ^2 + \Delta t^2 \|  \nabla \uephi^{m+1} \| ^2 + \Delta t^2 \|  \nabla \uephi^{m} \| ^2 \big)
				+ C_{\delta_0^{*}, M}(\Delta t^3 + N^{-k})^2.
		\end{split}
	\end{equation}
        Note that from the induction assumption in {\bf Step 1}, the above inequality holds for all $m \leq m^{*}$, where $m^{*}\Delta t \leq T$.
	An application of discrete Gronwall's inequality implies
	\begin{equation*}
	    \label{thm:main_sub_final}
		\begin{split}
			& \quad \|  \uep^{m^{*} + 1} \|  + \|  \uen^{m^{*} + 1} \|  + \|  \ueu^{m^{*} + 1} \|  + \Delta t \|  \nabla \uephi^{m^{*} + 1} \|   \\
			& \quad + C_{\delta_0^{*}, M}^{1} (\Delta t \sum_{l=1}^{m^{*} + 1}(\|  \nabla \uep^{l} \| ^2 + \| \nabla  \uen^{l} \| ^2 + \|  \nabla \ueru^{l} \| ^2 ))^{\frac{1}{2}} \\
			& \leq C_{\delta_0^{*}, M}^{2} (\Delta t^{3} + N^{-k}),
		\end{split}
	\end{equation*}
	where $C_{\delta_0^{*}, M}^1, C_{\delta_0^{*}, M}^2$ are positive constants,  independent  of $\Delta t, N$. Then we obtain higher order error estimate for $\up, \un, \upsi, \uuu$ and are able to recover our induction assumption \eqref{thm:res1}-\eqref{thm:res3} with { $\alpha = \frac{11}{4}, \ \beta = k - \frac{1}{4}$ and} $\Delta t, \frac{1}{N}$ chosen small enough. 
	This completes the proof of Theorem \ref{thm:main_sub}.
\end{proof}

\subsection{{Proof of Theorem \ref{thm:main}}}
Now we are ready to prove our main result Theorem \ref{thm:main} {,  which is a direct combination of} Theorem \ref{regularity} and Theorem \ref{thm:main}.
\begin{proof}
    Given $p^{in}, n^{in} \geq \delta_0$ for some $\delta_0 >0$, from Corollary \ref{maximal_principle}, we have solution $p,n \geq \delta_0$ in $\Omega \times [0,T]$. 
    
    Also from Theorem \ref{regularity} and \eqref{interpolation-ineq}, we have
    \begin{align*}
        & \| \pt^4 p \|_{L^{\infty} L^{2}(\Omega \times [0,T])}^{2} \lesssim \| \pt^4 p \|_{L^{2}H^{1}(\Omega \times [0,T])} \| \pt^5 p\|_{L^2 H^{-1}(\Omega \times [0,T])} \leq C(T, \| p^{in} \|_{H^{8}(\Omega)}), \\
        & \| \pt^3 p \|_{L^{\infty} H^{k+1}(\Omega \times [0,T]) }^{2} \lesssim \| \pt^3 p \|_{L^{2}H^{k+2}(\Omega \times [0, T])} \| \pt^4 p \|_{L^{2}H^{k}(\Omega \times [0,T])} \leq C(T, \| p^{in} \|_{H^{k+7}(\Omega)}).
    \end{align*}
    Similar results hold for $(n, \uu)$. Then given $(p^{in}, n^{in}, \uu^{in}) \in H^{k+7}(\Omega) \times H^{k+7}(\Omega) \times H^{k+7}(\Omega)$ with $k \geq 2$, we have
    \begin{equation*}
        (\pt^4 p, \pt^4 n, \pt^4 \uu) \in L^{\infty}(0, T, L^{2}(\Omega)), (\pt^3 p, \pt^3 n, \pt^3 \uu) \in L^{\infty}(0, T, H^{k+1}(\Omega)), \ (k \geq 2).
    \end{equation*}
    Hence assumptions in Lemma \ref{appendix:hoa} are satisfied, and Theorem \ref{thm:main_sub} follows. 

    From the error term definition \eqref{err_term_def} \eqref{err_term_mod_def}, we have 
    \begin{equation}\label{err_relations}
        \begin{split}
            &  \ep^{m} = \uep^{m}-\Delta t p_{\Delta t, 1}^{m}-\Delta t^2 p_{\Delta t, 2}^{m}, \\
            &  \en^{m} = \uen^{m}-\Delta t n_{\Delta t, 1}^{m}-\Delta t^2 n_{\Delta t, 2}^{m}, \\
            &  \eu^{m} = \ueu^{m}-\Delta t \uu_{\Delta t, 1}^{m}-\Delta t^2 \uu_{\Delta t, 2}^{m}.
        \end{split}
    \end{equation}
    From the construction process in the appendix, the modification functions $(p_{\Delta t,i}, n_{\Delta t, i}, \uu_{\Delta t ,i})(i = 1,2)$ have sufficient regularity. 
    Combining Theorem \ref{thm:main_sub} with \eqref{err_relations}, Theorem \ref{thm:main} is proved.
\end{proof}

{\color{blue}
\begin{remark}
As shown in Theorem \ref{thm:main},  the numerical scheme \eqref{scheme:p1} - \eqref{scheme:u3} is a first-order temporal accurate scheme.  There are some recent studies \cite{liu2023second} which extends the PNP scheme to a 2nd order one using Crank-Nicolson type of scheme which preserves positivity, energy stability and unique solvability.  However, it is challenging to extend the current method to a second-order temporal accurate scheme that still preserve those nice properties, and at the same time keeping the PNP system and NS system solving process decoupled.  The major challenges are:
\begin{itemize}
	\item The numerical technique relaxing the conviction term in \eqref{scheme:p1}  \eqref{scheme:n1} by adding $\mathcal{O}(\Delta t)$ term, which decouples the PNP and NS system,  is not extendable to 2nd order scheme. 
	\item To design a unconditionally energy stable Crank-Nicolson type numerical scheme for Navier-Stokes scheme is non-travil \cite{shen1996error}, and it would take further difficulties to decouple the two systems and preserve the energy law at the same time.
\end{itemize}

\end{remark}
}

\section{Numerical Examples}
\label{ch-PNPNS-sec-num}

In this section, we present numerical experiments to validate the stability, positivity, and accuracy of our numerical schemes. We consider periodic boundary conditions and implement the Fourier spectral method in $\Omega = [0, 2\pi]^2$.

\subsection{Accuracy Test}

To verify the accuracy and convergence rate of our numerical scheme, we introduce an artificial exact solution by adding external forces to the PNP-NS system, formulated as
\begin{equation*}
	\begin{aligned}
		& p_t + (\mathbf{u} \cdot \nabla) p = \nabla \cdot (\nabla p + p \nabla \psi) + f_p, \\
		& n_t + (\mathbf{u} \cdot \nabla) n = \nabla \cdot (\nabla n - n \nabla \psi) + f_n, \\
		& -\varepsilon \Delta \psi = p - n, \\
		& \mathbf{u}_t + (\mathbf{u} \cdot \nabla) \mathbf{u} - \Delta \mathbf{u} + \nabla P = -\nabla \psi (p - n) + f_{\mathbf{u}}, \\
		& \nabla \cdot \mathbf{u} = 0,
	\end{aligned}
\end{equation*}
where {\color{blue} we set $\varepsilon = 1$} and the source terms $f_p$, $f_n$, and $f_{\mathbf{u}}$ are determined from the exact solutions
\begin{equation*}
	\left\{
	\begin{aligned}
		p(x, y, t) &= 1.1 + \cos(x) \cos(y) \sin(t), \\
		n(x, y, t) &= 1.1 - \cos(x) \cos(y) \cos(t), \\
		\mathbf{u}(x, y, t) &=
		\begin{pmatrix}
			\sin^2(x) \sin(2y) \sin(t) \\
			-\sin(2x) \sin^2(y) \cos(t)
		\end{pmatrix}, \\
		{\color{blue} P(x, y, t)} &= \cos(x) \cos(y) \sin(t),
	\end{aligned}
	\right.
\end{equation*}
defined in the domain $\Omega \times [0, T] = [0, 2\pi]^2 \times [0, T]$. We use $N = 64$ Fourier modes with different time steps $\Delta t$. Using scheme \eqref{scheme:p1}--\eqref{scheme:u3}, we compute the $L^2$ errors between the numerical solutions and the exact solutions. The results are shown in Table~\ref{Table:accuracy_test}, where first-order convergence rates are observed for the different variables.

\begin{table}[ht]
\caption{$L^2$ errors and convergence orders for the numerical solutions of $p$, $\psi$, $\mathbf{u}$, and $\psi$}
\centering
    \begin{tabular}{ccccccccc}
	\toprule
		$\Delta t$ 				&  	$L^2$ error in $p$ 	&	Order 	& 	$L^2$ error in $\psi$ 		& 	Order	& $L^2$ error in $\mathbf{u}$	& 	Order 	&	$L^2$ error in $\psi$		& 	Order   \\
		\midrule
		$1\times 10^{-2}$				& $1.01\times 10^{-2}$	& 	--		& 	$4.24\times 10^{-3}$	& 	--		& 	$6.33\times 10^{-4}$	&	--		&	$1.21\times 10^{-2}$	&	--		\\
		$\frac{1}{2} \times 10^{-2}$		& $5.11\times 10^{-3}$	& 	0.98		& 	$2.15\times 10^{-3}$	& 	0.98	& 	$3.17\times 10^{-4}$	&	1.00	&	$6.13\times 10^{-3}$	&	0.99	\\
		$\frac{1}{4} \times 10^{-2}$		& $2.57\times 10^{-3}$	& 	0.99		& 	$1.08\times 10^{-3}$	& 	0.99	&	$1.59\times 10^{-4}$	&	1.00	&	$3.08\times 10^{-3}$	&	0.99	\\
		$\frac{1}{8} \times 10^{-2}$		& $1.29\times 10^{-3}$	& 	1.00		& 	$5.45\times 10^{-4}$	& 	0.99	&	$7.93\times 10^{-5}$	&	1.00	&	$1.54\times 10^{-3}$	&	1.00	\\
		$\frac{1}{16} \times 10^{-2}$		& $6.46\times 10^{-4}$	& 	1.00		& 	$2.73\times 10^{-4}$	& 	1.00	&	$3.97\times 10^{-5}$	&	1.00	&	$7.73\times 10^{-4}$	&	1.00	\\
		$\frac{1}{32} \times 10^{-2}$		& $3.23\times 10^{-4}$	& 	1.00		& 	$1.37\times 10^{-4}$	& 	1.00	&	$1.98\times 10^{-5}$	&	1.00	&	$3.87\times 10^{-4}$	&	1.00	\\
	\bottomrule
	\end{tabular}
	\label{Table:accuracy_test}
\end{table}	

\subsection{Property Test}
\label{ch-PNPNS-sec-example2}
We also perform numerical simulations to test the mass-conserving and positivity-preserving properties of our scheme. The positivity-preserving scheme is applied to solve the following PNP-NS system:
\begin{align}\label{pnp:new}
    & p_t + (\mathbf{u} \cdot \nabla) p = \nabla \cdot (\nabla p + p \nabla \psi), \nonumber\\
    & n_t - (\mathbf{u} \cdot \nabla) n = \nabla \cdot (\nabla n - n \nabla \psi), \nonumber\\
    & -\varepsilon \Delta \psi = p - n, \\
    & \mathbf{u}_t + (\mathbf{u} \cdot \nabla) \mathbf{u} + \nabla P - \Delta \mathbf{u}  = -\kappa \nabla \psi (p - n), \nonumber\\
    & \nabla \cdot \mathbf{u} = 0. \nonumber
\end{align}
We set the parameters in \eqref{pnp:new} to be $\varepsilon = 1$ and {\color{blue}$\kappa = 10000$}, with the initial data given by
\begin{equation*}
    \left\{
    \begin{aligned}
         p(x, y, 0) &= 1 + 10^{-6} - \tanh\left(2\big( (x - 0.8 \pi)^2 + (y - 0.8 \pi)^2 - (0.2 \pi)^2 \big)\right), \\
         n(x, y, 0) &= 1 + 10^{-6} - \tanh\left(2\big( (x - 1.2 \pi)^2 + (y - 1.2 \pi)^2 - (0.2 \pi)^2 \big)\right), \\
         \mathbf{u}(x, y, 0) &=
             \begin{pmatrix}
                 0 \\
                 0
             \end{pmatrix}.
    \end{aligned}
    \right.
\end{equation*}
The initial condition indicates that the positive and negative ions accumulate in two regions centered at $(0.8 \pi, 0.8 \pi)$ and $(1.2 \pi, 1.2 \pi)$, respectively.

With time step $\Delta t = 10^{-4}$, in Figure~\ref{fig:PNP_snap}, we plot the profiles of $p - n$ and the velocity field $\mathbf{u}$ at times $T = 0.005$, $0.025$, $0.05$, $0.075$, $0.1$, and $1$. We observe that the positive and negative ions move toward each other and drag the fluid along with them. Later, the outflowing fluid between them prevents the ions from approaching each other further and carries the ions toward the corners. At the end of the computation, the fluid becomes almost electro-neutral.

\begin{figure}[h!]
    \centering
    \includegraphics[width=0.3\textwidth]{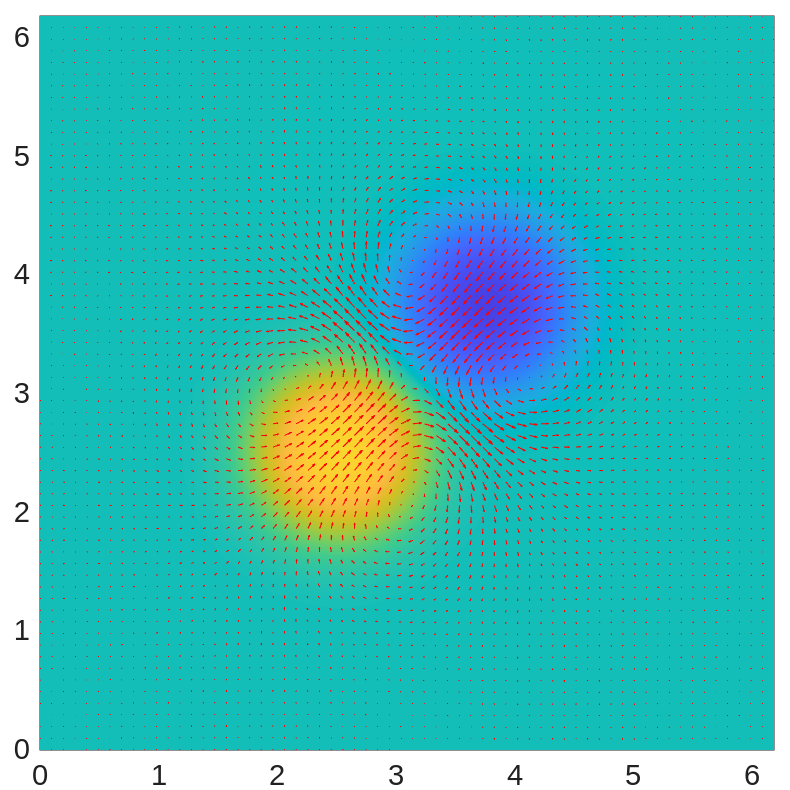}
    \includegraphics[width=0.3\textwidth]{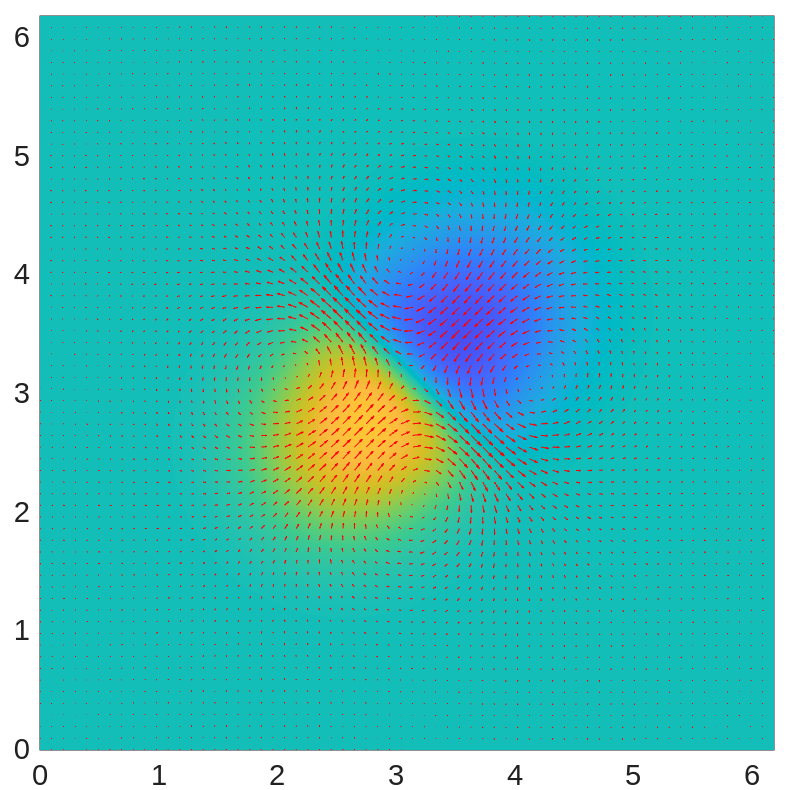}
    \includegraphics[width=0.3\textwidth]{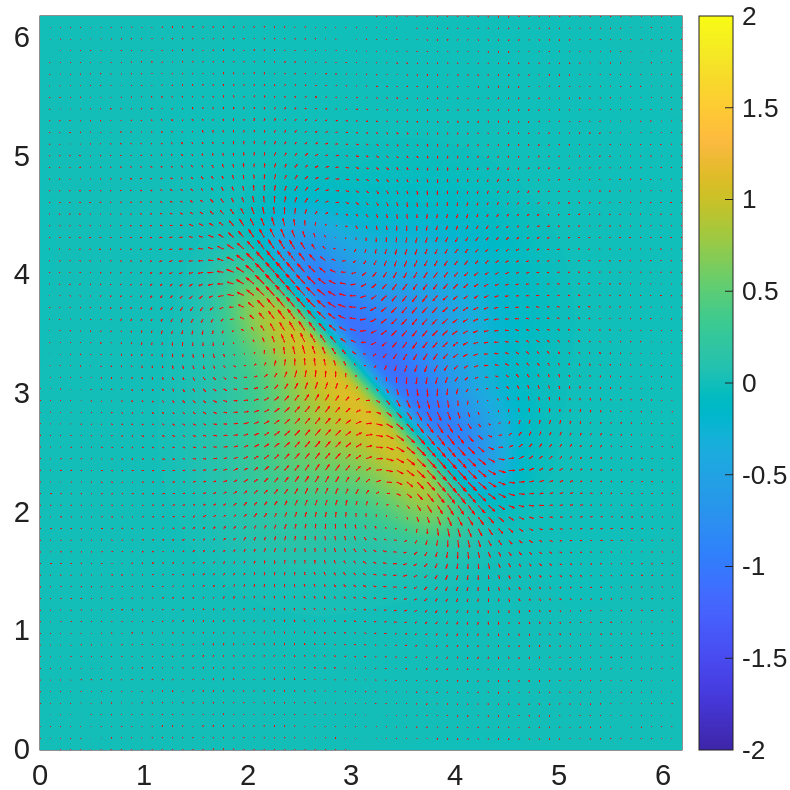}
    \includegraphics[width=0.3\textwidth]{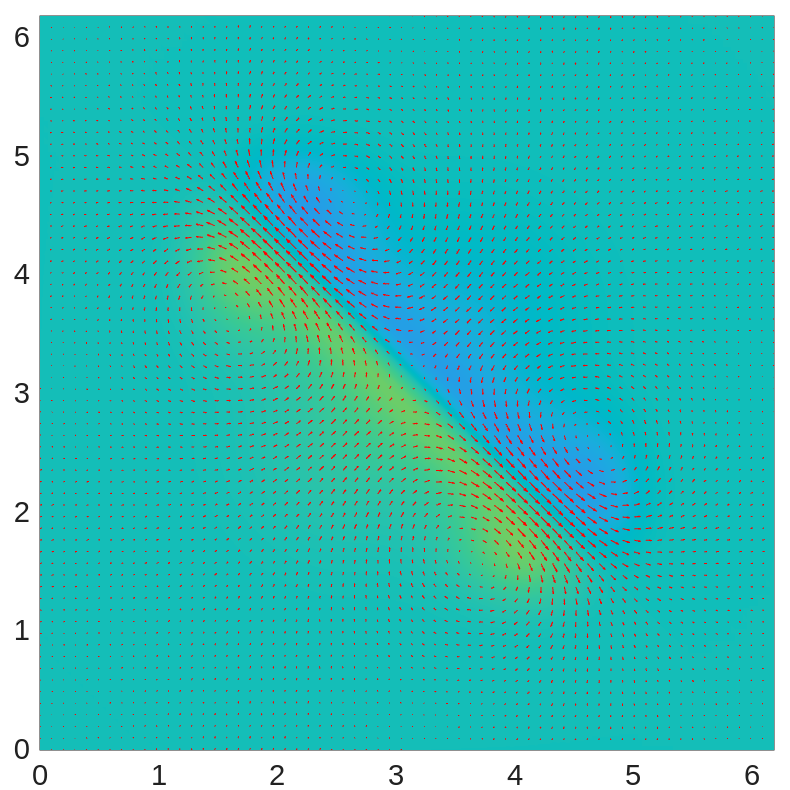}
    \includegraphics[width=0.3\textwidth]{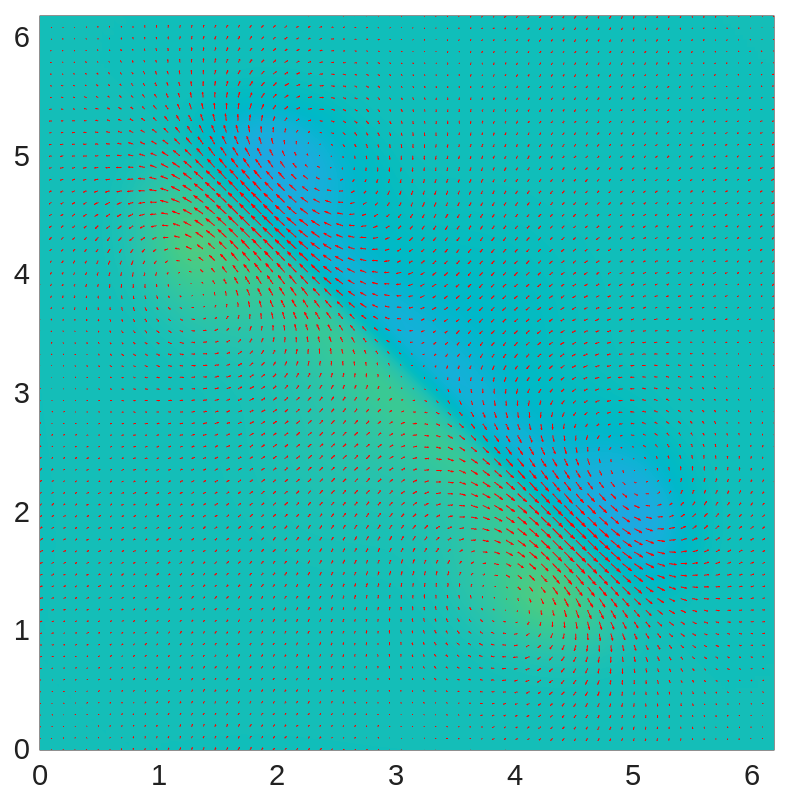}
    \includegraphics[width=0.3\textwidth]{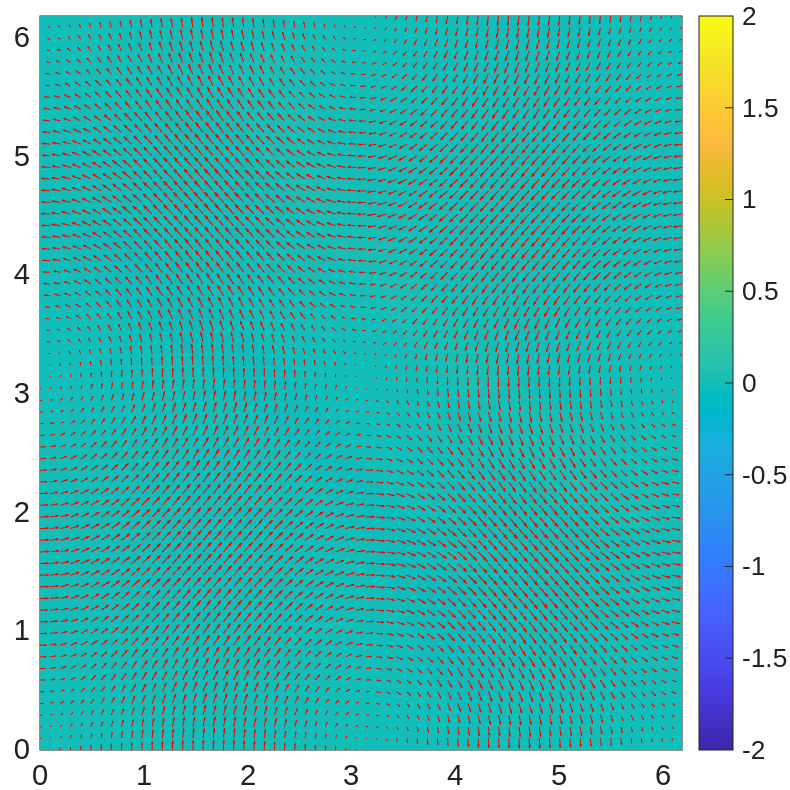}
    \caption{ {\color{blue} Snapshots of $p - n$ and velocity field $\mathbf{u}$ at times $T = 0.005$, $0.025$, $0.05$, $0.075$, $0.1$, and $1$. } }
    \label{fig:PNP_snap}
\end{figure}

We also examine the energy dissipation of the system in Figure~\ref{fig:PNPNS_property}(left), where the system energy is shown to be dissipative as we have proved. We plot the mass change for positive and negative ions in Figure~\ref{fig:PNPNS_property}(middle), showing that the mass of ions is preserved within machine precision. We also plot the minimum and maximum of $(p, n)$ in Figure~\ref{fig:PNPNS_property}(right), demonstrating that the ionic concentrations remain positive throughout the simulation.

\begin{figure}[ht]
    \centering
    \includegraphics[width=0.3\textwidth]{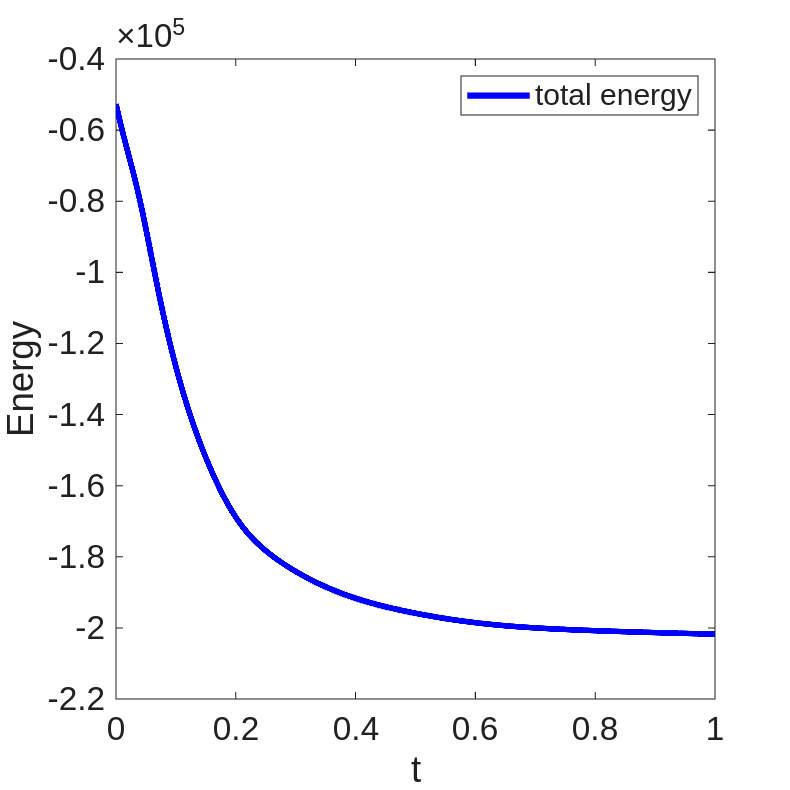}
    \includegraphics[width=0.3\textwidth]{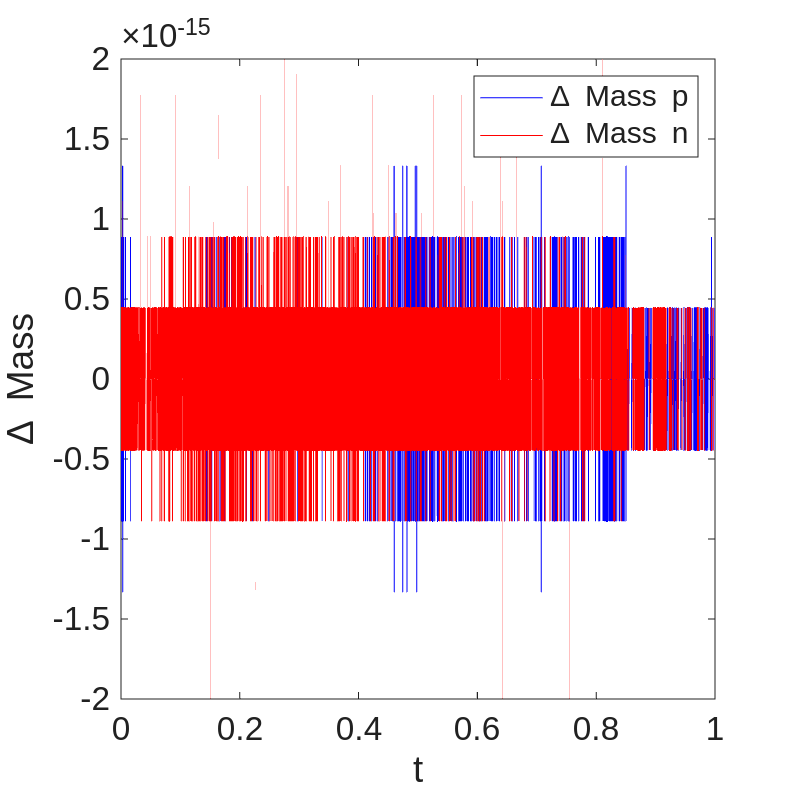}
    \includegraphics[width=0.3\textwidth]{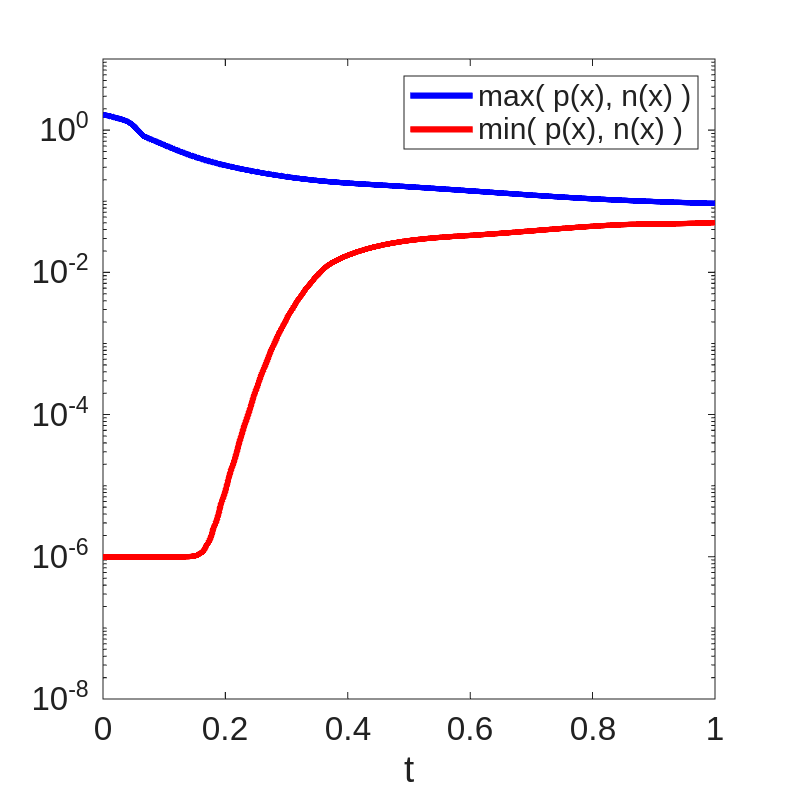}
    \caption{ {\color{blue} Left: Total energy of the PNP-NS system. Middle: Change of mass for $(p, n)$. Right: Lower and upper bounds of $(p, n)$.} }
    \label{fig:PNPNS_property}
\end{figure}

\section{Concluding Remarks}

In this paper, we mainly consider numerical approximations for  the PNP-NS system. Firstly, we give the results of  unique solvability and regularity for the solution of  PNP-NS system with suitable assumptions on initial conditions. To efficiently solve this coupled system, we propose a decoupled, mass-conserving, positivity-preserving and energy stable scheme which can also be unique solvable.  Furthermore, we also carry out a rigorous error analysis for the fully discretized scheme, and derive optimal convergence results. The error analysis mainly depends on the $L^{\infty}$ bounds for  the numerical solutions $n$ and $p$, which are obtained by using a high-order asymptotic expansion for the PNP-NS system combing  with a mathematical induction technique. We also present some numerical examples to validate  the accuracy and stability of our decoupled scheme.

\appendix
\section{Appendix}

\subsection{High order correction}

\begin{lemma}
    \label{appendix:hoa}
	Let $(p, n , \uu)$ be the solution of the PNP-NS system \eqref{gov_equ_1}-\eqref{gov_equ_5}  which satisfies the following properties:
	\begin{enumerate}
		\item The ionic concentrations are strictly positive
		    \begin{equation*}
		            p, n \geq \delta_0 > 0,
                \end{equation*}
		\item The solution satisfies
		      \begin{equation*}
		        \begin{array}{ll}
		            & (\pt^{4} p, \pt^{4} n, \pt^{4}\uu) \in L^{\infty}(0, T; L^{2}(\Omega)), (\pt^{3} p, \pt^{3} n, \pt^{3} \uu) \in L^{\infty}(0, T; H^{k+1}(\Omega)) \ (k \ge 2),
		        \end{array}
		      \end{equation*}
	\end{enumerate}
	then we can construct
	{ correction
	 functions $(p_{\Delta t, i}, n_{\Delta t, i}, \uu_{\Delta t, i}, \phi_{\Delta t, i})(i = 1,2)$
	 depending only on $(p, n, \uu, \psi)$
	 } such that
	  the 
	 { supplementary fields}
	  $(\up, \un, \uuu, \uphi, \umu, \unu, \upsi)$ (defined by \eqref{lem:hoa_vardef}) has higher order consistency truncation error(as defined in \eqref{hoa:p1}-\eqref{hoa:u2}):
	\begin{equation*}
		\rvert \utau^{m+1}_{p}(v_N)\rvert , \rvert \utau^{m+1}_{n}(v_N)\rvert , \rvert \utau^{m+1}_{\uu}(v_N)\rvert  \leq C(\Delta t^3 + N^{-k}) \|  v_N \| _{H^{1}}. 
	\end{equation*}
	Moreover, with $\Delta t, \frac{1}{N}$ chosen small enough, we have 
	\begin{enumerate}
		\item The { supplementary} functions are strictly positive 
		    \begin{equation*}
		    \label{appendix:mod_pos}
		        \up, \un \ge \delta_0^{*} > 0,
		    \end{equation*}
		\item The { supplementary} functions satisfy
		    \begin{equation*}
		        \label{appendix:mod_reg}
		        (\up, \un, \uuu) \in L^{\infty}(0, T, W^{1, \infty}).
		    \end{equation*}
	\end{enumerate}
\end{lemma}
\begin{proof}
Let $(p^{m}, n^{m}, \uu^{m}, \phi^{m})$ be the $L^2$-orthogonal projection of continuous solution $(p, n, \uu, \phi)(m\Delta t)$ onto $X_N \times X_N \times X_N^2 \times X_N$, as defined in \eqref{lem:ea_var_def}. 
From Taylor expansion, the local truncation error may be written into two parts, time discretization error and spatial discretization error, we have 
\begin{align}
	& \langle \frac{p^{m+1}-p^{m}}{\Delta t}, v_N \rangle-\langle p^{m} \uu^{m}, \nabla v_N \rangle + \langle p^{m}(1 + 2\Delta t p^{m}) \nabla \mu^{m+1}, \nabla v_N \rangle
		\nonumber \\
	& \qquad = -\langle \Delta t f^{m+1}_{p,1} + \Delta t^2 f^{m+1}_{p, 2} + \bigO(\Delta t^3) + g^{m+1}_{p}, v_N \rangle,
		\label{lem:hoa_o_p} \\
	& \langle \frac{n^{m+1}-n^{m}}{\Delta t}, v_N \rangle-\langle n^{m} \uu^{m}, \nabla v_N \rangle + \langle n^{m}(1 + 2\Delta t n^{m}) \nabla \nu^{m+1}, \nabla v_N \rangle
		\nonumber \\
	& \qquad = -\langle \Delta t f^{m+1}_{n,1} + \Delta t^2 f^{m+1}_{n, 2} + \bigO(\Delta t^3) + g^{m+1}_{n}, v_N \rangle,
		\label{lem:hoa_o_n} \\
	& \langle \frac{R_N\uu^{m+1}-\uu^{m}}{\Delta t}, v_N \rangle +  \langle (\uu^{m} \cdot \nabla)R_N\uu^{m+1}, v_N \rangle + \langle \nabla R_N\uu^{m+1}, \nabla v_N \rangle + \langle \nabla \phi^{m}, v_N \rangle \nonumber \\
	& \quad + \langle p^{m} \nabla \mu^{m+1} + n^{m} \nabla \nu^{m+1}, v_N \rangle 
	  =	-\langle \Delta t f^{m+1}_{\uu,1} + \Delta t^2 f^{m+1}_{\uu, 2} + \bigO(\Delta t^3) + g^{m+1}_{\uu}, v_N \rangle,
		\label{lem:hoa_o_u}
\end{align}
where $(f_{p, i}^{m+1}, f_{n, i}^{m+1}, f_{\uu, i}^{m+1})_{i = 1,2}$ are the temporal part of truncation error and $(g_p^{m+1}, g_n^{m+1}, g_{\uu}^{m+1})$ are the spatial part of the truncation error. From Taylor expansion, we can compute 
\begin{align*}
     f_{p, 1} &= \frac{1}{2}\ptt p + \divergence(\pt(p \uu))-\divergence((\pt p-2 p^2)  \nabla (\ln{p} + \psi)),  \\
     f_{n, 1} &= \frac{1}{2}\ptt n + \divergence(\pt(n \uu))-\divergence((\pt n-2 n^2)  \nabla (\ln{n}-\psi)), \\
     f_{\uu, 1} &= \frac{1}{2} \ptt \uu + \pt \uu \cdot \nabla \uu + \pt p \nabla (\ln{p} + \psi) + \pt n \nabla (\ln{n}-\psi); 
\end{align*}
\begin{align*}
    f_{p, 2} &= -\frac{1}{6} \pttt p-\frac{1}{2} \divergence (\ptt(p \uu)) + \divergence \big( (\frac{1}{2} \ptt p-4 p \pt p) \nabla(\ln{p} + \psi) \big) \\
    f_{n, 2} &= -\frac{1}{6} \pttt n-\frac{1}{2} \divergence (\ptt(n \uu)) + \divergence \big( (\frac{1}{2} \ptt n-4 n \pt n) \nabla(\ln{n}-\psi) \big) \\
    f_{\uu, 2} &= -\frac{1}{6} \pttt \uu-\frac{1}{2} \ptt \uu \cdot \nabla \uu-(\uu \cdot \nabla) \nabla \pt \phi-\Delta (\nabla \pt \phi) \\
            & \quad-\frac{1}{2} \ptt p \nabla (\ln{p} + \psi)-\frac{1}{2} \ptt n \nabla (\ln{n}-\psi)
\end{align*}
and 
\begin{align*}
    \langle g_{p}^{m+1}, v_N \rangle &\lesssim N^{-k} (\| p_t \|_{H^{k}} + \| p \uu \|_{H^{k}} +  \| p \nabla (\ln{p} + \psi) \|_{H^{k}}) \big((m+1)\Delta t\big) \| v_N \|_{H^{1}}, \\ 
    \langle g_{n}^{m+1}, v_N \rangle &\lesssim N^{-k} (\| n_t \|_{H^{k}}  + \| n \uu \|_{H^{k}} +  \| n \nabla (\ln{n}-\psi) \|_{H^{k}}) 
    \big((m+1)\Delta t\big) \| v_N \|_{H^{1}}, \\
    \langle g_{\uu}^{m+1}, v_N \rangle &\lesssim N^{-k} \big( \| \uu_t \|_{H^{k}} + \| (\uu \cdot \nabla) \uu \|_{H^{k}} + \| \nabla \uu \|_{H^{k}} + \| \nabla \psi \|_{H^{k}} \\
    & \quad \quad \quad + \| p\nabla (\ln{p} + \psi) + n \nabla(\ln{n}-\psi \|_{H^{k}}  \big) \big((m+1)\Delta t\big) \| v_N \|_{H^{1}}.
\end{align*}
Applying the regularity assumption \eqref{hoa:reg}, we have
\begin{equation}
    \label{lem:hoa_reg_f}
    \begin{split}
        & \pt^2 f_{p,1} \in L^{\infty}(0, T;L^{2}(\Omega)), \ \pt f_{p,1} \in L^{\infty}(0, T; H^{k+1}(\Omega)), \\
        &\pt^2 f_{n,1} \in L^{\infty}(0, T;L^{2}(\Omega)), \ \pt f_{n,1} \in L^{\infty}(0, T; H^{k+1}(\Omega)), \\
        &\pt^2 f_{\uu,1} \in L^{\infty}(0, T;L^{2}(\Omega)), \ \pt f_{\uu,1} \in L^{\infty}(0, T; H^{k+1}(\Omega)),
    \end{split}
\end{equation}
\begin{equation*}
    \langle g^{m+1}_{p, 1}, v_N \rangle, \langle g^{m+1}_{n, 1}, v_N \rangle, \langle g^{m+1}_{\uu, 1}, v_N \rangle
    \lesssim 
    N^{-k} \|  v_N \| _{H^{1}}.
\end{equation*}
With those $(f_{p, 1}, f_{n, 1}, f_{\uu, 1})$, we construct and solve the leading order temporal correction function $(p_{\Delta t, 1}, n_{\Delta t, 1}, \uu_{\Delta t, 1}, \phi_{\Delta t, 1})$ from the following equation: 
\begin{align}
	 \pt p_{\Delta t, 1} 
		& = \divergence(p \nabla (\frac{p_{\Delta t, 1}}{p} + \psi_{\Delta t, 1}) + p_{\Delta t, 1} \nabla (\ln{p} + \psi)) \nonumber \\
		&\quad-\divergence(p_{\Delta t, 1} \uu + p \uu_{\Delta t, 1}) -f_{p,1}, \label{lem:hoa_pf_cont_p1} \\ 
	 \pt n_{\Delta t, 1} 
		& = \divergence(n \nabla (\frac{n_{\Delta t, 1}}{n}-\psi_{\Delta t, 1}) + n_{\Delta t, 1} \nabla (\ln{n}-\psi)) \nonumber \\
		&\quad -\divergence(n_{\Delta t, 1} \uu + n \uu_{\Delta t, 1}) -f_{n,1}, \label{lem:hoa_pf_cont_n1} \\ 
	-\Delta \psi_{\Delta t, 1} & = p_{\Delta t, 1}-n_{\Delta t, 1}, \label{lem:hoa_pf_cont_psi1} \\ 
         \pt \uu_{\Delta t, 1}
		& = \Delta \uu_{\Delta t, 1}-\nabla \phi_{\Delta t, 1}-(\uu \cdot \nabla) \uu_{\Delta t, 1}-(\uu_{\Delta t, 1} \cdot \nabla) \uu \nonumber \\
		&\qquad-p\nabla(\frac{p_{\Delta t, 1}}{p} + \psi_{\Delta t, 1})-p_{\Delta t, 1} \nabla(\ln{p} + \psi) \nonumber \\
		&\qquad-n\nabla(\frac{n_{\Delta t, 1}}{n}-\psi_{\Delta t, 1})-n_{\Delta t, 1} \nabla(\ln{n}-\psi)-f_{\uu, 1}, \label{lem:hoa_pf_cont_u1} \\
	 \divergence \uu_{\Delta t, 1} &= 0, \label{lem:hoa_pf_cont_div1}
\end{align}
subject to the periodic boundary condition and zero initial condition. 
The PDE system \eqref{lem:hoa_pf_cont_p1}-\eqref{lem:hoa_pf_cont_div1} is very similar to the PNP-NS system \eqref{gov_equ_1}-\eqref{gov_equ_5}, and the existence of solution could be established similarly. Moreover, given the regularity of $(p, n, \uu, \phi)$ and $(f_{p, 1}, f_{n, 1}, f_{\uu, 1})$ in \eqref{lem:hoa_reg_f}, the solution satisfies
\begin{equation}
	\label{hoa:func_mod1_regularity}
	(\pt^3 p_{\Delta t, 1}, \pt^3 n_{\Delta t, 1}, \pt^3 \uu_{\Delta t, 1}) \in L^{\infty}(0, T, L^{2}(\Omega)), \ 
	(\pt^2 p_{\Delta t, 1}, \pt^2 n_{\Delta t, 1}, \pt^2 \uu_{\Delta t, 1}) \in L^{\infty}(0, T, H^{k+1}(\Omega)).
\end{equation}
The discretization of the above system implies that 
\begin{align}
	\langle -f^{m+1}_{p,1}, v_N \rangle  
		& = \langle \frac{p^{m+1}_{\Delta t, 1}-p^{m}_{\Delta t, 1}}{\Delta t}, v_N \rangle  
			- \langle p^{m}_{\Delta t, 1} \uu^{m} + p^{m} \uu^{m}_{\Delta t, 1}, \nabla v_N \rangle 
			\nonumber \\
		& \quad + \langle p^{m}(1 + 2\Delta t p^{m}) \nabla (\frac{p^{m+1}_{\Delta t, 1}}{p^{m+1}} + \psi^{m+1}_{\Delta t, 1}) \nonumber\\
		& \quad	\quad + p^{m}_{\Delta t, 1}(1 + 2\Delta t p^{m}) \nabla (\ln{p^{m+1} + \psi^{m+1}}), \nabla v_N \rangle \nonumber  \\
		& \quad	- \langle \Delta t f_{p_{\Delta t, 1}, 1}^{m+1} + g_{p_{\Delta t, 1}}^{m+1}+ \bigO(\Delta t^2) , v_N \rangle   \label{lem:hoa_pf_solve_p1},	\\
	\langle -f^{m+1}_{n,1}, v_N \rangle 
		& = \langle \frac{n^{m+1}_{\Delta t, 1}-n^{m}_{\Delta t, 1}}{\Delta t}, v_N \rangle
			- \langle n^{m}_{\Delta t, 1} \uu^{m} + n^{m} \uu^{m}_{\Delta t, 1}, \nabla v_N \rangle \nonumber \\
		& \quad + \langle n^{m}(1 + 2\Delta t n^{m}) \nabla (\frac{n^{m+1}_{\Delta t, 1}}{n^{m+1}}-\psi^{m+1}_{\Delta t, 1}) \nonumber\\
		& \quad \quad	+ n^{m}_{\Delta t, 1}(1 + 2\Delta t n^{m}) \nabla (\ln{n^{m+1}-\psi^{m+1}}), \nabla v_N \rangle \nonumber \\
		& \quad	-\langle  \Delta t f_{n_{\Delta t, 1}, 1}^{m+1} + g_{n_{\Delta t, 1}}^{m+1} + \bigO(\Delta t^2) , v_N \rangle  \label{lem:hoa_pf_solve_n1}, \\
	\langle -f^{m+1}_{\uu, 1}, v_N \rangle 
		&= \langle \frac{R_N\uu^{m+1}_{\Delta t, 1}-\uu^{m}_{\Delta t, 1}}{\Delta t}, v_N \rangle 
			+ \langle \nabla R_N\uu^{m+1}_{\Delta t, 1} , \nabla v_N \rangle 
				+ \langle \nabla \phi^{m}_{\Delta t, 1}, v_N \rangle \nonumber  \\
		&\quad + \langle (\uu^{m} \cdot \nabla) R_N\uu^{m+1}_{\Delta t, 1} + (\uu^{m}_{\Delta t, 1} \cdot \nabla) R_N\uu^{m+1}, v_N \rangle  \nonumber  \\
		&\quad	+ \langle p^{m} \nabla (\frac{p^{m+1}_{\Delta t, 1}}{p^{m+1}} + \psi^{m+1}_{\Delta t, 1})
							+ p^{m}_{\Delta t, 1} \nabla (\ln{p^{m+1}} + \psi^{m+1}), v_N \rangle \nonumber  \\
		&\quad + \langle n^{m} \nabla (\frac{n^{m+1}_{\Delta t, 1}}{n^{m+1}}-\psi^{m+1}_{\Delta t, 1})
							+ n^{m}_{\Delta t, 1} \nabla (\ln{n^{m+1}}-\psi^{m+1}), v_N \rangle \nonumber \\
		& \quad	- \langle  \Delta t f_{\uu_{\Delta t, 1}, 1}^{m+1} +g_{\uu_{\Delta t, 1}}^{m+1} +  \bigO(\Delta t^2), v_N \rangle \label{lem:hoa_pf_solve_u1}, \\
	\langle \nabla \psi^{m}_{\Delta t, 1}, \nabla v_N & \rangle 
		 = \langle p^{m}_{\Delta t, 1}-n^{m}_{\Delta t, 1}, v_N \rangle,  \\
	\langle \uu^{m}_{\Delta t, 1}, \nabla v_N & \rangle = 0.
\end{align}
where $(f_{p_{\Delta t, 1}, 1}, f_{n_{\Delta t, 1}, 1}, f_{\uu_{\Delta t, 1}, 1})$ and $(g_{p_{\Delta t, 1}}, g_{n_{\Delta t, 1}}, g_{\uu_{\Delta t, 1}})$ are the temporal part and spatial part of the truncation error, from Taylor expansion, we have 
\begin{align*}
    f_{p_{\Delta t, 1}, 1} &=  \frac{1}{2} \ptt p_{\Delta t, 1} + \divergence \big( \pt(p_{\Delta t,1} \uu + p \uu_{\Delta t,1}) \big)-\divergence \big( (\pt p-2 p^2) \nabla(\frac{p_{\Delta t,1}}{p} + \psi_{\Delta t,1}) \big) \nonumber \\
        & \quad-\divergence \big( (\pt p-4 p_{\Delta t, 1} p) \nabla(\ln{p} + \psi) \big) \\
    f_{n_{\Delta t, 1}, 1} &= \frac{1}{2} \ptt n_{\Delta t, 1} + \divergence \big( \pt(n_{\Delta t, 1} \uu + n \uu_{\Delta t, 1}) \big)-\divergence \big( (\pt n-2n^2) \nabla(\frac{n_{\Delta t, 1}}{n}-\psi_{\Delta t, 1}) \big) \nonumber \\
        & \quad-\divergence \big( (\pt n-4 n_{\Delta t, 1} n) \nabla (\ln{n}-\psi) \big) \\ 
    f_{\uu_{\Delta t, 1}, 1} &= \frac{1}{2} \ptt \uu_{\Delta t, 1} +
         (\pt \uu \cdot \nabla) \uu_{\Delta t, 1} +
         (\pt \uu_{\Delta t, 1} \cdot \nabla) \uu \\
        &\quad + \pt p \nabla(\frac{p_{\Delta t, 1}}{p} + \psi_{\Delta t, 1}) + \pt p_{\Delta t, 1} \nabla (\ln{p + \psi}) \nonumber \\
        &\quad + \pt n \nabla (\frac{n_{\Delta t, 1}}{n}-\psi_{\Delta t, 1}) + \pt n_{\Delta t, 1} \nabla (\ln{n-\psi}) 
\end{align*}
and 
\begin{equation*}
    \begin{array}{ll}
         \langle g_{p, \Delta t, 1}^{m+1}, v_N \rangle & \lesssim N^{-k} \big( \| \pt p_{\Delta t, 1}  \|_{L^{\infty}_{t}H^{k}} + \| p_{\Delta t, 1}\|_{L^{\infty}_{t} H^{k+1}} + \| p \nabla \psi_{\Delta t, 1} \|_{L^{\infty}_{t} H^{k}} \vspace{0.3em} \\
            & \quad + \| p_{\Delta t, 1} \nabla p \|_{L^{\infty}_{t} H^{k}} + \| f_{p, 1} \|_{L^{\infty}_{t} H^{k}}  \big) \| v_N \|_{H^{1}} \vspace{0.5em}, \\
         \langle g_{n, \Delta t, 1}^{m+1}, v_N \rangle & \lesssim N^{-k} \big( \| \pt n_{\Delta t, 1} \|_{L^{\infty}_{t}H^{k}} + \| n_{\Delta t, 1}\|_{L^{\infty}_{t} H^{k+1}} + \| n \nabla \psi_{\Delta t, 1} \|_{L^{\infty}_{t} H^{k}} \vspace{0.3em} \\
            & \quad + \| p_{\Delta t, 1} \nabla p \|_{L^{\infty}_{t} H^{k}} + \| f_{p, 1} \|_{L^{\infty}_{t} H^{k}}  \big) \| v_N \|_{H^{1}} \vspace{0.5em}, \\
         \langle g_{\uu, \Delta t, 1}^{m+1}, v_N \rangle & \lesssim N^{-k} \big( \| \pt \uu_{\Delta t, 1} \|_{L^{\infty}_{t}H^{k}} + \| \nabla \uu_{\Delta t, 1} \|_{L^{\infty}_{t}H^{k}} + \| \nabla \psi_{\Delta t, 1} \|_{L^{\infty}_{t}H^{k}} \vspace{0.3em} \\
            & \quad + \| (\uu_{\Delta t, 1} \cdot \nabla) \uu \|_{L^{\infty}_{t}H^{k}}
            + \| (\uu \cdot \nabla) \uu_{\Delta t, 1} \|_{L^{\infty}_{t}H^{k}}
            + \| p \nabla \psi_{\Delta t, 1} \|_{H^{k}} + \| p_{\Delta t, 1} \nabla \psi_{\Delta t, 1} \|_{L^{\infty}_{t}H^{k}} \vspace{0.3em} \\
            & \quad 
            + \| n \nabla \psi_{\Delta t, 1} \|_{L^{\infty}_{t}H^{k}} + \| n_{\Delta t, 1} \nabla \psi_{\Delta t, 1} \|_{L^{\infty}_{t}H^{k}} + \| f_{\uu, 1} \|_{L^{\infty}_{t}H^{k}} \big) \| v_N \|_{H^{1}}.
    \end{array}
\end{equation*}
From the regularity result in \eqref{lem:hoa_reg_f} and \eqref{hoa:func_mod1_regularity}, we have 
\begin{equation*}
    \langle g^{m+1}_{p, 2}, v_N \rangle, \langle g^{m+1}_{n, 2}, v_N \rangle, \langle g^{m+1}_{\uu, 2}, v_N \rangle
    \lesssim 
    N^{-k} \|  v_N \| _{H^{1}}.
\end{equation*}
Combining \eqref{lem:hoa_o_p}-\eqref{lem:hoa_o_u} and \eqref{lem:hoa_pf_solve_p1}-\eqref{lem:hoa_pf_solve_u1} leads to the second order temporal local truncation error for $\up_{1} = \projN (p + \Delta t p_{\Delta t, 1}), \ \un_{1} = \projN (n + \Delta t n_{\Delta t, 1}), \ \uuu_{1} = \projN (\uu + \Delta t \uu_{\Delta t, 1}), \ \uphi_{1} = \projN (\phi + \Delta t \phi_{\Delta t, 1})$:
\begin{align}
	& \langle \frac{\up_{1}^{m+1}-\up_{1}^{m}}{\Delta t}, v_N \rangle-\langle \up_{1}^{m} \uuu_{1}^{m}, \nabla v_N \rangle + \langle \up_{1}^{m}(1 + 2\Delta t \up_{1}^{m}) \nabla \umu_{1}^{m+1}, \nabla v_N \rangle 
		\nonumber \\
	& \qquad = -\langle \Delta t^2 f^{m+1}_{\up_{1},2} + \bigO(\Delta t^3) + \bigO(N^{-k}), v_N \rangle 
		\label{lem:hoa_o_p_1}, \\
	& \langle \frac{\un_{1}^{m+1}-\un_{1}^{m}}{\Delta t}, v_N \rangle-\langle \un_{1}^{m} \uuu_{1}^{m}, \nabla v_N \rangle + \langle \un_{1}^{m}(1 + 2\Delta t \un_{1}^{m}) \nabla \unu_{1}^{m+1}, \nabla v_N \rangle 
		\nonumber \\
	& \qquad = -\langle \Delta t^2 f^{m+1}_{\un_{1},2} + \bigO(\Delta t^3) + \bigO(N^{-k}), v_N \rangle 
		\label{lem:hoa_o_n_1}, \\
	& \langle \frac{R_N\uuu_{1}^{m+1}-\uuu_{1}^{m}}{\Delta t}, v_N \rangle +  \langle (\uuu_{1}^{m} \cdot \nabla)R_N\uuu_{1}^{m+1}, v_N \rangle + \langle \nabla R_N\uuu_{1}^{m+1}, \nabla v_N \rangle + \langle \nabla \uphi_{1}^{m}, v_N \rangle \nonumber \\
	& + \langle \up_{1}^{m} \nabla \umu_{1}^{m+1} + \un_{1}^{m} \nabla \unu_{1}^{m+1}, v_N \rangle 
	  =	-\langle \Delta t^2 f^{m+1}_{\uuu_{1},2} + \bigO(\Delta t^3) + \bigO(N^{-k}), v_N \rangle
		\label{lem:hoa_o_u_1},
\end{align}
where 
\begin{align*}
	& \upsi_{1} = \projN [(-\Delta)^{-1}(\up_{1}-\un_{1})], \\
	& \umu_{1} = \projN (\ln{\up_{1}} + \upsi_{1}), \ \unu_{1} = \projN (\ln{\un_{1}}-\upsi_{1}),
\end{align*}
and 
\begin{align*}
    f_{\up_{1}, 2} &= f_{p, 2} + f_{p_{\Delta t, 1}, 1} + \divergence(p_{\Delta t, 1} \uu_{\Delta t, 1})-\divergence(2 p p_{\Delta t, 1} \nabla (\ln{p} + \psi)) \nonumber \\
    & \quad-\divergence(p \nabla (\frac{p_{\Delta t, 1}}{p})^2 ) 
    + \divergence (p_{\Delta t, 1} \nabla (\frac{p_{\Delta t, 1}}{p} + \psi)), \\
    f_{\un_{1}, 2} &= f_{n, 2} + f_{n_{\Delta t, 1}, 1} + \divergence(n_{\Delta t, 1} \uu_{\Delta t, 1})-\divergence(2 n n_{\Delta t, 1} \nabla (\ln{n}-\psi)) \nonumber \\
    & \quad-\divergence(n \nabla (\frac{n_{\Delta t, 1}}{n})^2 ) 
   -\divergence (n_{\Delta t, 1} \nabla (\frac{n_{\Delta t, 1}}{n}-\psi)), \\
    f_{\uuu_{1}, 2} &= f_{\uu, 2} + f_{\uu_{\Delta t, 1}, 1} + (\uu_{\Delta t, 1} \cdot \nabla) \uu_{\Delta t, 1} \nonumber \\
        &\quad + p_{\Delta t, 1} \nabla(\frac{p_{\Delta t, 1}}{p} + \psi_{\Delta t, 1}) 
       -p\nabla ((\frac{p_{\Delta t, 1}}{p})^2) \nonumber \\
        &\quad + n_{\Delta t, 1} \nabla (\frac{n_{\Delta t, 1}}{n}-\psi_{\Delta t, 1})-n \nabla ((\frac{n_{\Delta t, 1}}{n})^2).
\end{align*}
Since $(p_{\Delta t, 1}, n_{\Delta t, 1})$ are bounded, we may choose $\Delta t, \frac{1}{N}$ so small that $\up_1, \un_1 > \frac{\delta_0}{2} > 0$. 
And $(f_{\up_{1},2}^{m+1}, f_{\un_{1}, 2}^{m+1}, f_{\uuu_{1},2}^{m+1})$ are the temporal projection of functions $(f_{\up_{1}, 2}, f_{\un_{1}, 2}, f_{\uuu_{1},2})$ onto $X_N \times X_N \times X_N^2$.
From \eqref{hoa:reg} \eqref{hoa:func_mod1_regularity} we have
\begin{equation*}
	(\pt f_{\up_{1},2}, \pt f_{\un_{1},2}, \pt f_{\uuu_{1},2}) \in L^{\infty}(0, T; L^{2}(\Omega)),\  (f_{\up_{1}, 2}, f_{\un_{1},2}, f_{\uuu_{1}, 2}) \in L^{\infty}(0, T; H^{k+1}(\Omega)).
\end{equation*}
Similarly, the next order temporal correction function $(p_{\Delta t, 2}, n_{\Delta t, 2}, \uu_{\Delta t, 2}, \phi_{\Delta t, 2})$ is given by the following system: 
\begin{align}
	& \pt p_{\Delta t, 2} 
		= \divergence(\up_{1} \nabla (\frac{p_{\Delta t, 2}}{\up_{1}} + \psi_{\Delta t, 2}) + p_{\Delta t, 2} \nabla (\ln{\up_{1}} + \upsi_{1})) \nonumber \\
		&\qquad -\divergence(p_{\Delta t, 2} \uuu_{1} + \up_{1} \uu_{\Delta t, 2})-f_{\up_{1},2} \label{lem:hoa_pf_cont_p2}, \\ 
	& \pt n_{\Delta t, 2} 
		= \divergence(\un_{1} \nabla (\frac{n_{\Delta t, 2}}{\un_{1}}-\psi_{\Delta t, 2}) + n_{\Delta t, 2} \nabla (\ln{\un_{1}}-\upsi_{1})) \nonumber \\
		&\qquad -\divergence(n_{\Delta t, 2} \uuu_{1} + \un_{1} \uu_{\Delta t, 2})-f_{\un_{1},2} \label{lem:hoa_pf_cont_n2}, \\ 
	& -\Delta \psi_{\Delta t, 2} = p_{\Delta t, 2}-n_{\Delta t, 2} \label{lem:hoa_pf_cont_psi2}, \\ 
	& \pt \uu_{\Delta t, 2}
		= \Delta \uu_{\Delta t, 2}-\nabla \phi_{\Delta t, 2}-(\uuu_{1} \cdot \nabla) \uu_{\Delta t, 2}-(\uu_{\Delta t, 2} \cdot \nabla) \uuu_{1} \nonumber \\
		&\qquad-\up_{1} \nabla(\frac{p_{\Delta t, 2}}{\up_{1}} + \psi_{\Delta t, 2})-p_{\Delta t, 2} \nabla(\ln{\up_{1}} + \upsi_{1}) \nonumber \\
		&\qquad-\un_{1} \nabla(\frac{n_{\Delta t, 2}}{\un_{1}}-\psi_{\Delta t, 2})-n_{\Delta t, 2} \nabla(\ln{\un_{1}}-\upsi_{1})-f_{\uuu_{1}, 2} \label{lem:hoa_pf_cont_u2}, \\
	& \divergence \uuu_{\Delta t, 2} = 0 \label{lem:hoa_pf_cont_div2}.
\end{align}
subject to the periodic boundary condition and zero initial condition.
Then we have 
\begin{equation*}
	(\pt^2 p_{\Delta t, 2}, \pt^2 n_{\Delta t, 2}, \pt^2 \uu_{\Delta t, 2}) \in L^{\infty}(0, T, L^{2}(\Omega)), \ (\pt p_{\Delta t, 2}, \pt n_{\Delta t, 2}, \pt \uu_{\Delta t, 2}) \in L^{\infty}(0, T, H^{k+1}(\Omega)).
\end{equation*}
The discretization of the above system implies that 
\begin{align}
	\langle -f^{m+1}_{\up_{1},2}, v_N \rangle  
		& = \langle \frac{p^{m+1}_{\Delta t, 2}-p^{m}_{\Delta t, 2}}{\Delta t}, v_N \rangle  
			- \langle p^{m}_{\Delta t, 2} \uuu_{1}^{m} + \up_{1}^{m} \uu^{m}_{\Delta t, 2}, \nabla v_N \rangle 
			\nonumber \\
		& \quad + \langle \up^{m}_{1}(1 + 2\Delta t \up^{m}_{1}) \nabla (\frac{p^{m+1}_{\Delta t, 2}}{\up^{m+1}_{1}} + \psi^{m+1}_{\Delta t, 2}) \nonumber \\
		& \quad	+ p^{m}_{\Delta t, 2}(1 + 4\Delta t \up^{m}_{1}) \nabla (\ln{\up^{m+1}_{1} + \upsi^{m+1}_{1}}), \nabla v_N \rangle \nonumber  \\
		& \quad + \bigO(\Delta t) + \bigO(N^{-k})  \label{lem:hoa_pf_solve_p2}, \\
	\langle -f^{m+1}_{\un_{1},2}, v_N \rangle 
		& = \langle \frac{n^{m+1}_{\Delta t, 2}-n^{m}_{\Delta t, 2}}{\Delta t}, v_N \rangle  
			- \langle n^{m}_{\Delta t, 2} \uuu_{1}^{m} + \un^{m}_{1} \uu^{m}_{\Delta t, 2}, \nabla v_N \rangle  
			\nonumber \\
		& \quad + \langle \un^{m}_{1}(1 + 2\Delta t \un^{m}_{1}) \nabla (\frac{n^{m+1}_{\Delta t, 2}}{\un^{m+1}_{1}}-\psi^{m+1}_{\Delta t, 2}) \nonumber \\
		& \quad	+ n^{m}_{\Delta t, 2}(1 + 4\Delta t \un_{1}^{m}) \nabla (\ln{\un_{1}^{m+1}-\upsi_{1}^{m+1}}), \nabla v_N \rangle \nonumber \\
		& \quad	+ \bigO(\Delta t) + \bigO(N^{-k}) \label{lem:hoa_pf_solve_n2}, \\
	\langle -f^{m+1}_{\uuu_{1}, 2}, v_N \rangle 
		&= \langle \frac{R_N\uu^{m+1}_{\Delta t, 2}-\uu^{m}_{\Delta t, 2}}{\Delta t}, v_N \rangle 
			+ \langle \nabla R_N\uu^{m+1}_{\Delta t, 2} , \nabla v_N \rangle 
				+ \langle \nabla \phi^{m}_{\Delta t, 2}, v_N \rangle \nonumber  \\
		&\quad + \langle (\uuu^{m}_{1} \cdot \nabla) R_N\uu^{m+1}_{\Delta t, 2} + (\uu^{m}_{\Delta t, 2} \cdot \nabla) R_N\uuu_{1}^{m+1}, v_N \rangle  \nonumber  \\
		&\quad + \langle \up_{1}^{m} \nabla (\frac{p^{m+1}_{\Delta t, 2}}{\up_{1}^{m+1}} + \psi^{m+1}_{\Delta t, 2})
							+ p^{m}_{\Delta t, 2} \nabla (\ln{\up_{1}^{m+1}} + \upsi_{1}^{m+1}), v_N \rangle \nonumber  \\
		&\quad + \langle \un_{1}^{m} \nabla (\frac{n^{m+1}_{\Delta t, 2}}{\un_{1}^{m+1}}-\psi^{m+1}_{\Delta t, 2})
							+ n^{m}_{\Delta t, 2} \nabla (\ln{\un_{1}^{m+1}}-\upsi_{1}^{m+1}), v_N \rangle \nonumber \\
		& \quad				+ \bigO(\Delta t) + \bigO(N^{-k}) \label{lem:hoa_pf_solve_u2}, \\
	\langle \nabla \psi^{m}_{\Delta t, 2}, \nabla v_N & \rangle 
		 = \langle p^{m}_{\Delta t, 2}-n^{m}_{\Delta t, 2}, v_N \rangle, \\
	\langle \uu^{m}_{\Delta t, 2}, \nabla v_N & \rangle = 0. 
\end{align}
Finally, a combination of \eqref{lem:hoa_o_p_1}-\eqref{lem:hoa_o_u_1} and \eqref{lem:hoa_pf_solve_p2}-\eqref{lem:hoa_pf_solve_u2} yields the third order temporal truncation error for $(\up, \un, \uuu, \uphi)$: 
\begin{align*}
	& \langle \frac{\up^{m+1}-\up^{m}}{\Delta t}, v_N \rangle-\langle \up^{m} \uuu^{m}, \nabla v_N \rangle + \langle \up^{m}(1 + 2\Delta t \up^{m}) \nabla \umu^{m+1}, \nabla v_N \rangle 
		= \tau^{m+1}_{p}(v_N)
		, \\
	& \langle \frac{\un^{m+1}-\un^{m}}{\Delta t}, v_N \rangle-\langle \un^{m} \uuu^{m}, \nabla v_N \rangle + \langle \un^{m}(1 + 2\Delta t n^{m}) \nabla \unu^{m+1}, \nabla v_N \rangle 
		= \tau^{m+1}_{n}(v_N)
		, \\
	& \langle \frac{R_N\uuu^{m+1}-\uuu^{m}}{\Delta t}, v_N \rangle +  \langle (\uuu^{m} \cdot \nabla)R_N\uuu^{m+1}, v_N \rangle + \langle \nabla R_N\uuu^{m+1}, \nabla v_N \rangle + \langle \nabla \uphi^{m}, v_N \rangle \nonumber \\
	& \quad + \langle \up^{m} \nabla \umu^{m+1} + \un^{m} \nabla \unu^{m+1}, v_N \rangle 
	  =	\tau^{m+1}_{\uu}(v_N),
\end{align*}
where 
\begin{equation*}
	\utau^{m+1}_{p}(v_N), \utau^{m+1}_{n}(v_N), \utau^{m+1}_{\uu}(v_N) \leq C(\Delta t^3 + N^{-k}) \|  v_N \| _{H^{1}}.
\end{equation*}
Since $(p_{\Delta t, 2}, n_{\Delta t, 2})$ are bounded, we may find $\Delta t, \frac{1}{N}$ so small that $\up, \un > \delta_0^{*} \triangleq \frac{\delta_0}{4} > 0$. Moreover, given the regularity of $(p_{
\Delta t, i}, n_{\Delta t, i}, \uu_{\Delta t, i})(i = 1, 2)$, we have 
\begin{equation*}
	(\up, \un, \uuu) \in L^{\infty}(0, T, W^{1, \infty}(\Omega)).
\end{equation*}
\end{proof}
\begin{remark}
	Since we set the initial data of our modified solution to be the same as the initial data of the exact solution, i.e.
	$(\up, \un, \uuu, \uphi)(\cdot, t = 0) = (\projN p, \projN n, \projN \uu, \projN \phi)(\cdot, t = 0)$,
	we will assume trivial initial data
	\begin{equation}
		(p_{\Delta t, i}, n_{\Delta t, i}, \uu_{\Delta t, i}, \phi_{\Delta t, i})(\cdot, t = 0) = {\bf 0},
	\end{equation} 
	for $i = 1,2$ in \eqref{lem:hoa_pf_cont_p1}-\eqref{lem:hoa_pf_cont_div1} and \eqref{lem:hoa_pf_cont_p2}-\eqref{lem:hoa_pf_cont_div2}.
\end{remark}

\bibliographystyle{plain}
\bibliography{ref.bib}

\end{document}